\documentclass[11pt]{article}
\usepackage{amsmath}
\usepackage{amsthm}
\usepackage{amssymb}
\usepackage{dsfont}
\usepackage[french]{babel}
\usepackage[utf8]{inputenc}
\usepackage{mathrsfs}
\newtheorem{lemma}{Lemme}
\newtheorem{prop}{Proposition}
\newtheorem{theorem}{Théorème}
\newtheorem*{theorems}{Théorème}

\newtheorem{defi}{Définition}
\makeatletter
\def\moverlay{\mathpalette\mov@rlay}
\def\mov@rlay#1#2{\leavevmode\vtop{%
   \baselineskip\z@skip \lineskiplimit-\maxdimen
   \ialign{\hfil$\m@th#1##$\hfil\cr#2\crcr}}}
\newcommand{\charfusion}[3][\mathord]{
    #1{\ifx#1\mathop\vphantom{#2}\fi
        \mathpalette\mov@rlay{#2\cr#3}
      }
    \ifx#1\mathop\expandafter\displaylimits\fi}
\makeatother

\title{Lois limites dans le problème des réseaux. \\
I. Le cas des boîtes.}
\author{Julien Trevisan}
\begin{document}
\maketitle
\bigskip
\begin{abstract}
We study the error of the number of unimodular lattice points that fall into a dilated 
and translated parallelogram. By using an article from Skriganov, we see that this error can be compared to an ergodic sum that involves the discrete geodesic flow over the space of unimodular lattices.
With the right normalization, we show, by using tools from a previous work of Fayad and Dolgopyat,
that a certain point process converges in law towards a Poisson process and 
deduce that the ergodic sum converges towards a Cauchy centered law when the unimodular lattice is distributed according to the normalized Haar measure. 
Strong from this experience, we apply the same kind of approach, with more difficulties, to the study of the asymptotic behaviour of the error and show that this error, normalized by $\log(t)$ with $t$ the factor of dilatation of the parallelogram, also converges in law towards a Cauchy centered law when the dilatation parameter tends to infinity and when the lattice and the vector of translation are random.
In a next article, we will show that, in the case of a ball in dimension $d$ superior or equal to $2$, the error, normalized by $t^{\frac{d-1}{2}}$ with $t$ the factor of dilatation of the ball, converges in law when $t \rightarrow \infty$ and the limit law admits a moment of order $1$.
\end{abstract}
\begin{abstract}
On étudie l'erreur commise lorsque l'on estime le nombre de points d'un réseau unimodulaire qui se trouvent dans un parallélogramme dilaté et translaté par son aire. À l'aide d'un des travaux de Skriganov, on voit que cette erreur peut être vue comme une somme ergodique portant sur le flot géodésique discret sur l'espace des réseaux unimodulaires. 
En normalisant correctement et en utilisant des outils d'un travail de Fayad et Dolgopyat, on montre qu'un certain processus converge en loi vers un processus de Poisson. On en déduit que la somme ergodique converge en loi vers une loi de Cauchy centrée lorsque le réseau est distribué selon la mesure de Haar normalisée. 
Fort de cette expérience, on applique le même type d'approche, avec plus de difficultés, pour étudier le comportement asymptotique de l'erreur et on montre que l'erreur, normalisée par $\log(t)$ avec $t$ le paramètre de dilatation du parallélogramme, converge en loi vers une loi de Cauchy centrée quand le paramètre de dilatation tend vers l'infini et lorsque le réseau est distribué selon la mesure de Haar normalisée et lorsque le vecteur de translation est aléatoire. 
Dans un prochain article, nous démontrerons que, dans le cas d'une boule de dimension $d$ supérieure ou égale à $2$, l'erreur, normalisée par $t^{\frac{d-1}{2}}$ avec $t$ le facteur de dilatation de la boule, converge en loi quand $t \rightarrow \infty$ et la loi limite admet un moment d'ordre $1$.
\end{abstract}
\section{Introduction}
On va s'intéresser dans ce papier à un cas particulier de l'estimation du nombre de points d'un réseau qui tombent dans un ensemble donné. Ce problème a une origine ancienne dans la mesure où le fameux problème du cercle de Gauss s'y rapporte. \\
 Soit $d \in \mathbb{N}-\{0,1\}$.
Soient $X \in \mathbb{R}^{d}$ et $L$ un réseau de $\mathbb{R}^{d}$ et $P$ un ensemble mesurable de $\mathbb{R}^{d}$ de mesure strictement positive et finie. On aimerait évaluer le cardinal suivant quand $t \rightarrow \infty$ : 
$$N(tP + X, L ) = | (t P + X) \cap L| \textit{.}$$
Sous certaine condition de régularité de l'ensemble $P$, on peut montrer que : 
$$ N(tP + X, L) = t^{d}\frac{\text{Vol}(P)}{\text{Covol}(L)} + o(t^{d}) \textit{.}$$
Il est donc naturel de s'intéresser à l'erreur commise : 
$$\mathcal{R}(tP + X,L) = N(tP + X, L) - t^{d}\frac{\text{Vol}(P)}{\text{Covol}(L)} \textit{.}$$
Dans le cas où $P$ est le disque unité $\mathbb{D}^{2}$, la conjecture de Hardy donnée dans $\cite{hardy1917average}$ stipule qu'on devrait avoir, pour tout $\epsilon >0$, $$\mathcal{R}(t D^{2}, \mathbb{Z}^{2}) = O(t^{\frac{1}{2}+\epsilon}) \textit{.}$$
Le meilleur résultat dans cette direction a été établi par Iwaniec et Mozzchi dans $\cite{iwaniec1988divisor}$. Ils ont prouvé que pour tout $\epsilon > 0$, $$\mathcal{R}(t D^{2}, \mathbb{Z}^{2}) = O(t^{\frac{7}{11}+\epsilon}) \textit{.}$$
Dans le cas où $P$ est un polytope convexe et où $L \in \mathscr{S}_{d}$ (l'espace des réseaux unimodulaires de $\mathbb{R}^{d}$, c'est-à-dire l'espace des réseaux de covolume $1$) est typique, Skriganov a réussi à relier $\mathcal{R}$ à une somme ergodique qui a donné, en particulier, l'estimation suivante : 
\begin{theorems}[\cite{Skriganov}]
\label{thm1}
Pour presque tout $L \in \mathscr{S}_{d}$, pour tout $\epsilon > 0$,
$$\max_{X \in \mathbb{R}^{d}} |\mathcal{R}(tP + X, L)| = O(\log(t)^{d-1+\epsilon}) $$
où $\mathscr{S}_{d}$ est munie de l'unique mesure de probabilité $\mu_{d}$ qui soit de Haar.
\end{theorems}
Notons que ce théorème n'est pas valable pour tout réseau $L \in \mathscr{S}_{d}$. Il suffit de considérer le réseau $L=\mathbb{Z}^{d}$ et pour $P$ un hypercube centré en $0$ dont les côtés sont parallèles aux axes. L'erreur maximale est dans ce cas de l'ordre de $t^{d-1}$. \\
En suivant l'approche de Kesten dans $\cite{Kesten60u}$ et dans $\cite{Kesten62u}$, nous nous intéressons au comportement asymptotique de $\mathcal{R}$ quand $L \in \mathscr{S}_{2}$ et $X \in \mathbb{R}^{2}/L$ sont aléatoires. \\
Plus formellement, on se donne $\tilde{\mu}_{2}$ une mesure de probabilité absolument continue par rapport à $\mu_{2}$, de densité régulière et bornée sur $\mathscr{S}_{2}$ et où $\mu_{2}$ désigne la mesure de Haar normalisée sur $\mathscr{S}_{2}$, en appelant $\lambda_{2}$ la mesure de Lebesgue normalisée sur $\mathbb{R}^{2}/L$, on se donne aussi $\tilde{\lambda}_{2}$ une mesure de probabilité absolument continue par rapport à $\lambda_{2}$, de densité régulière sur $ \mathbb{R}^{2}/L$. \\
Nous allons alors montrer le résultat suivant en supposant que $P$ est un parallélogramme :  
\begin{theorem}
\label{thm100}
Lorsque $L \in \mathscr{S}_{2}$ est distribué selon la mesure $\tilde{\mu}_{2}$ et $X$ est distribué selon $\tilde{\lambda}_{2}$ alors on a : 
$$\frac{\mathcal{R}(t,P,X,L)}{\log(t)} \overset{\mathcal{L}}{\underset{t \rightarrow \infty}{\rightarrow}} \mathcal{C}_{c} $$ 
où $\overset{\mathcal{L}}{\rightarrow}$ signifie que la convergence a lieu en loi et où $\mathcal{C}_{c}$ désigne une loi de Cauchy centrée. 
\end{theorem}
Suivant Skriganov, il nous a semblé judicieux pour étudier cette erreur $\frac{\mathcal{R}}{\log(t)}$, et cela a été fructueux, de s'intéresser à l'étude de la distribution des quantités : $ (\lVert\delta_{t} L \rVert^{-2})_{t \in [0,T-1]} $ quand $T \rightarrow \infty$, 
où on a posé $\delta_{t} = \begin{pmatrix} e^{t} & 0 \\ 0 & e^{-t} \end{pmatrix}$ et où $\lVert L \rVert$, pour tout réseau $L$, est défini par : $$\lVert L \rVert=\inf \{ \lVert l \rVert | l \in L-\{0\} \}$$ avec $\lVert X \rVert$ qui désigne la norme euclidienne habituelle de $X \in \mathbb{R}^{2}$.  \\
Plus précisément, lorsque l'on commence à étudier la convergence asymptotique de $\frac{\mathcal{R}}{\log(t)}$, on voit que l'on est ramené à l'étude d'une somme finie sur $l \in L$ dont les termes sont des fractions. Heuristiquement d'ailleurs le numérateur se comporte comme une fonction qui oscille de plus en plus vite dans un segment centré en $0$ tandis que le dénominateur comporte un terme de la forme $\textit{Num}(l) \log(t) = l_{1} l_{2} \log(t)$. Les termes ne vont donc pas être négligeables seulement lorsque $\textit{Num}(l)$ sera petit. Or, dans ce cas-là, $\textit{Num}(l)$ peut être vu comme un $\lVert \delta_{t} L \rVert^{2}$. \\
Ainsi, un bon indicateur pour savoir si $\frac{\mathcal{R}}{\log(t)}$ converge en loi est de savoir si : $$S(\omega,L,T) = \frac{1}{T} \sum_{ t \in [0,T-1]} \frac{\theta_{t}(\omega)}{\lVert \delta_{t} L \rVert^{2}} $$ converge en loi quand $T \rightarrow \infty$ ($T$ joue ici le rôle de $\log(t)$) avec $(\omega,L) \in \Omega \times \mathscr{S}_{2}$ où $(\Omega,\mathbb{P}_{1})$ est un espace probabilisé et les $\theta_{t}$ sont des variables aléatoires réelles à support compact, symétriques, identiquement distribuées, indépendantes entre elles.  \\
On va ainsi montrer d'abord le théorème suivant avant de montrer le théorème $\ref{thm100}$ : 
\begin{theorem}
\label{thm21}
Lorsque $L$ est distribué selon la loi de probabilités $\mu_{2}$ et $\omega$ selon $\mathbb{P}_{1}$,
$$S(\omega,L,T) $$ converge, quand $T \rightarrow \infty$, vers une loi de Cauchy centrée.
\end{theorem}
Dans la suite, nous présenterons seulement la démonstration du théorème $\ref{thm100}$ et du théorème $\ref{thm21}$ pour $d=2$, la preuve étant plus simple. Pour ce faire, dans la prochaine section, nous exposerons quelques préliminaires de probabilités. Ensuite, nous rappellerons quelques résultats sur l'espace des réseaux. On démontrera ensuite le théorème $\ref{thm21}$ qui est pertinent au vu de ce que nous avons énoncé précédemment et qui nous donnera d'ailleurs une idée de la démarche qu'il faudra suivre pour démontrer le théorème $\ref{thm100}$. Puis nous traiterons un cas simplifié issu de l'étude de $\mathcal{R}$ (mais qui sera toutefois plus compliqué que l'étude de $S(\omega,L,T)$ comme nous le verrons). Nous terminerons enfin l'étude de $\mathcal{R}$ et conclurons en ouvrant sur un autre problème connexe. 

\section{Préliminaires de probabilités}
Les nuages de points qui sont des processus de Poisson joueront un rôle essentiel dans nos études de convergence. Nous allons donc faire quelques rappels de base à ce sujet. Nous allons essentiellement suivre l'exposition faite dans $\cite{bassam}$. \\
Une variable aléatoire $N$ a une distribution de Poisson de paramètre $\lambda > 0$ si pour tout entier $k \geqslant 0$, $P(N=k) = e^{- \lambda } \frac{\lambda^{k}}{k!}$. Cette loi sera noté $\mathcal{P}(\lambda)$. On dispose alors des résultats suivants : \\
\begin{itemize}
\item Si $N_{1}, \dots, N_{m}$ sont des variables aléatoires indépendantes qui suivent des distributions de Poisson de paramètres respectifs $\lambda_{j}$ alors $N = \sum_{j=1}^{m} N_{j}$ admet une distribution de Poisson de paramètre $\sum_{j=1}^{m} \lambda_{j}$. \\
\item Inversement, si on se donne $N$ points distribués selon une distribution de Poisson de paramètre $\lambda$, que l'on colore chaque point indépendamment avec une couleur $j \in \{1, \cdots, m \}$ choisie avec un probabilité $p_{j}$ (et donc $\sum_{j=1}^{m} p_{j} =1$) et que l'on appelle $N_{j}$ le nombre de points de la couleur $j$ obtenus alors les $N_{j}$ sont indépendants et suivent des distributions de Poisson de paramètres respectifs $\lambda_{j} = p_{j} \lambda$.
\\
\end{itemize}
Soit $(\mathbf{X}, \mathbf{m})$ un espace mesuré. On appelle processus de Poisson sur cet espace un nuage de points aléatoire sur $\mathbf{X}$ tel que si $\mathbf{X}_{1}, \cdots, \mathbf{X}_{m}$ sont des ensembles (mesurables) disjoints et si $N_{j}$ est le nombre de poins qui tombent dans $\mathbf{X}_{j}$ alors les $N_{j}$ sont des variables aléatoires qui admettent pour lois respectives $\mathcal{P}(\mathbf{m}(\mathbf{X}_{j}))$. Cette définition est cohérente grâce au premier fait énoncé. Nous noterons $\{ x_{j} \} \sim \mathfrak{P}(\mathbf{X},\mathbf{m})$ pour indiquer que $\{x_{j} \}$ est un processus de Poisson de paramètre $(\mathbf{X},\mathbf{m})$. Si $\mathbf{X} \subset \mathbb{R}^{d}$ et si $\mathbf{m} \ll \lambda_{d}$ (la mesure de Lebesgue sur $\mathbb{R}^{d}$) alors $\mathbf{m}$ admet une densité par rapport à $\lambda_{d}$ et cette densité sera désignée comme étant l'intensité du processus de Poisson. \\
\\
On dispose ainsi du lemme suivant, qui nous sera utile par la suite : 
\begin{lemma}
\label{lemme20}
(a) Si $\{\theta_{j} \} \sim \mathfrak{P}(\mathfrak{X},\mathbf{m})$ et $\{\theta_{j}' \} \sim \mathfrak{P}(\mathfrak{X},\mathbf{m}')$ sont indépendants alors $$\{\theta_{j} \} \cup \{\theta_{j}' \} \sim \mathfrak{P}(\mathfrak{X},\mathbf{m}+\mathbf{m}') \textit{.}$$
(b) Si $\{\theta_{j} \} \sim \mathfrak{P}(\mathfrak{X},\mathbf{m})$ et $f : \mathfrak{X} \longrightarrow \mathfrak{Y}$ est une application mesurable alors 
$$\{ f(\theta_{j}) \} \sim \mathfrak{P}(\mathfrak{Y}, f^{-1} \mathbf{m}) \textit{.} $$
(c) Soient $\mathfrak{X} = \mathfrak{Y} \times Z$, $\mathbf{m} = \nu \times \lambda $ où $\lambda$ est une mesure de probabilité sur $Z$. Alors $ \{( \theta_{j}, \Gamma_{j}) \} \sim \mathfrak{P}(\mathfrak{X}, \mathbf{m})$ si et seulement si $\{ \theta_{j} \} \sim \mathfrak{P}(\mathfrak{Y}, \nu)$ et $ \Gamma_{j}$ sont des variables aléatoires indépendantes de $\{ \theta_{j} \}$ et entre elles et toutes sont distribuées selon $\lambda$. \\
\\
(d) Si en (c) $\mathfrak{Y} = Z = \mathbb{R} $ alors $\tilde{\theta} = \{ \Gamma_{j} \theta_{j} \} $ est un processus de Poisson. Si $ \{ \theta_{j} \}$ admet comme mesure $f(\theta)d \theta$ alors $\tilde{\theta}$ admet comme mesure $\tilde{f}(\theta) d\theta$ où $$\tilde{f}(\theta) = \mathbb{E}_{\Gamma} ( f( \frac{\theta}{\Gamma })\frac{1}{|\Gamma |}) \textit{.} $$
\end{lemma}

Un autre lemme qui nous sera utile est le suivant : 
\begin{lemma}
\label{lemme21}
Si $\{ \Xi_{i} \}$ suit un processus de Poisson sur $\mathbb{R}^{+}$ d'intensité constante $C > 0$ et si $\{ \Gamma_{i} \}$ sont des variables aléatoires réelles indépendantes (entre elles et avec les $\Xi_{i}$) identiquement distribuées, de lois de probabilité symétriques à support compact alors 
$$ \lim_{\epsilon  \rightarrow 0 } \frac{1}{\rho} \sum_{\Xi_{i} < \frac{1}{\epsilon}} \frac{\Gamma_{i}}{\Xi_{i}} $$ est distribuée selon la loi de Cauchy standard $\mathcal{C}(0,1)$ où $ \rho = \frac{C}{2} \mathbb{E}(|\Gamma|) \pi$.
\end{lemma}
Une variante de celui-ci, qui nous sera aussi utile est la suivante : 
\begin{lemma}
\label{lemme22}
Si $\{ \Xi_{i} \}$ suit un processus de Poisson sur $\mathbb{R}$ d'intensité constante $C > 0$ et si $\{ \Gamma_{i} \}$ sont des variables aléatoires réelles indépendantes (entre elles et avec les $\Xi_{i}$) identiquement distribuées, de lois de probabilité symétriques à support compact alors 
$$ \lim_{\epsilon  \rightarrow 0 } \frac{1}{\rho} \sum_{|\Xi_{i}| < \frac{1}{\epsilon}} \frac{\Gamma_{i}}{\Xi_{i}} $$ est distribuée selon la loi de Cauchy standard $\mathcal{C}(0,1)$ où $ \rho = C \mathbb{E}(|\Gamma|) \pi$.
\end{lemma}
Une preuve de ces deux derniers lemmes peut être trouvée dans $\cite{dolgoquenched}$. On déduit de la définition de la convergence en loi les deux lemmes suivants : 
\begin{lemma}
\label{lemme41}
Si $\{ \Xi_{i}^{T} \}$ tend, quand $T \rightarrow \infty$, vers un processus de Poisson sur $\mathbb{R}^{+}$ d'intensité constante $C > 0$ et si $\{ \Gamma_{i}^{T} \}$ tendent vers des variables aléatoires réelles indépendantes entre elles, identiquement distribuées, de lois de probabilité symétriques à support compact et indépendantes du processus limite de $\{ \Xi_{i}^{T} \}$ alors 
$$ \lim_{\epsilon  \rightarrow 0 } \lim_{T \rightarrow \infty} \frac{1}{\rho} \sum_{\Xi_{i,T} < \frac{1}{\epsilon}} \frac{\Gamma_{i,T}}{\Xi_{i,T}} $$ est distribuée selon la loi de Cauchy standard $\mathcal{C}(0,1)$ où $ \rho = \frac{C}{2} \mathbb{E}(|\Gamma|) \pi$.
\end{lemma}
\begin{lemma}
\label{lemme42}
Si $\{ \Xi_{i,T} \}$ tend, quand $T \rightarrow \infty$, vers un processus de Poisson sur $\mathbb{R}$ d'intensité constante $C > 0$ et si $\{ \Gamma_{i}^{T} \}$ tendent vers des variables aléatoires réelles indépendantes entre elles, identiquement distribuées, de lois de probabilité symétriques à support compact et indépendantes du processus limite de $\{ \Xi_{i}^{T} \}$ alors 
$$ \lim_{\epsilon  \rightarrow 0 } \lim_{T \rightarrow \infty} \frac{1}{\rho} \sum_{|\Xi_{i,T}| < \frac{1}{\epsilon}} \frac{\Gamma_{i,T}}{\Xi_{i,T}} $$ est distribuée selon la loi de Cauchy standard $\mathcal{C}(0,1)$ où $ \rho = C \mathbb{E}(|\Gamma|) \pi$.
\end{lemma}
Pour pouvoir appliquer ces deux lemmes, il nous faudra donc prouver qu'à l'infini, des nuages de points, adaptées, se comportent comme des processus de Poisson. Pour ce faire, un outil qui nous sera particulièrement utile est le théorème 6.1 démontré dans $\cite{bassam}$ qui est particulièrement adapté quand on a affaire à des flots qui mélangent suffisamment rapidement (vitesse exponentielle, par exemple). Ce sera notre cas comme on le verra un peu plus loin. On rappelle maintenant l'énoncé de ce théorème. \\
Soit $(\Omega,\mathbb{P})$ un espace de probabilité. On désigne par $\mathbb{E}$ l'espérance relativement à $\mathbb{P}$. \\
On se donne $\lambda_{1}(x) = a x$ une forme linéaire non nulle sur $\mathbb{R}$. \\
Soit $(\mathbf{X},\mathbf{m})$ un espace mesurable. Soit $\mathbf{Q}$ une collection dénombrable de partitions finies de $\mathbf{X}$ tel que $\mathbf{Q}$ converge vers la partition ponctuelle. \\
Pour tout $M$, on considère une suite $\{\xi_{t}^{M} \}_{t \in [0,M]}$ de variables aléatoires prenant des valeurs entières positives et une suite $\{ \nu_{t}^{M} \}_{t \in [0,M]}$ de variable aléatoires de $\Omega$ dans $\mathbf{X}$. On pourra imaginer, par exemple, que $\xi_{t}^{M}$ est une fonction qui indique si un flot mélangeant, à l'instant $t$, rentre dans une boule de rayon $\frac{1}{\sqrt{M}}$ (situé à l'infini de l'espace) et $\nu_{t}^{M}$ indique jusqu'où le flot est rentré à cet instant. \\
Pour toute partition $Q=(K_{1},\cdots, K_{P}) \in \mathbf{Q}$, on suppose que $\xi_{t}^{M}$ peut s'écrire $\xi_{t}^{M} = \sum_{p=1}^{P} \xi_{t,p}^{M}$ où $\xi_{t,p}^{M}$ prend des valeurs dans les entiers naturels et sur l'ensemble $\{\xi_{t}^{M}=1 \}$, $\xi_{t,p}^{M}$ vérifie : $$\xi_{t,p}^{M} = \mathbf{1}_{\nu_{t}^{M} \in K_{p}} \textit{.}$$ \\
On pose $\eta_{t}^{M} = \mathbf{1}_{\xi_{t}^{M}=1}$ et $\eta_{t,p}^{M} = \mathbf{1}_{\xi_{t,p}^{M}=\xi_{t}^{M}=1}$.\\
Toutes les variables dépendant de $M$, dans la suite on omettra pour simplifier la lettre $M$. \\
Dans toute cette partie, quand on utilise la notation $Y= \mathcal{O}(X)$ cela signifie que $|Y| \leqslant C |X|$ où $C$ dépend éventuellement de $\mathbf{Q}$ mais pas de $M,t$ ou encore de $\overline{\delta}$ (introduit un peu plus loin). \\
On suppose que pour tout $M$ fixé, une suite de partitions $F_{t}$, $t \in \Pi$ de $(\Omega,\mathbb{P})$ est donnée. Pour $\omega \in \Omega$, on appelle $F_{t}(\omega)$ l'élément de $F_{t}$ qui possède $\omega$. On appelle $\mathcal{F}_{t}$ la $\sigma$-algèbre engendrée par $F_{t}$. \\
On suppose que les hypothèses suivantes sont vérifiées : il existe $R>0$ (qui ne dépend pas de $M$) et un ensemble mesurable $E$ tel que 
$$ \mathbb{P}(E^{c})=\mathcal{O}(M^{-100 }) $$ 
et \\
\begin{itemize}
\item (h1) Pour tout $t \in [0,M]$, $$\mathbb{E}(\xi_{t})= \mathcal{O}(M^{-1}) \textit{;} $$
\item (h2) Pour tout $t \in [0,M]$, $$\mathbb{P}(\xi_{t}>1)= \mathcal{O}(M^{-2}) \textit{;} $$ 
\item (h3) Pour tout $t \neq t' \in [0,M]$, $$\mathbf{P}(\xi_{t} \geqslant 1, \xi_{t'} \geqslant 1)= \mathcal{O}(M^{-2}) \textit{;}$$
\item (h4) Pour tout $t,t' \in [0,M]$, où $\lambda_{1}(t) > \lambda_{1}(t') + R \log(M)$, pour tout $p \in [|1,P|]$ et pour tout $\omega \in E$ : \begin{center}
(h4a) $ \mathbb{E}(\xi_{t} \textit{ } | \textit{} \mathcal{F}_{t'})(\omega) = \frac{\mathbf{c} \mathbf{m} (\mathbf{X})}{M} + \mathcal{O}(M^{-2}) $ 
\end{center}
\begin{center}
(h4b) $ \mathbb{E}(\xi_{t,p} \textit{ } | \textit{} \mathcal{F}_{t'})(\omega) = \frac{\mathbf{c} \mathbf{m} (K_{p})}{M} + \mathcal{O}(M^{-2}) $ 
\end{center}
\begin{center}
(h4c) $ \mathbb{E}(\eta_{t,p} \textit{ } | \textit{} \mathcal{F}_{t'})(\omega) = \frac{\mathbf{c} \mathbf{m} (K_{p})}{M} + \mathcal{O}(M^{-2}) $ 
\end{center}
\item (h5) Pour $t,\overline{t} \in [0,M]$ où $\lambda_{1}(\overline{t}) > \lambda_{1}(t) + R \log(M)$, pour tout $p \in [1,P]$, pour tout $\omega \in E$, $$\xi_{t,p} \textit{ est constant sur } F_{\overline{t}}(\omega)\textit{;} $$
\item (h6) Les algèbres $\{\mathcal{F}_{t} \}$ ont une propriété de type filtration au sens où pour $t, \overline{t} \in [0,M]$, avec $\lambda_{1}(\overline{t}) > \lambda_{1}(t) + R \log(M)$, pour tout $\omega \in E$, 
$$F_{\overline{t}}(\omega) \subset F_{t}(\omega) \textit{.}$$
\end{itemize}
\begin{theorem}[\cite{bassam}]
\label{thm9}
Lorsque $(\lambda_{1},\{ \xi_{t}\},\{ \nu_{t} \})$ vérifient les hypothèses (h1) jusqu'à (h6), le nuage de points $$\{\nu_{t}^{M}, \frac{t}{M} \}_{\xi_{t}^{M}=1 \textit{, }t \in [0,M]} $$ converge lorsque $M \rightarrow \infty$ vers un processus de Poisson d'intensité $\mathbf{c}$ sur $(\mathbf{X} \times J, \mathbf{b} \times Leb) \textit{.}$
\end{theorem} 
En rajoutant des hypothèses, on peut avoir une conclusion plus forte qui nous sera utile pour obtenir, entre autres, des variables aléatoires réelles équidistribuées indépendantes entre elles, et vis-à-vis d'un processus de Poisson ad hoc, et ceci grâce au lemme $\ref{lemme20}$. Précisons ces hypothèses et cette conclusion. \\
Avec les notations du théorème $ \ref{thm9}$, supposons que l'on dispose de $\hat{\lambda}(t)$ une forme affine telle que $\hat{\lambda}(t) > \lambda_{1}(t)$ (où $\lambda_{1}$ peut maintenant être une forme affine non constante au lieu de linéaire et non nulle) sur $[0,M]$.  \\
Supposons que pour tout $M$ on dispose d'une séquence de $\zeta_{t}^{M}$ à valeurs dans $(\tilde{\mathbf{X}}, \tilde{\mathbf{m}})$, un espace probabilisé. Cet espace est d'ailleurs supposé doté d'une collection $\tilde{\mathbf{Q}}$ dénombrable de partitions finies de $\tilde{\mathbf{X}}$ qui converge vers la partition ponctuelle. Supposons enfin que pour chaque élément $\tilde{Q} = (\tilde{K}_{1}, \cdots, \tilde{K}_{P} ) \in \tilde{\mathbf{Q}}$, on dispose de $\tilde{E}$ tel que $\mathbb{P}(\tilde{E}^{c})= O(M^{-100})$ et tel que les hypothèses suivantes sont satisfaites : \\
\\
(h7) Il existe une suite $\nu_{M} \rightarrow 0$ quand $M \rightarrow \infty$ et $R > 0$ tels que pour tout $t, t' \in \Pi_{M}$ qui satisfont $\hat{\lambda}(t) \geqslant \lambda_{1}(t') + R \log(M) \geqslant \lambda_{1}(t) + 2 R \log(M) $, alors pour tout $\omega \in \tilde{E}$ tel que $\xi_{t}(\omega) = 1$,
$$|\mathbb{P}(\zeta_{t}^{M} \in \tilde{K}_{p} | \mathcal{F}_{t'})(\omega) - \tilde{\mathbf{m}}(\tilde{K}_{p})| \leqslant \nu_{M} \textit{.}$$ \\
(h8) Pour $t,\overline{t} \in \Pi_{M}$ où $\lambda_{1}(\overline{t}) > \hat{\lambda}(t) + R \log(M)$ pour tout $p \in [|1,P|]$, pour tout $\omega \in \tilde{E}$ tel que $\xi_{t}(\omega)=1$, 
$$\mathbf{1}_{\zeta_{t} \in \tilde{K}_{p}} \textit{ est constant sur }  F_{\overline{t}}(\omega) \textit{.}$$ 
On dispose alors du résultat suivant, plus fort que celui du théorème $\ref{thm9}$ : 
\begin{theorem}[\cite{bassam}]
\label{thm13}
Lorsque $(\lambda_{1},\hat{\lambda},\{ \xi_{t}\},\{ \nu_{t} \},\{ \zeta_{t} \})$ vérifient les hypothèses (h1) jusqu'à (h8), le nuage de points 
$$\{\nu_{t}^{M} \textit{, } \zeta_{t}^{M} \textit{, } \frac{t}{M} \}_{\xi_{t}^{M}=1 \textit{, } t \in \Pi_{M} } $$
converge quand $M \rightarrow \infty $ vers un processus de Poisson d'intensité $\mathbf{c}$ sur $$(\mathbf{X} \times \tilde{\mathbf{X}} \times J \textit{, } \mathbf{m} \times \tilde{\mathbf{m}} \times \text{Leb}) \text{.}$$
\end{theorem} 
Remarquons que dans ces deux théorèmes, on peut remplacer $[0,M]$ par $[-M,0]$, cela ne change rien à leurs validités. On utilisera cette remarque dans la section 5.
\section{Préliminaires sur l'espace des réseaux}
Dans cette section on rappelle quelques résultats sur l'espace des réseaux unimodulaires et l'action du flot géodésique sur celui-ci. Ils proviennent essentiellement de $\cite{Skriganov}$ et $\cite{bassam}$. \\
\\
Nous avons dit précédemment que l'espace des réseaux unimodulaires en dimension $d$ est noté $\mathscr{S}_{d}$. Autrement dit, on a l'égalité : $\mathscr{S}_{d} = SL_{d}(\mathbb{R})/ SL_{d}(\mathbb{Z})$. \\
\begin{defi}
Soit $L \in \mathscr{S}_{d}$. On définit $L^{\perp}\in \mathscr{S}_{d}$ le réseau dual de $L$ de la manière suivante : 
$$L^{\perp} = \{ l \in \mathbb{R}^{d} \text{ } | \text{ } \forall k \in L \textit{, } <k,l> \in \mathbb{Z} \}$$ où $<\cdot , \cdot >$ désigne le produit scalaire usuel.
\end{defi}
Cette définition intervient en particulier lorsque l'on commence à manipuler des séries de Fourier. \\
On a $\mathscr{S}_{d} = SL_{d}(\mathbb{R})/ SL_{d}(\mathbb{Z})$ et $\mathscr{S}_{d}$ est donc muni d'une structure de groupe de Lie. Cependant, pour $d \geqslant 2$, cet espace n'est pas compact. Mais il admet tout de même une mesure de Haar qui est de probabilités et que l'on note $\mu_{d}$. Elle est donc, en particulier, unimodulaire. On peut donc faire de la dynamique mesuré sur cet espace. \\
\begin{defi}
On dit qu'un sous-groupe $H$ de $SL_{d}(\mathbb{R})$ est ergodique (pour l'action naturelle sur $\mathscr{S}_{d}$) si tout ensemble mesurable $H$-invariant de $\mathscr{S}_{d}$ est de mesure nulle ou son complémentaire l'est.
\end{defi}
On dispose du théorème suivant (dit de Moore) : 
\begin{theorem}
\label{thm2}
Un sous-groupe $H$ est ergodique si, et seulement si, $H$ n'est contenu dans aucun sous-groupe compact de $SL_{d}(\mathbb{R})$.
\end{theorem}
En particulier, en posant $ \Gamma = \{ (t_{1}, \dots, t_{d}) \in \mathbb{R}^{d} \text{ } | \text{ } t_{1} + \cdots + t_{d} = 0 \}$,  le sous-groupe $ \Delta = \{ \text{Diag}(e^{t_{1}}, \cdots, e^{t_{d}}) \text{ } | \text{ } (t_{1},\cdots,t_{d}) \in \mathbb{Z}^{d} \cap \Gamma \} \subset SL_{d}(\mathbb{R})$ est ergodique. Ainsi, l'action du flot géodésique (discrétisé) $(\delta,L) \in \Delta \times \mathscr{S}_{d} \longmapsto \delta L $ est ergodique. Posons : 
$$\Delta_{r} = \{ \text{Diag}(e^{t_{1}}, \cdots, e^{t_{d}}) \text{ } | \text{ } (t_{1},\cdots,t_{d}) \in \mathbb{Z}^{d} \cap \Gamma \text{ tel que } \lVert (t_{1},\cdots,t_{d}) \rVert < r \} \textit{.}$$
 On appelle aussi $n_{r} = |\Delta_{r}|$ et on note que la suite d'ensembles $(\Delta_{r})$ est croissante et tend vers $\Delta$. De l'ergodicité de $\Delta$, on tire le théorème suivant (dit théorème ergodique individuel) : 
\begin{theorem}
\label{thm20}
Soit $\psi \in L^{1}(\mathscr{S}_{d},\mu_{d})$. Alors pour presque tout $L \in \mathscr{S}_{d}$, on a : 
$$\lim_{r \rightarrow \infty} \frac{1}{n_{r}} \sum_{\delta \in \Delta_{r}} \psi (\delta L) = \int_{\mathscr{S}_{d}} \psi(L) d \mu_{d}(L) \textit{.} $$
\end{theorem}
Dans $\cite{Skriganov}$, le lemme suivant, qui donne des exemples de fonctions intégrables sur $\mathscr{S}_{d}$ et qui nous sera utile, est démontré : 
\begin{lemma}
\label{lemme23}
Pour tout $\epsilon > 0$, $L \longmapsto \lVert L \rVert^{-d + \epsilon}$ est intégrable sur $\mathscr{S}_{d}$.

\end{lemma}
Nous avons encore besoin de définir quelques nombres caractéristiques d'un réseau $L \in \mathscr{S}_{d}$. 
\begin{defi}
Soit $k \in \{ 1 , \cdots,  d \}$. On pose : 
$$\lVert L \rVert_{k} =  \inf \{ r > 0 \text{ } | \text{ } B_{f}(0,r) \text{ contient une famille libre de vecteurs de } L \text{ de k vecteurs} \}$$
où $B_{f}(0,r)$ désigne la boule fermée de centre $0$ et de rayon $r$ relativement à $\lVert \rVert$. En particulier, on a : $ \lVert L \rVert_{1} = \lVert L \rVert$. 
\end{defi}
\begin{defi}
Un vecteur $l$ d'un réseau $L$ sera dit premier si pour tout $l' \in L$, l'égalité $l=k l'$ avec $k \in \mathbb{Z}$ implique que $k = \pm 1$.
\end{defi}
 Tout vecteur $l$ tel que $\lVert l \rVert = \lVert L \rVert_{k}$, pour $k \in \{ 1 , \cdots,  d \}$, est un vecteur premier et pour tout $\delta \in \Delta$, si $l$ est premier, $\delta l $ est premier. \\
\\
Soient $f,f_{1},f_{2}$ des fonctions $C^{\infty}$ par morceaux à support compact de $\mathbb{R}^{d}$ dans $\mathbb{R}$. On pose, pour tout $L \in \mathscr{S}_{d}$ : 
$$F(L) = \sum_{\substack{ l \in L \\ l \text{ premier} }} f(l) \text{ et } \overline{F}(L) = \sum_{\substack{ l_{1}\neq \pm l_{2} \in L \\ l_{1} \text{, } l_{2} \text{ premiers} }} f_{1}(l_{1})f_{2}(l_{2})\textit{.} $$ 
\begin{defi}
$F$ est appelée la transformée de Siegel de $f$. 
\end{defi}
On appelle $\zeta$ la fonction zêta de Riemann. \\
Alors, on dispose des identités (dites respectivement de Siegel et de Rogers) suivantes : 
\begin{lemma}[voir \cite{marklof1998n}, \cite{vinogradov2010limiting}]
\label{lemme24}
$$\int_{\mathscr{S}_{d}} F d\mu_{d} = \zeta(d)^{-1} \int_{\mathbb{R}^{d}} f d \lambda_{d} \textit{ ,} $$ 

$$\int_{\mathscr{S}_{d}} \overline{F} d \mu_{d} = \zeta(d)^{-2} \int_{\mathbb{R}^{d}} f_{1} d \lambda_{d} \int_{\mathbb{R}^{d}} f_{2} d \lambda_{d}  \textit{ et on en déduit} $$
$$\int_{\mathscr{S}_{d}} F^{2} d \mu_{d} = \zeta(d)^{-1} \int_{\mathbb{R}^{d}} f^{2} d \lambda_{d} + \zeta(d)^{-1} \int_{\mathbb{R}^{d}} f(x) f(-x ) dx + \zeta(d)^{-2} (\int_{\mathbb{R}^{d}} f d \lambda_{d})^{2} \textit{.} $$
\end{lemma}
La formule de Rogers peut bien sûr se généraliser à un produit fini quelconque de fonctions. \\
Supposons pour l'heure $d=2$. On déduit, des formules de Rogers et de Siegel, le lemme suivant qui est une des bases de notre réflexion : 
\begin{lemma} 
\label{lemme1}
Pour tout $a > 0$ assez grand, $$ \mu_{2}(\{ L \in \mathscr{S}_{2} | \lVert L \rVert^{-2} > a \}) = C a^{-1} \textit{ }$$ où on a posé $C = \frac{1}{2} c_{1} \pi = \frac{3}{\pi}$.
\end{lemma}
Pour démontrer ce lemme nous avons besoin de deux résultats préliminaires. 
\begin{lemma}
\label{lemme16}
Soient $l_{1},l_{2}$ deux vecteurs de $L$ premiers. Supposons qu'ils voient leurs premières coordonnées respectives être positives. Supposons que l'un des deux voit sa norme $2$ être égale à $\lVert L \rVert$. Alors, s'ils appartiennent à une même droite passant par $0$, ils sont égaux. 
\end{lemma}
\begin{proof}
Supposons que $l_{1} \neq l_{2}$ et que, par exemple, $\lVert l_{1} \rVert = \lVert L \rVert$. Alors, on a $l_{2} = k l_{1}$ avec $k \geqslant 0$, $k \notin \mathbb{N}$ car $l_{2}$ est premier. Soit dit en passant, comme $\lVert L \rVert = \lVert l_{1} \rVert$, nécessairement $k >1$. On pose $v = l_{2} - \lfloor k \rfloor l_{1}$. Alors, on a $\lVert v \rVert < \lVert l_{1} \rVert$, ce qui est exclu.
\end{proof}
De cette preuve, on en déduit que l'on n'a pas besoin de supposer $d=2$. \\
On en déduit alors le lemme suivant : 
\begin{lemma}
\label{lemme25}
Si $\lVert L \rVert < 1$ alors $\lVert L \rVert$ est seulement atteint en un vecteur premier $l$ et son opposée.
\end{lemma}
\begin{proof}
Supposons que $\lVert L \rVert$ soient atteints en deux vecteurs $l_{1}$ et $ l_{2}$. Ils sont nécessairement premiers et quitte à prendre leurs opposés, on peut supposer que leur première coordonnée est positive. Le lemme $\ref{lemme16}$ donne alors que soit $l_{1} = l_{2}$, soit $(l_{1}, l_{2})$ forme une famille libre de $\mathbb{R}^{2}$. On a donc dans ce dernier cas que $\lVert L \rVert_{2} = \lVert L \rVert_{1}$. \\
Or, on a : $\text{covol}(L) = \lVert L \rVert_{1} \lVert L \rVert_{2} \sin(\alpha)$ où $\alpha$ est l'angle non-orienté entre deux vecteurs de $L$ $h,k$ tels que $h,k$ voient leur première coordonnée être positive et $\lVert h \rVert = \lVert L \rVert_{1}$ et $\lVert k \rVert = \lVert L \rVert_{2}$. Ainsi, on obtient : $\text{covol} (L) < 1$, ce qui est exclu et donc $l_{1} = l_{2}$.
\end{proof}
Démontrons maintenant le lemme $\ref{lemme1}$ : 
\begin{proof}[Démonstration du lemme $\ref{lemme1}$]
On se donne $a > 0$. On pose, pour tout $x \in \mathbb{R}^{2}$, $f(x) = \mathbf{1}_{B_{f}(0,a^{-1/2})}(x)$. On appelle $F$ la transformée de Siegel associée à $f$. Pour tout $a > 0$ assez grand, on a pour tout $L \in \mathscr{S}_{2}$, $F(L) \in \{ 0, 2 \}$. En effet, pour $a > 0$ assez grand, $a^{- \frac{1}{2}} < 1$ et donc le lemme $\ref{lemme25}$ s'applique. Ainsi, on a : 
$$\frac{\mathbb{E_{\mu_{2}}}(F)}{2} = \mu_{2}(\{ L \in \mathscr{S}_{2} |\lVert L \rVert^{-2} > a \}) \textit{.} $$
Le lemme $\ref{lemme24}$ donne alors le résultat voulu.
\end{proof}
Notons au passage que le raisonnement fait au lemme $\ref{lemme25}$ donne le lemme suivant : 
\begin{lemma}
\label{lemme11}
Soit $L \in \mathscr{S}_{2}$. On suppose $\lVert L \rVert \leqslant 1$. On appelle $\lVert L \rVert_{2}$ le plus petit rayon $r$ tel que $B_{f}(0,r)$ contient une famille libre de deux vecteurs de $L$. Alors, deux cas se présentent : 
\begin{itemize}
\item Soit $\lVert L \rVert = \lVert L \rVert_{2}$. Dans ce cas, $\lVert L \rVert=1$ et $L$ a été obtenu à partir de $\mathbb{Z}^{2}$ à l'aide d'une rotation. Par ailleurs, $\lVert L \rVert$ est atteint exactement en $4$ vecteurs du réseau $L$.
\item Soit $\lVert L \rVert < \lVert L \rVert_{2}$. Dans ce cas, $\lVert L \rVert$ est atteint en exactement $2$ vecteurs du réseau $L$.
\end{itemize}
\end{lemma}
En dimension supérieure, c'est-à-dire pour $d \geqslant 3$, le lemme $\ref{lemme1}$ se généralise de la manière suivante : 
\begin{lemma}
\label{lemme26}
$$ \mu_{d}(\{ L \in \mathscr{S}_{d} | \lVert L \rVert^{-d} > a \}) = D a^{-1} + O(\frac{1}{a^{2}}) \textit{ }$$ où on a posé $D = \frac{\zeta(d)^{-1} \text{Vol}(B_{f}(0,1)) }{2}$.
\end{lemma}
Le problème, qu'il faut d'ailleurs gérer très souvent en dimension $d \geqslant 3$, est le suivant : même si $\lVert L \rVert$ est petit, cela ne signifie pas que $\lVert L \rVert_{2}$ est grand, contrairement à la dimension $2$. Voyons la preuve.
\begin{proof}
On pose, comme précédemment, pour tout $x \in \mathbb{R}^{d}$, $f(x) = \mathbf{1}_{B_{f}(0,a^{- \frac{1}{d}})}(x)$. On appelle $F$ la transformée de Siegel de $f$ et $\overline{F}$ la fonction définie par pour tout $L \in \mathscr{S}_{d}$, $$\overline{F}(L) = \sum_{\substack{e_{1} \neq \pm e_{2} \\ e_{1} \text{,} e_{2} \text{premiers} }} f(e_{1}) f(e_{2}) \textit{.} $$
D'une part, on a, d'après le lemme $\ref{lemme24}$ : \begin{equation}
\label{eq138}
\mathbb{E}(F) = \zeta(d)^{-1} \int_{\mathbb{R}^{d}} f d \lambda_{d} \textit{.} 
 \end{equation}
D'où, l'on tire, par homogénéité du volume d'une boule : $$\mathbb{E}(F) = \frac{\zeta(d)^{-1} \text{Vol}(B_{f}(0,1)) }{a} \textit{.}$$ 
D'autre part, on a : $\mathbb{E}(F) = 2 \mathbb{P}(F=2) + \mathbb{E}(F \mathbf{1}_{F \geqslant 3}) \textit{.}$
Or, on a : $F \mathbf{1}_{F \geqslant 3} \leqslant \overline{F} \mathbf{1}_{F \geqslant 3} \leqslant \overline{F}$. \\
D'où on a : $ \mathbb{E}(F \mathbf{1}_{F \geqslant 3}) \leqslant \mathbb{E}(\overline{F}) = O(\frac{1}{a^{2}})$ d'après le lemme $\ref{lemme24}$. \\
Ainsi, on obtient : 
\begin{equation}
\label{eq139}
\mathbb{E}(F) = 2 \mathbb{P}(F=2) + O(\frac{1}{a^{2}}) \textit{.} 
 \end{equation}
De ($\ref{eq138}$) et de ($\ref{eq139}$), on tire : 
 \begin{equation}
\label{eq140}
\mathbb{P}(F=2) = \frac{\zeta(d)^{-1} \text{Vol}(B_{f}(0,1)) }{2} \frac{1}{a} + O(\frac{1}{a^{2}}) \textit{.}
\end{equation}
Enfin, notons que $$\mu_{d}(\{ L \in \mathscr{S}_{d} | \lVert L \rVert^{-d} > a \})= \mathbb{P}(F > 0) = \mathbb{P}(F=2) +  \mathbb{P}(F \geqslant 3)$$ et que
$$\mathbb{P}(F \geqslant 3) \leqslant \mathbb{E}(F \mathbf{1}_{F \geqslant 3}) = O(\frac{1}{a^{2}}) \textit{.}$$\\
Ainsi, on obtient :  $$\mu_{d}(\{ L \in \mathscr{S}_{d} | \lVert L \rVert^{-d} > a \}) = \mathbb{P}(F=2) + O(\frac{1}{a^{2}}) \textit{.}$$
D'où le résultat voulu d'après ($\ref{eq140}$). 
\end{proof}
Rappelons que pour tout $k=(k_{1}, \cdots, k_{d})$ tels que $k_{1} + \cdots + k_{d} = 0$, on a posé dans l'introduction $\delta_{k} = \text{Diag}(e^{k_{1}}, \cdots, e^{k_{d}})$. Terminons maintenant cette sous-section par deux petits lemmes qui nous permettront de calculer des plus petites longueurs de réseaux. 
\begin{lemma}
\label{lemme14}
Pour tout $L \in \mathscr{S}_{d}$, pour tout $ k \in \Gamma$,
\begin{equation}
e^{-\lVert k \rVert_{\infty}} \lVert L \rVert \leqslant \lVert \delta_{k}L \rVert \leqslant e^{\lVert k \rVert_{\infty}} \lVert L \rVert \textit{.}
\end{equation}

\end{lemma}
\begin{proof}
C'est immédiat via les définitions de $\delta_{k}$ et de $\lVert \cdot \rVert$. 
\end{proof}
Le lemme $\ref{lemme14}$ nous sera essentiellement utile en dimension $2$. 
On a le lemme suivant, en dimension quelconque :
\begin{lemma}
\label{lemme27}
Soit $L \in \mathscr{S}_{d}$. On suppose que $\lVert L \rVert_{2} > \lVert L \rVert_{1}$. On appelle $l$ un vecteur de $L$ tel que $\lVert l \rVert=\lVert L \rVert>0$. Alors pour tout $k \in \mathbb{Z}^{d} \cap \Gamma $ tel que $\lVert L \rVert_{2} > e^{2 ||k||_{\infty}} \lVert L \rVert_{1}$ on a $$ \lVert \delta_{k} L \rVert_{1} = \lVert \delta_{k} l \rVert \textit{.} $$
\end{lemma}
\begin{proof}
Soient $k,l$ comme dans l'énoncé du lemme. Comme $\lVert L \rVert_{2} > \lVert L \rVert_{1}$, pour tout $h \in L-\{l, -l \}$, $h$ premier, $\lVert h \rVert \geqslant \lVert L \rVert_{2}$ d'après le lemme $\ref{lemme16}$. Ainsi, on a : $$ \lVert \delta_{k} h \rVert \geqslant e^{- \lVert k \rVert_{\infty}} \lVert h \rVert \geqslant e^{- \lVert k \rVert_{\infty}} \lVert L \rVert_{2} \textit{.}$$
Or, par hypothèse, on a $$e^{- \lVert k \rVert_{\infty}} \lVert L \rVert_{2} > e^{\lVert k \rVert_{\infty}} \lVert L \rVert_{1} \textit{.} $$
Ainsi, on obtient : $$\lVert \delta_{k} h \rVert > e^{\lVert k \rVert_{\infty}} \lVert L \rVert_{1} \geqslant \lVert \delta_{k} l \rVert \textit{.}$$
D'où le lemme. 
\end{proof}

En dimension $2$, on a une variante de ce lemme qui repose sur le lemme $\ref{lemme14}$ : 
\begin{lemma}
\label{lemme15}
Soit $L \in \mathscr{S}_{2}$. On suppose que $\lVert L \rVert < 1$. On appelle $l$ un vecteur de $L$ tel que $\lVert l \rVert=\lVert L \rVert$. Alors pour tout $k \in \mathbb{Z}$ tel que $ e^{|k|} \lVert L \rVert < 1$, $$\lVert \delta_{k}L \rVert = \lVert \delta_{k} l \rVert \textit{.} $$
\end{lemma}
\begin{defi}
\label{def3}
Soit $x=(x_{1},\cdots,x_{d}) \in \mathbb{R}^{d}$. La quantité $\text{Num}(x)$ est définie par : $$\text{Num}(x) =\prod_{i=1}^{d} x_{i} \textit{.} $$
Pour $L \in \mathscr{S}_{d}$, pour $r \geqslant \lVert L \rVert$, on définit $$\nu (L,r) = \inf \{ \text{Num}(l) \text{ } | \text{ } 0 < \lVert l \rVert < r \} \textit{.}$$
On dit qu'un réseau $L$ est faiblement admissible si pour tout $r \geqslant \lVert L \rVert$, $\nu (L,r) > 0$. 
\end{defi}
On peut noter au passage que $\nu(L,\cdot)$ est une fonction décroissante. Dans $\cite{Skriganov}$ (voir Lemme 4.5), la proposition suivante est démontrée  : 
\begin{prop}[\cite{Skriganov}]
\label{prop12}
Pour tout $\beta > 0$, pour presque tout $L \in \mathscr{S}_{d}$, il existe $C > 0$ tel que pour tout $l \in L-\{0 \}$, 
$$|\text{Num}(l)| \geqslant C |\log( \lVert l \rVert)|^{1-d-\beta} \textit{.}$$
En particulier, presque tout réseau $L \in \mathscr{S}_{d}$ est faiblement admissible.
\end{prop}
Cette proposition signifie que de manière générique, si un vecteur $l \in L$ devient grand, il ne pourra pas se rapprocher trop vite des axes de coordonnées. De manière équivalente, cela veut dire que la vitesse de divergence vers l'infini de $(\delta_{t} L)_{t \in \Gamma_{r}}$ n'est pas trop importante génériquement. Skriganov démontre cette proposition en utilisant le lemme de Borel-Cantelli. \\
\\
Intéressons-nous, pour finir cette section, au flot géodésique sur $\mathscr{S}_{d}$ et à un aspect plus "différentiel" de celui-ci (alors qu'avant on s'est intéressé à des choses plus "métriques"). Les résultats qui seront énoncés le seront en dimension $2$ et se généralise en dimension supérieure. \\
Gardons tout d'abord en tête que l'action de $\delta_{t}$ (pour $t \in \mathbb{R}$) sur $\mathscr{S}_{2}$ est partiellement hyperbolique dans le sens où : 
$$ T \mathscr{S}_{2} = E_{0} + E_{1}^{+} + E_{1}^{-} $$
où $E_{0}$ est tangent aux orbites du flot géodésique et $E_{1}^{\pm}$ sont des distributions invariantes de dimension $1$.
\begin{defi}
\label{def4}
Les exposants de Lyapunov correspondants sont $\pm \lambda_{1}$ où $$\lambda_{1} = 2 t \textit{.}$$
\end{defi}
Par ailleurs, $E_{1}^{\pm}$ sont tangents aux feuilletages $W_{1}^{\pm}$ pour les feuilletages d'orbites des groupes $h_{1}^{\pm}$ où $$ h_{1}^{+} =  \begin{pmatrix} 1 & u \\ 0 & 1 \end{pmatrix} \text{ et } $$
$h_{1}^{-}(u)$ est la transposée de $h_{1}^{+}(u)$. \\
Dans la suite on désignera $E_{1}^{+}$ par $E_{1}$ et $W_{1}^{+}$ par $W_{1}$ et $h_{1}^{+}$ par $h_{1}$ et les résultats énoncés seront encore valables pour $E_{1}^{-}$, $W_{1}^{-}$ et $h_{1}^{-}$. Par ailleurs $\lVert \cdot \rVert_{H^{s}}$ désignera la norme de Sobolev d'indice $s$. 
\begin{defi}
\label{def1}
Soient $s,r \in \mathbb{N}$. On dit qu'une fonction $F : \mathscr{S}_{2} \rightarrow \mathbb{R}$ est dans $H^{s,r}$ avec $\lVert A \rVert_{H^{s,r}} = K $ si pour tout $1 \geqslant \epsilon > 0$, il existe des $H^{s}$-fonctions $A^{-} \leqslant A \leqslant A^{+}$ telles que $$\lVert A^{+}-A^{-} \rVert_{L^{1}(\mu_{\mathscr{S}_{2}})} \leqslant \epsilon \text{ et } \lVert A^{\pm} \rVert_{H^{s}} \leqslant K \epsilon^{-r} \text{. }$$
\end{defi}

\begin{defi}
\label{def2}
On dit que $ \gamma$ est une $W_{1}$-courbe de longueur $L$ s'il existe $L \in \mathscr{S}_{2}$ tel que $$\gamma = \{ h_{1}(\tau) y \text{ } | \text{} \tau \in [0,L] \} \textit{.} $$
Dans ce cas, pour une fonction $F : \mathscr{S}_{2} \rightarrow \mathbb{R}$, on utilisera la notation $$\int_{\gamma}A =  \frac{1}{L} \int_{0}^{L} A(h_{1}(s) y) ds \text{.} $$
\end{defi}
\begin{defi}
Soit $\kappa_{0}>0$. Soit $L > 0$ et $\mathcal{P}$ une partition de $\mathscr{S}_{2}$ en $W_{1}$-courbes de longueur $L$. Désignons par $\gamma(x)$ l'élément de $\mathcal{P}$ possédant $x$. Étant donné une suite, finie ou infinie, d'entiers $(k_{n})$ et une fonction $A \in H^{s,r}$, on dit que $\mathcal{P}$ est $\kappa_{0}$-représentative relativement à $((k_{n}),A)$ si pour tout $n$, 
\begin{equation}
\label{eq200}
\mu_{\mathscr{S}_{2}}( x \in \mathscr{S}_{2} \text{ } | \text{ } |\int_{g^{k_{n}}\gamma(x)}A - \mu_{\mathscr{S}_{2}}(A) | \geqslant \mathcal{K}_{A} L_{n}^{- \kappa_{0}}) \leqslant L_{n}^{- \kappa_{0}} \end{equation}
où $\mathcal{K}_{A}= \lVert A \rVert_{s,r} +1$, $L_{n}= L e^{\lambda_{1}(k_{n})}$ est la longueur de $g^{k_{n}} \gamma(x)$ et $\mu_{\mathscr{S}_{2}}(A)= \int_{\mathscr{S}_{2}} A(x) d\mu_{\mathscr{S}_{2}}(x)$. \\
Les points $x$ tels que pour tout $n$, $$| \int_{g^{k_{n}}\gamma(x)}A - \mu_{\mathscr{S}_{2}}(A)| \leqslant \mathcal{K}_{A}L_{n}^{-\kappa_{0}}$$ sont dits représentatifs relativement à $(\mathcal{P},(k_{n}),A)$. \\
Notons au passage que si $\mathcal{P}$ est $\kappa_{0}$-représentative relativement à $((k_{n}),A)$ et si $$\sum_{n}(L_{n})^{- \kappa_{0}} \leqslant \epsilon $$ alors l'ensemble des points représentatifs est de mesure au moins $1- \epsilon$.
\end{defi}
Dans $\cite{bassam}$ (voir Proposition 7.3), la proposition suivante est prouvée : 
\begin{prop}[\cite{bassam}]
\label{prop13}
(a) Il existe $s, \kappa_{0}, \epsilon_{0} > 0$ tel que pour tout $0 \leqslant r \leqslant s$, $0 < \epsilon \leqslant \epsilon_{0}$, pour toute famille finie de fonctions $\mathfrak{F}$ de $H^{s,r}$, pour tout $L > 0$ pour toute suite $(k_{n})$ tel que $$\sum_{n} (L e^{\lambda_{1}(k_{n})})^{-\kappa_{0}} \leqslant \epsilon \textit{,}$$ 
il existe une partition $\mathcal{P}$ de $\mathscr{S}_{2}$ en $W_{1}$-courbes de longueur $L$ qui est $\kappa_{0}$-représentative relativement à $((k_{n}),F)$ et ce quelque soit $F \in \mathfrak{F}$. \\
(b) Si $L \in \mathscr{S}_{2}$ est distribué une mesure $\tilde{\mu}_{\mathscr{S}_{2}}$ de densité bornée relativement à la mesure $\mu_{\mathscr{S}_{2}}$ alors le résultat de (a) est encore valable pour peu que dans la définition de partition représentative, l'équation $\ref{eq200}$ soit remplacée par $$\tilde{\mu}_{\mathscr{S}_{2}}( x \in \mathscr{S}_{2} \text{ } | \text{ } |\int_{g^{k_{n}}\gamma(x)}A - \mu_{\mathscr{S}_{2}}(A) | \geqslant \mathcal{K}_{A} L_{n}^{- \kappa_{0}}) \leqslant \tilde{C}L_{n}^{- \kappa_{0}} $$ 
où $\tilde{C}$ représente la borne supérieure de la densité $f$ de $\tilde{\mu}_{\mathscr{S}_{2}}$ relativement à $\mu_{\mathscr{S}_{2}}$.
\end{prop}
Elle repose sur le fait que le flot géodésique mélange à vitesse exponentielle.
Toujours dans $\cite{bassam}$(Lemme A.2), on trouve le lemme suivant, qui va nous permettre d'estimer des $\lVert \cdot \rVert_{H^{s,r}}$. 
\begin{lemma}[\cite{bassam}]
\label{lemme19}
Pour chaque entier naturel $s$, pour chaque $R>0$, il existe une constante $C(R,s)$ telle que : pour $f \in C^{s,r}(\mathbb{R}^{d})$ à support compact inclus dans la boule euclidienne de centre $0$ et de rayon $R$ alors, si on appelle $F$ la transformée de Siegel de $f$, $F \in H^{s,r}$ et 
$$\lVert F \rVert_{H^{s,r}} \leqslant C(R,s) \lVert f \rVert_{C^{s,r}} \text{.}$$ 
\end{lemma}
Il faut noter que $C^{s,r}$ est définie d'une manière analogue à $H^{s,r}$ et que l'espace $C^{s}$ est l'espace des fonctions régulières, à support compact, munie de la norme $\lVert \cdot \rVert_{\infty, s}$ ($s$ indiquant l'ordre maximal de dérivation pris en compte). \\
\section{Étude de $S(\omega,L,T)$} 
\subsection{Introduction au problème}
Dans cette section, nous allons préciser le comportement asymptotique de $S(\omega, L, r)$.\\
Soit $(\theta_{t})_{t \in \mathbb{N}}$ une suite de variables aléatoires réelles dont l'aléa $\omega$ appartient à $ (\Omega, \mathbb{P}_{1}) $ un espace probabilisé, indépendantes, identiquement distribuées et symétriques, c'est-à-dire tels que $\mathbb{P}_{\theta_{t}}= \mathbb{P}_{- \theta_{t}}$.\\
\\
Notons que la construction d'un tel objet, où $\Omega = \mathscr{S}_{2}$ et avec indépendance par rapport aux $\lVert \delta L \rVert$, peut se faire de la manière suivante : pour tout $t \in \mathbb{N}$, pour tout $n \geqslant 1$, on pose $\theta_{t,n} = f(\delta_{t}^{n^{2}} \cdot) $ où $f$ est une fonction régulière mesurable de $\mathscr{S}_{2}$ à valeurs dans un segment $[-a,a]$ où $a > 0$ tel que $\mu_{2}(f(L) \in [0,b]) = \mu_{d}(f(L) \in [-b,0])$ et ce pour tout $0 \leqslant b \leqslant a$. La suite de variables $(\theta_{t})$ peut-être pensé comme la limite, quand $n \rightarrow \infty$, de $(\theta_{t,n})_{t \in \Gamma}$. \\
Cependant, en général, il n'est pas certain que cette limite existe. En revanche, comme chacune de ces suites (en $n$) de variables aléatoires est tendu, on va pouvoir trouver une extractrice $\phi$ tel que $(\theta_{t,\phi(n)})_{t \in \mathbb{N}}$ converge vers une suite de variables aléatoires réelles identiquement distribuées (comme $\mu_{2}$ est une mesure de Haar), symétriques, à support compact et indépendantes entre elles et des $\lVert \delta L \rVert$ (comme le flot géodésique mélange à vitesse exponentielle).  \\
Notons que cette construction peut se généraliser à la dimension $d \geqslant 2$. \\
\\
On s'intéresse maintenant au comportement quand $T \rightarrow \infty$ de $S(\omega,L,T)$. \\
Le principal résultat que l'on va démontrer à ce sujet est le théorème $\ref{thm21}$. \\
Rappelons-le : 
\begin{theorem}
\label{thm21}
Lorsque $L$ est distribué selon la loi de probabilités $\mu_{2}$ et $\omega$ selon $\mathbb{P}_{1}$,
$$S(\omega,L,T) = \sum_{t=0}^{T-1} \frac{\theta_{t}(\omega)}{\lVert \delta_{t} L \rVert^{2} T} $$ converge, quand $T \rightarrow \infty$, vers une loi de Cauchy centrée.
\end{theorem}
Dans les sous-sections qui suivent, on s'attache à démontrer ce théorème.
\subsection{Élimination des termes dont le dénominateur est trop grand}

Soit $\alpha > 0$. Pour tout $\epsilon > 0$, posons : \begin{equation}
\label{eq103}
A_{1}(\epsilon,T,L)= \{ i \in [0,T-1] \textit{ } | \textit{ }  \lVert \delta_{i}L \rVert^{2}T \leqslant \frac{1}{\epsilon} \}  \textit{.}
\end{equation}
Posons aussi 
\begin{equation}
\label{eq201}
S_{1}(\omega,L,\epsilon,T) = \sum_{t \in A_{1}(\epsilon,T,L) } \frac{\theta_{t}(\omega)}{T \lVert \delta_{t} L \rVert^{2}} \textit{.}
\end{equation}

\begin{prop}
\label{prop9}
Pour tout $\epsilon > 0$ assez petit, pour tout $T$ assez grand, on a : 
$$ \mathbb{P}(|S(\omega,L,T) - S_{1}(\omega,L,\epsilon,T)| \geqslant \alpha) \leqslant \alpha \textit{.}$$
\end{prop}
Cette proposition nous dit fondamentalement que les termes dont le dénominateur est trop grand peuvent être négligés.
\begin{proof}
On a, pour $L \in \mathscr{S}_{2}$, pour $T \in \mathbb{N}-\{0 \}$ : 
\begin{equation}
\label{eq92}
\sum_{i=0}^{T-1} \frac{\theta_{i}(\omega)}{T \lVert \delta_{i}L \rVert^{2}} - \sum_{\substack{ 0 \leqslant i \leqslant T-1 \\ T \lVert \delta_{i}L \rVert^{2} \leqslant \frac{1}{\epsilon}}} \frac{\theta_{i}(\omega)}{T \lVert \delta_{i}L \rVert^{2}} = \sum_{\substack{ 0 \leqslant i \leqslant T-1 \\ T \lVert \delta_{i}L \rVert^{2} > \frac{1}{\epsilon}}} \frac{\theta_{i}(\omega)}{T \lVert \delta_{i}L \rVert^{2}}  \textit{.} 
\end{equation}
Or, comme les $\theta_{i}$ sont des variables aléatoires identiquement distribuées, d'espérances nulles et que les $L \longmapsto \delta_{j} L$ sont identiquement distribuées et que l'on travaille sur l'espace probabilisé produit $\Omega \times \mathscr{S}_{d}$, on a :  
\begin{equation}
\label{eq100}
E(\sum_{ \substack{ 0 \leqslant t \leqslant T-1 \\ T || \delta_{i} t ||^{2} > \frac{1}{\epsilon}}} \frac{\theta_{i}(\omega)}{T \lVert \delta_{i}L \rVert^{2}} ) = 0 
\end{equation}
et
\begin{equation}
\begin{split}
\label{eq101}
V(\frac{1}{T}\sum_{ \substack{ 0 \leqslant t \leqslant T-1 \\ T || \delta_{i} L ||^{2} > \frac{1}{\epsilon}}} \frac{\theta_{i}(\omega)}{\lVert \delta_{i}L \rVert^{2}} ) = T \frac{1}{T^{2}} V(\frac{\theta_{0}(\omega)}{\lVert \delta_{0}L \rVert^{2}} \mathbf{1}_{T || \delta_{0} L ||^{2} > \frac{1}{\epsilon}})\\
= \frac{1}{T} E((\frac{\theta_{0}(\omega)}{\lVert \delta_{0}L \rVert^{2}} \mathbf{1}_{T|| \delta_{0} L ||^{2} > \frac{1}{\epsilon}})^{2}) \\
\leqslant \frac{M^{2}}{T} E(\frac{1}{\lVert L \rVert^{4}}\mathbf{1}_{T\lVert L \rVert^{2} > \frac{1}{\epsilon}})) 
\end{split}
\end{equation}
où $M = \lVert \theta_{0} \rVert_{\infty} \geqslant 0$.
Or le lemme $\ref{lemme1}$ donne que pour tout $\epsilon > 0$, pour tout $T$ assez grand, $E(\frac{1}{\lVert L \rVert^{4}}\mathbf{1}_{T\lVert L \rVert^{2} > \frac{1}{\epsilon}})) \sim \epsilon T$. D'où le résultat voulu via l'inégalité de Bienaymé-Tchebychev. 
\end{proof}
La proposition $\ref{prop9}$ nous ramène ainsi à l'étude de la quantité $S_{1}(\omega,L,T,\epsilon)$
où $\epsilon$ peut être supposé, et sera supposé par la suite, compris strictement entre $0$ et $1$. 
\subsection{Existence des minima locaux et centrage sur ceux-ci}
Posons 
\begin{equation}
\label{eq104}
A_{2}(\epsilon,T,L)= \{ i \in [0,T-1] \cap A_{1}(\epsilon,T,L) \textit{ } | \textit{ }  \lVert \delta_{i-1}L \rVert > \lVert \delta_{i}L \rVert< \lVert \delta_{i+1}L \rVert  \} \textit{.}
\end{equation} 
La proposition suivante nous assure qu'on peut "centrer" les termes de $S_{1}$ sur ces minima locaux : 
\begin{prop}
\label{prop14}
Il existe $k_{d}(i) \leqslant 0 $ et $k_{m}(i) \geqslant 0$ pour $i \in A_{2}(\epsilon,T,L)$ tels que pour tout $\alpha > 0$, pour tout $\epsilon > 0$ assez petit, pour tout $T$ assez grand, on a : 
$$ \mathbb{P}(|S_{1}(\omega,L,\epsilon,T) - S_{2}(\omega,L,\epsilon,T)| \geqslant \alpha) \leqslant \alpha \textit{}$$ 
où, pour tout $L \in \mathscr{S}_{2}$ et $\omega \in \Omega$,
 \begin{equation}
\label{eq106}
S_{2}(\omega, L,\epsilon, T) = \sum_{i \in A_{2}(\epsilon, T, L)} \sum_{ k_{d}(i) \leqslant k \leqslant k_{m}(i) } \frac{\theta_{i+k}(\omega)}{T \lVert \delta_{i+k}L \rVert^{2}} 
\end{equation}
et où : \begin{itemize}
\item  $|k_{d}(i)|,|k_{m}(i)| \geqslant  \frac{\log(T)}{2} + C_{\epsilon}$ où $C_{\epsilon} = \frac{\log(\epsilon)}{2} - 3 $ ; 
\item  tous les ensembles $\{i+k\}_{k_{d}(i) \leqslant k \leqslant k_{m}(i) }$ sont deux à deux disjoints ; 
\item  pour tout $k \in [k_{d}(i),k_{m}(i)]$, $1 > \lVert \delta_{i+k} L \rVert > \lVert \delta_{i} L \rVert$.
\end{itemize}
\end{prop}
Il faut noter que $k_{d}(i)$ et $k_{m}(i)$ dépendent éventuellement de $\lVert \delta_{i} L \rVert$. On verra un peu plus loin qu'on pourra éliminer cette dépendance éventuelle. \\
Avant de continuer, nous avons besoin de deux lemmes préliminaires. \\
Soient $l_{1},l_{2}$ deux réels non nuls et $l=(l_{1},l_{2})$. Pour tout $a > 0$, on pose :  $$f_{(l_{1},l_{2})}(a) = a l_{1}^{2} + \frac{1}{a} l_{2}^{2} \textit{.}$$ 
\begin{lemma}
\label{lemme17}
$f_{(l_{1},l_{2})}$ est strictement décroissante jusqu'en $a_{0} = | \frac{l_{1}}{l_{2}}|$, où elle atteint son minimum, puis est strictement croissante.
\end{lemma}
\begin{proof} 
Une étude des variations de $f_{(l_{1},l_{2})}$ via un calcul de dérivé donne le résultat.
\end{proof} 
\begin{lemma}
\label{lemme28}
Soient $\beta >0$, $C >0$ et $L \in \mathscr{S}_{2}$ tels que pour tout $l \in L-\{ 0 \}$, $$|\text{Num}(l)| \geqslant C |\log(\lVert l \rVert)|^{-1-\beta} \textit{.}$$
Alors il existe $D > 0$ tel que pour tout $T$ assez grand, pour tout $ i \in [|-\log(T), T + \log(T) |]$
\begin{equation}
\label{eq117}
\frac{\lVert \delta_{i} L \rVert^{2}}{2}  \geqslant D  T^{-1 - \beta} 
\end{equation}
pour peu que l'on suppose $\lVert L \rVert \geqslant \sqrt{\epsilon}$.
\end{lemma}
La première hypothèse de ce lemme est réalisée quitte à écarter un petit nombre de réseaux et ce d'après la proposition $\ref{prop12}$. Il en va de même pour la seconde hypothèse d'après le lemme $\ref{lemme1}$. On peut ainsi supposer ces deux hypothèses vérifiées et on le fera par la suite.  
\begin{proof}
Soit $ i \in [-\log(T), T + \log(T) ]$. Appelons $l=(l_{1},l_{2}) \in L-\{ 0 \}$ le vecteur tel que $\lVert \delta_{i}l \rVert = \lVert \delta_{i}L \rVert$.  \\
Par ailleurs, en supposant que $\lVert L \rVert > 1$, l'inégalité arithmético-géométrique donne, d'une part, que : 
$$\frac{\lVert \delta_{i} L \rVert^{2}}{2} \geqslant |l_{1} l_{2}| \geqslant C \log (\lVert l \rVert)^{-1- \beta} \textit{.}$$
D'autre part, on a, d'après le théorème de Minkowski: 
$$ e^{2 i} l_{1}^{2} + e^{-2 i} l_{2}^{2} = \lVert \delta_{i} L \rVert^{2} \leqslant \frac{4}{\pi}$$ 
et donc, comme $ i \in [-\log(T), T + \log(T) ]$,
$$\lVert l \rVert^{2} \leqslant \frac{4(e^{2 (T + \log(T))})}{\pi} \textit{.}$$
Ainsi, il existe $\tilde{C} > 0$ (indépendant de $i$ et de $L$) tel que pour tout $T$ assez grand, $$ \log (\lVert l \rVert)^{-1- \beta} \geqslant \tilde{C} T^{-1 - \beta}$$ 
et en particulier, 
\begin{equation}
\frac{\lVert \delta_{i} L \rVert^{2}}{2} \geqslant |l_{1} l_{2}| \geqslant C \tilde{C} T^{-1 - \beta} \textit{.}
\end{equation}
D'où le résultat voulu dans le cas où $\lVert L \rVert_{2} > 1$. L'autre cas est immédiat car on a toujours $\lVert l \rVert_{2} \geqslant \lVert L \rVert \geqslant \sqrt{\epsilon}$ et on peut toujours prendre $T$ plus grand ou $D$ plus petit si nécessaire.
\end{proof}
On peut maintenant aborder la démonstration de la proposition $\ref{prop14}$. L'idée principale est que si $i$ n'est pas un minimum local alors on a un sens à suivre qui va nous mener à un minimum local d'après le lemme $\ref{lemme17}$.
\begin{proof}[Démonstration de la proposition $\ref{prop14}$] 
$\bullet$ Dans cette première partie, on explique comment, à partir d'un terme qui a un dénominateur petit, on peut se ramener à un terme qui a un dénominateur petit et qui est un minimum local. \\
\\
On prend $T$ assez grand pour que $\log(\epsilon) + \log(T) - 2 > 0$. En particulier, on a $ \frac{1}{\epsilon T} < 1$. \\
 Par définition, on a $A_{2}(\epsilon,T,L) \subset A_{1}(\epsilon, T, L)$. \\
Réciproquement, soit $i \in A_{1}(\epsilon,T,L)$. Soit on a $i \in A_{2}(\epsilon,T,L)$, soit on a $ \lVert \delta_{i-1}L \rVert \leqslant \lVert \delta_{i}L \rVert$ ou $\lVert \delta_{i}L \rVert \geqslant \lVert \delta_{i+1}L \rVert$. \\
Plaçons-nous dans le cas où on a seulement $ \lVert \delta_{i-1}L \rVert \leqslant \lVert \delta_{i}L \rVert$ avec $T-1 \geqslant i \geqslant 0$.\\
On appelle $k \in \delta_{i}L$ tel que $\lVert \delta_{i}L \rVert=\lVert k \rVert_{2}$. On a alors, comme $\lVert \delta_{i}L \rVert \leqslant \sqrt{\frac{1}{\epsilon T}} < 1$ : 
$$ \lVert \delta_{i}L \rVert^{2} = k_{1}^{2} + k_{2}^{2} \textit{ et } \lVert \delta_{i-1}L \rVert^{2} = e^{-2}k_{1}^{2} + e^{2}k_{2}^{2} \textit{.}$$
Quitte à éliminer les réseaux non-faiblement admissibles (voir la définition $\ref{def3}$), qui constituent un ensemble négligeable pour $\mu_{\mathscr{S}_{2}}$, on peut supposer que $k_{1}, k_{2} \neq 0$. On a alors :  $$f_{(k_{1},k_{2})}(e^{-2})=\lVert \delta_{i-1}L \rVert^{2} \leqslant f_{(k_{1},k_{2})}(1)=\lVert \delta_{i}L \rVert^{2}  $$ et, en particulier, $f_{(k_{1},k_{2})}$ permet de calculer les plus courts vecteurs autour de l'instant $i$.\\
D'après le lemme $\ref{lemme17}$, il existe donc un unique $m_{i} \in \mathbb{Z} $ tel que $m_{i} < i$ et pour tout $m_{i} \leqslant k \leqslant i-1$, $ \lVert \delta_{k}L \rVert < \lVert \delta_{k+1}L \rVert$ et $\lVert \delta_{m_{i}-1}L \rVert > \lVert \delta_{m_{i}}L \rVert$. En particulier, $m_{i} \in A_{2}(\epsilon,T,L)$.\\
\\
Supposons qu'on se place maintenant dans le cas où on a seulement $\lVert \delta_{i}L \rVert \geqslant \lVert \delta_{i+1}L \rVert$ avec $0 \leqslant i \leqslant T-1$. De même qu'avant, on obtient un unique $m_{i}$ tel que $m_{i} \in \mathbb{Z}$, tel que $m_{i}>i$ et tel que pour tout $k \in [i,m_{i}-1]$, $\lVert \delta_{k}L \rVert> \lVert \delta_{k+1}L \rVert$ et $\lVert \delta_{m_{i}}L \rVert < \lVert \delta_{m_{i}+1}L \rVert $. En particulier, $m_{i} \in A_{2}(\epsilon,T,L)$. \\
\\
 Supposons maintenant qu'on soit dans le cas où $0 \leqslant i \leqslant T-1$ et $ \lVert \delta_{i-1}L \rVert \leqslant \lVert \delta_{i}L \rVert  \geqslant \lVert \delta_{i+1}L \rVert$. Le lemme $\ref{lemme17}$ exclut cela. \\
 \\
 Ainsi, pour tout $i \in A_{1}(\epsilon, T, L)$, on a défini un $m_{i} \in A_{2}(\epsilon, T, L)$ qui s'interprète comme le "premier minimum local que l'on peut rencontrer en suivant la direction indiquée"(dans le cas où $i \in A_{2}(\epsilon, T, L)$, on pose simplement $m_{i} = i$). \\
 \\
$\bullet$ Dans un second temps, on construit les $k_{d}(i)$ et les $k_{m}(i)$ et on démontre que l'étude de $S_{1}$ se ramène à l'étude de $S_{2}$ grâce à ce qui vient d'être fait. \\
\\
On pose, pour tout $i \in A_{2}(\epsilon, T, L)-\{T-1\}$, $k_{m}(i)$ le plus grand entier positif tel que pour tout $k \in [0, k_{m}(i)-1 ]$, $\lVert \delta_{i+k} L \rVert \leqslant \lVert  \delta_{i+k+1}L \rVert \leqslant \sqrt{\frac{1}{\epsilon T}}$. On définit pour l'autre sens et de manière analogue $k_{d}(i)$ pour tout $i \in A_{2}(\epsilon, T, L)-\{0 \}$. On notera que $k_{d}(i) \leqslant 0$ et $k_{m}(i) \geqslant 0$. \\
On pose aussi, dans le cas où $0 \in A_{2}(\epsilon, T, L)$, $k_{d}(0) = -\left\lfloor \frac{\log(\epsilon)}{2} + \frac{\log(T)}{2} -2 \right\rfloor$. On fait de même dans le cas où $ T-1 \in A_{2}(\epsilon, T, L)$ avec $k_{m}(T-1)$. 
On a pas noté les dépendances en $L$, en $\epsilon$ et en $T$ de $k_{d}$ et $k_{m}$ dans un souci de simplification. \\
 \\
Le raisonnement nous permettant de construire $m_{i}$ nous donne en sus qu'étudier la convergence en loi de $S_{1}$ revient à étudier la convergence en loi de 
\begin{equation}
\label{eq105}
S_{2}(\omega, L,\epsilon, T) = \sum_{i \in A_{2}(\epsilon, T, L)} \sum_{ k_{d}(i) \leqslant k \leqslant k_{m}(i) } \frac{\theta_{i+k}(\omega)}{T \lVert \delta_{i+k}L \rVert^{2}} 
\end{equation}
($S_{2}$ contient éventuellement des termes avant $0$ et après $T-1$). \\
En effet, le lemme $\ref{lemme28}$ s'applique et ainsi on a, pour tout $T$ assez grand : 
\begin{equation}
\label{eq118}
P( |S_{1}(\omega, L, \epsilon, T)-S_{2}(\omega, L, \epsilon, T)| \neq 0 ) \leqslant D \log(T)T^{ \beta} \lVert \theta_{0} \rVert_{\infty} \mathbb{P}(T \lVert L \rVert^{2} \leqslant \frac{1}{\epsilon}) \textit{, }
\end{equation}
en utilisant l'équation ($\ref{eq117}$), la compacité de la loi commune des $\theta_{i}$ et où $D > 0$. En choisissant $\beta > 0$ tel que $ \beta < 1$, on obtient le résultat voulu grâce au lemme $\ref{lemme1}$. On est donc ramenés à $S_{2}$.\\
\\
Soit $i \in A_{1}(\epsilon, T, L)$. Notons que $ \lVert \delta_{i} L \rVert^{2} \leqslant \frac{1}{\epsilon T}$ et que donc pour $k \in \mathbb{Z}$ tel que $ |k| < \frac{\log(\epsilon)}{2} + \frac{\log(T)}{2} $, $\lVert \delta_{k} L \rVert^{2} < 1$ d'après le lemme $\ref{lemme14}$. \\
En utilisant le lemme $\ref{lemme15}$ et le lemme $\ref{lemme17}$, on voit alors que pour tout $i \in A_{1}(\epsilon, T, L)$, les ensembles $$[- \lfloor \frac{\log(\epsilon)}{2} + \frac{\log(T)}{2} -2 \rfloor, \lfloor \frac{\log(\epsilon)}{2} + \frac{\log(T)}{2} -2 \rfloor]+\{i \}$$ sont disjoints pour tout $0 \leqslant i \leqslant T-1$.\\
 Quitte à diminuer les $k_{d}(i)$ et augmenter les $k_{m}(i)$ pour avoir $k_{d}(i) \leqslant - \lfloor \frac{\log(\epsilon)}{2} + \frac{\log(T)}{2} -2 \rfloor$ et $ k_{m}(i) \geqslant \lfloor \frac{\log(\epsilon)}{2} + \frac{\log(T)}{2} -2 \rfloor$ et ce pour tout $i$, on peut supposer effectivement ces deux inégalités satisfaites. \\
En effet, cela revient à rajouter des instants $t \in [- 2 \log(T), T + 2 \log(T) ]$ dans $S_{2}$ tels que $$T \lVert \delta_{t} L \rVert^{2} > \frac{1}{\epsilon}$$
et les équations $(\ref{eq100})$ et $(\ref{eq101})$ montrent que cet ajout est négligeable. 
 \end{proof}
On peut donc maintenant concentrer notre attention sur l'étude asymptotique de $S_{2}$. 
\subsection{Étude du comportement asymptotique de $S_{2}$}
Posons :
\begin{equation}
\label{eq107}
\Gamma_{i}(\omega,L,\epsilon,T)= \sum_{ k_{d}(i) \leqslant k \leqslant k_{m}(i) } \frac{\theta_{i+k}(\omega) \lVert \delta_{i}L \rVert^{2}}{\lVert \delta_{i+k}L \rVert^{2}} \textit{, }
\end{equation}
\begin{equation}
\label{eq108}
\Xi_{i}(L,T)= T \lVert \delta_{i}L \rVert^{2} 
\end{equation}
de sorte que l'on a : 
\begin{equation}
\label{eq109}
S_{2}(\omega,L,\epsilon, T) = \sum_{i \in A_{2}(\epsilon, T, L)} \frac{\Gamma_{i}(\omega,L,T,\epsilon)}{\Xi_{i}(L,T)} \textit{.}
\end{equation}
Pour appliquer le lemme $\ref{lemme41}$ et ainsi obtenir le théorème $\ref{thm21}$, on va maintenant montrer que, dans notre cas : $\{ \Gamma_{i}(\omega,L,\epsilon,T) \}_{i \in A_{2}(\epsilon, T, L) }$ sont asymptotiquement des variables aléatoires réelles indépendantes identiquement distribuées, symétriques et dont le support est compact, et indépendantes de $\{ \Xi_{i}(L,T) \}_{i \in A_{2}(\epsilon, T, L) }$ et que $\{ \Xi_{i}(L,T) \}_{i \in A_{2}(\epsilon, T, L) }$ converge vers un processus de Poisson sur $[0, \frac{1}{\epsilon}]$ d'intensité constante indépendante de $\epsilon$. Concentrons-nous tout d'abord sur le comportement asymptotique des $\Gamma_{i}$. 
\subsubsection{Comportement asymptotique des $\Gamma_{i}$ }
Dans cette sous-section, on va prouver la proposition suivante qui détermine le comportement asymptotique des $\Gamma_{i}$ : 
\begin{prop}
\label{prop15}
$\{ \Gamma_{i}(\omega,L,\epsilon,T) \}_{i \in A_{2}(\epsilon, T, L) }$ sont asymptotiquement des variables aléatoires réelles indépendantes, de même loi, symétriques, à support compact et sont indépendantes de $\{ \Xi_{i}(T,L) \}_{i \in A_{2}(\epsilon, T, L) }$.
\end{prop}
Pour démontrer cette proposition, on a besoin de quelques lemmes : 
\begin{lemma}
\label{lemme18}
Pour tout $T$ assez grand, pour tout $i \in A_{2}(\epsilon, T, L)$, pour tout $k \in [k_{d}(i),k_{m}(i)]$, il existe $K>0$ tel que
\begin{equation}
 \frac{\lVert \delta_{i} L \rVert^{2}}{\lVert \delta_{i+k} L \rVert^{2}} \leqslant  \frac{K}{\cosh(2 |k|)} \textit{.}
\end{equation}
\end{lemma}
\begin{proof}
Soient $i$ et $k$ comme dans l'énoncé du lemme. En appelant $h \in (\delta_{i}L)-\{ 0 \}$ le vecteur tel que $\lVert \delta_{i} L \rVert = \lVert h \rVert$ et tel que $h_{1} \geqslant 0$, on a d'après le lemme $\ref{lemme15}$ : 
\begin{equation}
\label{eq110}
 \frac{\lVert \delta_{i} L \rVert^{2}}{\lVert \delta_{i+k} L \rVert^{2}} = \frac{h_{1}^{2}+h_{2}^{2}}{e^{2k} h_{1}^{2} + e^{-2k} h_{2}^{2}} \textit{.}
\end{equation}
En appelant $(r, \alpha)$ les coordonnées polaires de $h$ (et donc $r > 0$ et $\alpha \in [- \frac{\pi}{2}, \frac{\pi}{2} ]$), on a : 
\begin{equation*}
 \frac{\lVert \delta_{i} L \rVert^{2}}{\lVert \delta_{i+k} L \rVert^{2}} = \frac{1}{\cos^{2}(\alpha)e^{2k} + \sin^{2}(\alpha)e^{-2k}} \textit{.}
\end{equation*}
Or, on a : $\frac{\lVert \delta_{i} L \rVert^{2}}{\lVert \delta_{i+k} L \rVert^{2}} \leqslant 1$ pour $k=-1$ et $k=1$. Ainsi il existe $\frac{\pi}{2}>\kappa > 0$ tel que : $ \frac{\pi}{2}- \kappa > |\alpha| > \kappa$. D'où le lemme.
\end{proof}
Notons au passage que d'après les hypothèses faites juste après le lemme $\ref{lemme28}$ avec $\beta= \frac{1}{2}$ et d'après le théorème de Minkowski qui donne que pour tout réseau $L'$, $\lVert L' \rVert \leqslant \frac{4}{\pi}$, on a donc d'après le lemme $\ref{lemme18}$ et la proposition $\ref{prop14}$ que pour tout $i$, $|k_{d}(i)|,|k_{m}(i)| \leqslant \log(T)$ dès que $T$ est assez grand. \\
Grâce au lemme précédent, on peut établir que $k_{d}(i)$ et $k_{m}(i)$ peuvent être choisis de sorte à ne pas dépendre de $\lVert \delta_{i}L \rVert$ pour $i \in A_{2}(\epsilon,T,L)$. C'est tout l'objet du lemme suivant : 
\begin{lemma}
\label{lemme60}
Posons : $$S_{3}(\omega, L,\epsilon, T) = \sum_{i \in A_{2}(\epsilon, T, L)} \frac{1}{\Xi_{i}(L,T)}  \sum_{ -(\frac{\log(T)}{2} + C_{\epsilon}) \leqslant k \leqslant \frac{\log(T)}{2} + C_{\epsilon} } \frac{\theta_{i+k}(\omega) \lVert \delta_{i} L \rVert^{2}}{ \lVert \delta_{i+k}L \rVert^{2}}  \textit{.} $$
Alors, quand $T \rightarrow \infty$, $S_{3}(\omega, L,\epsilon, T) - S_{2}(\omega, L,\epsilon, T)$ tend presque sûrement vers $0$.
\end{lemma}
Ainsi, on pourra considérer par la suite que : $k_{d}(i) = -(\frac{\log(T)}{2} + C_{\epsilon})$ et $k_{m}(i) =\frac{\log(T)}{2} + C_{\epsilon}$.
\begin{proof}
On a : 
\begin{equation}
\label{eq300}
|S_{3}(\omega, L,\epsilon, T) - S_{2}(\omega, L,\epsilon, T)| \leqslant \sum_{i \in A_{2}(\epsilon, T, L)} (\sum_{ |k| \geqslant \frac{\log(T)}{2} + C_{\epsilon}}  \frac{2 K \lVert \theta_{0} \rVert_{\infty}}{\cosh(2 |k|)}) \frac{1}{T \lVert \delta_{i} L\rVert^{2}} 
\end{equation}
car $|k_{d}(i)|, |k_{m}(i)| \geqslant \frac{\log(T)}{2} + C_{\epsilon}$ et car $\Xi_{i}(L,T)= T \lVert \delta_{i}L \rVert^{2}$. \\
Or, on a : $$ \sum_{ |k| \geqslant \frac{\log(T)}{2}+ C_{\epsilon}} \frac{1}{\cosh(2 |k|)} \sim \frac{H_{\epsilon}}{T}$$ quand $T \rightarrow \infty$ et où $H_{\epsilon}>0$. \\
D'où, il vient, d'après l'équation ($\ref{eq300}$) que : 
\begin{equation}
\label{eq301}
|S_{3}(\omega, L,\epsilon, T) - S_{2}(\omega, L,\epsilon, T)| \leqslant \sum_{i \in A_{2}(\epsilon, T, L)} \frac{M_{\epsilon}}{T^{2} \lVert \delta_{i} L\rVert^{2}} 
\end{equation}
où $M_{\epsilon} > 0$. \\
Or la quantité de droite de l'inéquation ($\ref{eq301}$) converge presque sûrement vers $0$ quand $T \rightarrow \infty$ d'après le lemme 3.2 de $\cite{Skriganov}$ (qui est basé sur le théorème ergodique individuel). D'où le résultat voulu.
\end{proof}
\begin{lemma}
\label{lemme51}
Supposons $\lVert L \rVert < 1$. Soit $\alpha(L)$ l'angle entre la droite engendré par un vecteur $ l$ de $L$ tel que $\lVert l \rVert = \lVert L \rVert$ et l'axe des abscisses. Conditionnellement à l'évènement $(\lVert L \rVert < 1)$, $\alpha(L)$ est indépendant de $ \lVert L \rVert$. 
\end{lemma}
\begin{proof}
Soient $- \frac{\pi}{2} \leqslant a \leqslant b \leqslant \frac{\pi}{2}$ et $ 0 < r_{1} \leqslant r_{2} < 1$. On pose $$f_{1}(x) = \mathbf{1}_{r_{1} \leqslant \lVert x \rVert \leqslant r_{2}}$$ et $$f_{2}(x) =  \mathbf{1}_{0 < \lVert x \rVert < 1} \mathbf{1}_{\alpha(x) \in [a,b]}$$ où $\alpha(x)$ désigne l'angle orienté entre $(1,0)$ et $x$.\\
Posons : 
$$\overline{F}(L) = \sum_{l_{1} \neq \pm l_{2} \in L } f_{1}(l_{1}) f_{2}(l_{2}) \textit{.} $$
Alors le lemme $24$ s'applique et nous donne,  que : 
\begin{align*}
& \mu_{2}(\alpha(L) \in [a,b], \lVert L \rVert \in [r_{1},r_{2}])  = \frac{1}{2} \mathbb{E}(\overline{F})  \\
& = \frac{1}{2} \zeta(2)^{-2} \int_{\mathbb{R}^{2}} f_{1}(x) dx \int_{\mathbb{R}^{2}} f_{2}(x) dx 
\end{align*}
En passant en coordonnées polaires, on a : 
$$\int_{\mathbb{R}^{2}} f_{1}(x) dx = \int_{ \substack{ r_{1} < r < r_{2} \\ 0 \leqslant \alpha \leqslant 2 \pi }} r dr d \alpha = \pi (r_{2}^{2} - r_{1}^{2})$$
et 
$$\int_{\mathbb{R}^{2}} f_{2}(x) dx = \int_{\substack{0 < r < 1 \\ a \leqslant \alpha \leqslant b}} r dr d\alpha = \frac{b-a}{2} \textit{.} $$
D'où on tire : 
$$ \mu_{2}(\alpha(L) \in [a,b], \lVert L \rVert \in [r_{1},r_{2}] \text{ }| \text{ } \lVert L \rVert < 1)   = \mu_{2}(\alpha(L) \in [a,b]\text{ }| \text{ } \lVert L \rVert < 1) \mu_{2}(\lVert L \rVert \in [r_{1},r_{2}] \text{ }| \text{ } \lVert L \rVert < 1)\textit{.}$$
\end{proof}
On peut maintenant prouver la proposition $\ref{prop15}$. 
\begin{proof}[Démonstration de la proposition $\ref{prop15}$]
La série des $\sum_{k \in \mathbb{Z}} \frac{1}{\cosh (2 |k|)}$ converge. \\
Ainsi, grâce au lemme $\ref{lemme18}$, grâce au fait que $\lVert \theta_{0} \rVert_{\infty} < \infty$ et en utilisant ($\ref{eq107}$), on obtient que les $ \Gamma_{i} $, pour $i \in A_{1}(T,\epsilon,L)$, convergent vers des variables aléatoires réelles à support compact. Ces variables sont nécessairement symétriques dans la mesure où les $\theta_{i}$ sont symétriques et indépendantes entre elles. \\
Par ailleurs, comme les $\theta_{i}$ sont des variables aléatoires réelles indépendantes entre elles, comme $\mu_{2}$ est $SL_{2}(\mathbb{R})$-invariante et comme $|k_{d}(i)|,|k_{m}(i)| \rightarrow \infty$, il est clair qu'asymptotiquement les $\Gamma_{i}$ sont identiquement distribuées. \\
\\
Voyons maintenant qu'asymptotiquement les $\Gamma_{i}$ sont indépendants des $\Xi_{i}$. \\
Regardons tout d'abord ce qu'il se passe entre $\Gamma_{i}$ et $\Xi_{i}$. Rappelons-nous que $\Xi_{i} = T || \delta_{i} L||^{2}$ et que $\Gamma_{i}$ s'écrit : 
\begin{equation}
\label{eq111}
\Gamma_{i}(\omega,L,T,\epsilon)= \sum_{ k_{d}(i) \leqslant k \leqslant k_{m}(i) } \frac{\theta_{i+k}(\omega) \lVert \delta_{i}L \rVert^{2}}{\lVert \delta_{i+k}L \rVert^{2}} \textit{, } \nonumber
\end{equation}
où $k_{m}(i) = \frac{\log(T)}{2} + C_{\epsilon} $ et $k_{d}(i) = - k_{m}(i)$. \\
L'équation $(\ref{eq110})$ montre que $\frac{\theta_{i+k}(\omega) \lVert \delta_{i}L \rVert^{2}}{\lVert \delta_{i+k}L \rVert^{2}}$ ne dépend pas de $\lVert \delta_{i}L \rVert$ mais seulement de $\alpha(\delta_{i}L)$. \\
Or, $\alpha(\delta_{i}L)$ est indépendant de $\lVert \delta_{i}L \rVert$, comme $\mu_{2}$ est une mesure de Haar à gauche et à droite et d'après le lemme $\ref{lemme51}$. \\
Ainsi, asymptotiquement, $\Gamma_{i}$ est indépendant de $\Xi_{i} = T \lVert \delta_{i} L \rVert^{2}$.\\
\\
Enfin, notons que pour $i \neq j \in A_{2}(\epsilon,L,T)$, on peut créer un écart entre les blocs $\{i+k \}$ et $\{j+k \}$ de taille, par exemple, au moins $\log(\log(T))$, c'est-à-dire à la place de considérer $k_{m}(h)$, considérer $\tilde{k}_{m}(h) = k_{m}(h)-\frac{\log(\log(T))}{2}$ et à la place de considérer $k_{d}(h)$ considérer $\tilde{k}_{d}(h) = k_{d}(h) + \frac{\log(\log(T))}{2}$, sans que cela change quelque chose de significatif dans l'étude asymptotique des $\Gamma_{h}$, des $\Xi_{h}$ ou de $S_{2}$ (voir les équations $\ref{eq300}$ et $\ref{eq301}$).
Finalement, en utilisant la vitesse exponentielle de mélange du flot géodésique sur $\mathscr{S}_{2}$ et en utilisant l'indépendance des $\theta_{k}$, on obtient qu'asymptotiquement : \begin{itemize}
\item[1)] Les $\Gamma_{i}$ sont asymptotiquement indépendants entre eux
\item[2)] Les $ \Gamma_{i}$ sont asymptotiquement indépendants de $\{ \Xi_{i} \}$
\end{itemize}
\end{proof}
Étudions maintenant le comportement asymptotique des $\{ \Xi_{i} \}_{i \in A_{2}(\epsilon, T, L)}$.
\subsubsection{Convergence des $\{ \Xi_{i} \}_{i \in A_{1}(\epsilon, T, L)}$ vers un processus de Poisson sur $[0,\frac{1}{\epsilon}]$ et démonstration du théorème $\ref{thm21}$}
Le but de cette section est d'établir la proposition suivante : 
\begin{prop}
\label{prop16}
$\{ \Xi_{i}(L,T) \}_{i \in A_{2}(\epsilon, T, L) }$ converge vers un processus de Poisson sur $[0, \frac{1}{\epsilon}]$ d'intensité constante $D = \zeta(2)^{-1} \int_{0}^{\frac{\pi}{2}} \mathbf{1}_{e^{-2} \cos^{2}(\theta) + e^{2} \sin^{2}(\theta) \leqslant 1 } \mathbf{1}_{e^{2} \cos^{2}(\theta) + e^{-2} \sin^{2}(\theta)\leqslant 1} d \theta > 0$.
\end{prop}
Cette proposition nous donne une meilleure idée de la répartition à l'infini des points de la trajectoire géodésique dans l'espace $\mathscr{S}_{2}$. C'est aussi à ce genre de chose que l'on s'intéresse dans $\cite{paulin2016logarithm}$. \\
Pour établir la convergence voulue, on va s'appuyer sur le théorème $\ref{thm9}$.\\
Ici, quitte à regarder la somme qui nous intéresse comme allant de $0$ à $T$ (au lieu de $T-1$), $T$ joue le rôle de $M$. \\
Par ailleurs, on pose $$\nu_{0}^{T}(L) = (\mathbf{1}_{\lVert \delta_{-1} L \rVert \geqslant \lVert L \rVert \leqslant \lVert \delta_{1} L \rVert}) T\lVert L \rVert^{2}  \textit{,} $$ $$\nu_{t}^{T}(L) = \nu_{0}^{T}(\delta_{t} L)  \textit{,} $$ $$\xi_{t}^{T}(L)= \mathbf{1}_{\nu_{t}^{T}(L) \in ]0,\frac{1}{\epsilon}]}$$
et la suite de partitions considérée $\mathbf{Q}$ est donnée par $(]\frac{k}{p}\frac{1}{\epsilon}, \frac{k+1}{p} \frac{1}{\epsilon}])_{p \in \mathbb{N}-\{0 \} \textit{, } k \in [0, p-1 ] }$. \\
Remarquons que l'on a, pour $T$ assez grand (par rapport à $\epsilon > 0$) : \begin{equation}
\label{eq113}
\xi_{t}^{T}(L) = \frac{1}{2} F_{0}^{T}(\delta_{t} L)
\end{equation}
où $F_{0}^{T}(L)$ est la transformée de Siegel de la fonction $f_{0}^{T}$ telle que : 
\begin{equation*}
 f_{0}^{T}(x)= \mathbf{1}_{\lVert \delta_{-1} x \rVert \geqslant \lVert x \rVert \leqslant \lVert \delta_{1}  x \rVert} \mathbf{1}_{0 < T \lVert  x \rVert^{2} \leqslant \frac{1}{\epsilon}} \textit{.}
 \end{equation*} 
Enfin, $\lambda_{1}$ a été définie dans la définition $\ref{def4}$.  \\
Maintenant que nous avons fait ces choix, nous pouvons énoncer la proposition suivante qui s'énonce en deux parties : 
\begin{prop}
\label{prop40}
On a : \\
\begin{itemize}
\item $\{ \Xi_{t}(L,T) \}_{t \in A_{2}(\epsilon, T, L) }$ est exactement le même processus que $\{\nu_{t}^{M} \}_{\xi_{t}^{M}=1 \textit{, }t \in [0,M]} $
\item $(\lambda_{1},\{ \xi_{t} \},\{ \nu_{t} \})$ satisfont les hypothèses (h1) jusqu'à (h6) du théorème $\ref{thm9}$.
\end{itemize}
\end{prop}
La première partie de la proposition est évidente au vu des choix faits. Il reste donc à démontrer la deuxième partie de la proposition sur laquelle on va se concentrer maintenant. \\
Pour voir que les hypothèses (h1) jusqu'à (h6à sont vérifiées, nous allons suivre une démarche analogue à celle suivie dans \cite{bassam}. Seulement analogue, car les zones à l'infini dans $\mathscr{S}_{2}$ qui nous intéressent ne sont pas les mêmes, ce qui explique que des différences, des adaptations apparaissent dans la preuve. \\
Avant de montrer $\ref{prop40}$, nous avons besoin de quelques lemmes préliminaires.
\begin{lemma}
\label{lemme9}
Pour tout $t \neq t' \in \mathbb{Z}$, pour tout $b > 0$, pour tout $T > 0$ assez grand (la grandeur dépendant uniquement de $b$), on a : 
$$\mathbb{P}(T \lVert \delta_{t} L \rVert^{2} \leqslant b, T ||\delta_{t'} L ||^{2} \leqslant b) = \frac{D_{1}b^{2}}{T^{2}} + \frac{D_{2}b}{T} \arccos( \frac{1-e^{-2 |t-t'|}}{1+e^{-2 |t-t'|}})  $$
où $D_{1},D_{2} > 0$.
\end{lemma}
\begin{proof}
Posons pour tout $x \in \mathbb{R}$, $f_{t}(x) = \mathbf{1}_{\lVert\delta_{t}x \rVert^{2} \leqslant \frac{b}{T}}$ et $f_{t'}(x) = \mathbf{1}_{||\delta_{t'}x||^{2} \leqslant \frac{b}{T}}$ et $F_{t}$ et $F_{t'}$ les transformées de Siegel correspondantes. Alors, en suivant la démarche du lemme $\ref{lemme1}$ et en appliquant les formules du lemme $\ref{lemme24}$ et en remarquant que pour tout $r > 0$ 
$$\int_{\mathbb{R}^{2}} \mathbf{1}_{B_{f}(0,r)} (\delta_{t} \cdot) \mathbf{1}_{B_{f}(0,r)} (\delta_{t'} \cdot) = 2 r^{2} \arccos (\frac{1-e^{-2 |t-t'|}}{1+e^{-2 |t-t'|}}) \textit{, }$$
on obtient le résultat voulu.
\end{proof}
Pour voir que (h4) et (h5) sont vérifiées, on va s'inspirer de ce qui a été fait dans $\cite{bassam}$. On pose ainsi $K = ]0, \frac{1}{\epsilon}]$ et $K_{i} = ]\frac{i}{p} \frac{1}{\epsilon}, \frac{i+1}{p} \frac{1}{\epsilon}]$ pour $0 \leqslant i \leqslant p-1$ avec $p \geqslant 1$ ($p$ est sous-entendu). On a avec ces notations : $\xi_{t} = \mathbf{1}_{\nu_{t} \in K}$ et $\xi_{t,i} = \mathbf{1}_{\nu_{t} \in K_{i}}$. \\
On pose de même $\hat{K} = \{ x \in ]0, \infty [ \text{ } | \text{ } d(x,\partial K) \leqslant T^{-1000} \}$ et de même $\hat{K_{i}}= \{ x \in ]0, \infty [ \text{ } | \text{ } d(x,\partial K_{i}) \leqslant T^{-1000} \}$. \\ 
On pose $\hat{\xi_{t}}= \mathbf{1}_{ \nu_{t} \in \hat{K}}$ et de même $\hat{\xi_{t,i}} = \mathbf{1}_{ \nu_{t} \in \hat{K_{i}}}$. On dispose alors des lemmes suivants : 
\begin{lemma}
\label{lemme29}
On a : 
\begin{equation}
\label{eq141} 
\lVert \xi_{0}^{T} \rVert_{H^{s,s}} = O(1) \textit{.}
\end{equation}
Pour tout $i \in \{1,\dots, p\}$, on a aussi : 
\begin{equation}
\label{eq142}
\lVert \xi_{0,i}^{T} \rVert_{H^{s,s}} = O(1) \textit{.}
\end{equation}

\end{lemma}
\begin{proof}
Le lemme $\ref{lemme19}$, appliqué à $F_{0}^{T}$, donne que $F_{0}^{T} \in H^{s,s}$ et que $\lVert F_{0}^{T} \rVert_{H^{s,s}} = O(1)$. \\
Pour démontrer ce dernier point, on procède en approchant via des fonctions affines par morceaux la fonction indicatrice $\mathbf{1}_{B_{f}(0, \frac{1}{\sqrt{\epsilon T}})}$ et en régularisant ces fonctions à l'aide d'approximation de l'unité. \\
D'où le premier résultat voulu d'après ($\ref{eq113}$). On procède de même pour obtenir le second résultat. 
\end{proof}
\begin{lemma}
\label{lemme30}
On a : 
\begin{align}
\label{eq123}
\mu_{\mathscr{S}_{2}}(\hat{\xi_{0,i}} ) = O(T^{-1000}) & \textit{ et } \\
\mu_{\mathscr{S}_{2}}(\hat{\xi_{0}}) = O(T^{-1000})
\end{align}
\end{lemma}
\begin{proof}
C'est une application du lemme $\ref{lemme24}$.
\end{proof}
Prouvons maintenant la proposition $\ref{prop40}$.
\begin{proof}[Démonstration de la proposition $\ref{prop40}$]
$\bullet$ Vérification de (h1) \\
On rappelle que pour $T$ assez grand (par rapport à $\epsilon > 0$)$$
\xi_{t}^{T}(L) = \frac{1}{2} F_{t}^{T}(L)
$$
où $F_{t}^{T}(L)$ est la transformée de Siegel de la fonction $f_{t}^{T}$ telle que : 
\begin{equation}
\label{eq114}
 f_{t}^{T}(x)= \mathbf{1}_{\lVert \delta_{-1}\delta_{t}x \rVert\geqslant \lVert\delta_{t}x \rVert \leqslant \lVert \delta_{1} \delta_{t} x \rVert} \mathbf{1}_{0 < T \lVert \delta_{t} x \rVert^{2} \leqslant \frac{1}{\epsilon}} \textit{.}
 \end{equation}
En utilisant le lemme $\ref{lemme24}$, en faisant un changement de variable $u= \delta_{t} x$ et en passant en coordonnées polaires, on trouve que : 
\begin{equation}
\label{eq115}
E(\xi_{t}^{T}) = \frac{D}{T} \frac{1}{\epsilon}
\end{equation}
où $$D = \zeta(2)^{-1} \int_{0}^{\frac{\pi}{2}} \mathbf{1}_{e^{-2} \cos^{2}(\theta) + e^{2} \sin^{2}(\theta) \leqslant 1 } \mathbf{1}_{e^{2} \cos^{2}(\theta) + e^{-2} \sin^{2}(\theta)\leqslant 1} d \theta \textit{.}$$ 
Avant de continuer, notons qu'une étude des variations de la fonction $\theta \longmapsto \frac{1}{a^{2}} \cos^{2}(\theta) + a^{2} \sin^{2}(\theta)$, avec $a > 0$, permet de voir qu'en réalité $D$ vaut :
$$D = \zeta(2)^{-1} (\arccos(\frac{1}{\sqrt{e^{2}+1}}) - \arccos(\frac{e}{\sqrt{e^{2}+1}})) \textit{.} $$
(qui peut encore se simplifier en utilisation une formule de différence d'$\arccos$) \\
\\
$\bullet$ Vérification de (h2)  \\
(h2) est vérifiée car $\xi_{t} \in \{0, 1 \}$. \\
\\
$\bullet$ Vérification de (h3) \\
Si on a $t,t' \in \Pi$, différents tels que $\xi_{t},\xi_{t'}=1$ alors $t,t' \in A_{2}(\epsilon, T, L)$ et sont donc séparés au moins de $\log(T)+ 2C_{\epsilon}$ d'après ce que l'on a vu précédemment. \\
En utilisant le lemme $\ref{lemme9}$, on voit donc que l'on a :
 \begin{equation}
\label{eq120}
P(\xi_{t} \geqslant 1, \xi_{t'} \geqslant 1) \leqslant \mathcal{O}(\frac{1}{T} \arccos (\frac{1-e^{-2 (\log(T)+ 2C_{\epsilon})}}{1+e^{-2 (\log(T)+ 2C_{\epsilon})}}) )
\end{equation}
et le terme en $\arccos$ est un $\mathcal{O}(\frac{1}{T})$. \\
\\
$\bullet$ Vérification de (h4a), (h4b) et (h4c) \\
Dans notre cas, notons que $\eta_{t,p}= \xi_{t,p}$. Il suffit donc d'avoir (h4a) et (h4b) vérifiées pour obtenir que les trois hypothèses sont vérifiées. \\
On se fixe $T > 0$. On se fixe $t \in \Pi$ (l'exposant $T$ est sous-entendu). On appelle $\Pi^{+}(t)$ l'ensemble des éléments $t' \in [0,T]$ tels que $\lambda_{1}(t') > \lambda_{1}(t) + R \log(T)$ où $R > 0$ sera fixé par la suite. On considère $\mathfrak{F}= \{\xi_{0}, \xi_{0,1},  \cdots , \xi_{0,p} \}$. On appelle $F_{t}$ une partition en $W_{1}$-courbes de taille $L_{t}= (e^{\lambda_{1}(t)}T^{1000})^{-1}$ donnée par la proposition $\ref{prop13}$ qui est $\kappa_{0}$-représentative relativement à $(\Pi^{+}(t), \mathfrak{F})$. C'est possible dans la mesure où : 
\begin{equation}
\label{eq122}
\sum_{t \in \Pi \text{ } t' \in \Pi^{+}(t)} (L_{t}e^{\lambda_{1}(t')})^{-\kappa_{0}} \leqslant T^{2} T^{1000 \kappa_{0}} T^{-R \kappa_{0}} = O(T^{-1000})
\end{equation}
pour $R$ choisi assez grand.\\
En particulier, si on appelle $\tilde{E}_{1}$ l'ensemble des $L \in \mathscr{S}_{2}$ tel que pour tout $t \in \Pi$, $L$ est un point représentatif de $(F_{t}, \Pi^{+}(t), \mathfrak{F})$, on voit que l'on a $\mu_{2}(\tilde{E}_{1}^{c}) = O(T^{-1000})$. \\
Pour $L \in \tilde{E}_{1}$, d'après la définition de point représentatif, d'après le calcul fait pour obtenir l'égalité $(\ref{eq115})$, d'après le lemme $\ref{lemme29}$, on voit que $(h4)$ est vérifiée pour ce choix de $E=\tilde{E}_{1}$ et cette partition $F_{t}$. \\ 
\\
$\bullet$ Vérification de (h5) \\
Donnons-nous $t \in [0,T]$ et $\overline{t} \in \Pi^{+}(t)$ et $\gamma_{\overline{t}} \in F_{\overline{t}}$. Quitte à prendre $R$ encore plus grand, on peut supposer que la longueur de $\delta_{t} \gamma_{t}$, qui vaut $e^{- \lambda_{1}(\overline{t})} T^{-1000} e^{\lambda_{1}(t)}$, est plus petite que $T^{-10^{9}}$. \\
Ceci étant dit, si $\xi_{t}$ est nul sur $\gamma_{t}$, le résultat est clairement acquis sur $\gamma_{\overline{t}}$. Sinon, il existe $\overline{L} \in \gamma_{\overline{t}}$ et $p$ tel que $\xi_{t,p}(\overline{L})= 1$ et il s'en suit donc qu'il existe $u \in  \overline{L}$ tel que $ \nu_{t} \in K_{p}$.\\
Par ailleurs, notons que $\delta_{t} \gamma_{\overline{t}} = \{ h(\tau) \delta_{t} \overline{L} \}_{\tau \in \mathscr{I} }$ où $\mathscr{I}$ est un intervalle fermé possédant $0$ de longueur plus petite que $T^{-10^{9}}$. On observe alors que $h(\tau) \delta_{t} \overline{L}$ admet $h(\tau) \delta_{t} u$ comme plus petit vecteur et $$T \lVert h(\tau) \delta_{t} u \rVert^{-2} = T \lVert \delta_{t} u \rVert^{-2} + O(T^{-10^{9}+1})\textit{.}$$ 
Pour peu que l'on remplace la condition $t \in A_{2}(\epsilon, T, L)$ par $$t \in \tilde{A}_{2}(\epsilon, L,T) = \{ i \in [0,T-1] \cap A(\epsilon,T,L) \textit{ } | \textit{ }  \lVert \delta_{i-1}L \rVert > \lVert \delta_{i}L \rVert+ \frac{1}{T^{2001}}< \lVert \delta_{i+1}L \rVert  \} \textit{,}$$ si $\delta_{t} \gamma_{\overline{t}}$ n'intersecte pas (ou du moins son image) $$ \{ x \in ]0, \infty [ \text{ } | \text{ } d(x,\partial K_{p}) \leqslant T^{-10^{3}} \} \textit{, }$$ elle est complètement incluse dans $K_{p}$ et donc $\xi_{t,p} = 1$ le long de cette courbe.\\
La mesure de l'ensemble des $L$ tel que $\delta_{t} \gamma_{\overline{t}}$ intersecte $\{ x \in ]0, \infty [ \text{ } | \text{ } d(x,\partial K_{p}) \leqslant T^{-10^{3}} \}$ est donc $O(T^{-1000})$ car cet ensemble est inclus dans l'ensemble des $L$ qui « sont » dans $\hat{K_{p}}$ et qui a une mesure en $O(T^{-1000})$ d'après le lemme $\ref{lemme30}$. Et l'ensemble négligé par le remplacement au début de ce paragraphe ne coûte aussi qu'un  $O(T^{-1000})$.\\
En prenant le complémentaire de l'union de tous ces exceptionnels $L$ pour $t \in [0,T]$, pour $\overline{t} \in \Pi^{+}(t)$ (et en considérant son intersection avec $E_{1}$), on obtient un ensemble mesurable $E_{2}$ tel que $\mu_{2}(E_{2}^{c})=O(T^{-998})$ sur lequel (h5) est valide (ainsi que (h4)). \\
\\
$\bullet$ Vérification de (h6) \\
La taille des pièces de $F_{t}$ est $L_{t} = (e^{\lambda_{1}(t)} T^{1000})^{-1}$ et celle de $F_{\overline{t}}$ est $L_{\overline{t}} = (e^{\lambda_{1}(\overline{t})}T^{1000})^{-1}$, si on pose $$E_{3} = \{ L \in \mathscr{S}_{2} \text{ } | \text{ } F_{\overline{t}}(L) \subset F_{t}(L) \text{ pour tout } t \in [0,T] \text{, } \overline{t} \in \Pi^{+}(t) \} \textit{,} $$ on voit que son complémentaire est de mesure en $O(T^{-998})$ quitte à prendre $R$ assez grand. En effet, à $t,\overline{t}$ fixé, pour chaque courbe de $F_{t}$, la longueur totale des ensembles ne vérifiant pas la condition voulue est au plus de $2 L_{\overline{t}}$. \\
\\
Sur l'ensemble $E=E_{1} \cap E_{2} \cap E_{3}$, on voit que (h4) jusqu'à (h6) sont vérifiées. De plus, $\mu_{2}(E^{c}) = O(T^{-100})$. Par ailleurs, (h1) jusqu'à (h3) sont vérifiées. D'où le résultat voulu.
\end{proof}
\begin{proof}[Démonstration du théorème $\ref{thm21}$]
En utilisant le lemme $\ref{lemme42}$ conjugué à la proposition $\ref{prop14}$, la proposition $\ref{prop15}$ et la proposition $\ref{prop40}$ qui nous permet d'appliquer le théorème $\ref{thm9}$ on obtient finalement la validité du théorème $\ref{thm21}$.
\end{proof}
Note : le résultat reste valable si l'on suppose $L$ distribué selon $\tilde{\mu}_{2}$ car le flot géodésique mélange à vitesse exponentielle.

\section{Étude asymptotique de $\mathcal{R}$}
\subsection{Premières considérations sur l'erreur $\mathcal{R}$}

Le problème qui nous intéresse maintenant est celui de la convergence en loi de $\frac{\mathcal{R}(t,P,L,X)}{\log(t)}$ où $X \in \mathbb{R}^{2}$, $t >0$, $P$ est un parallélogramme d'intérieur non vide et $L \in \mathscr{S}_{2}$. On peut supposer pour simplifier, étant donné que l'on travaille avec $L$ distribué selon $\tilde{\mu}_{2}$ et $X$ selon $\tilde{\lambda}_{2}$, que le parallélogramme $P$ d'intérieur non vide est un rectangle dont les côtés sont parallèles aux axes de coordonnées (cela peut se faire en déformant $P$ via les matrices de rotation et via les matrices de la forme $\begin{pmatrix} 1 & \tau \\ 0 & 1 \end{pmatrix}$) et, quitte à déformer le rectangle $P$ via les matrices de la forme $\begin{pmatrix} \lambda & 0 \\ 0 & \frac{1}{\lambda} \end{pmatrix}$, on peut supposer que $P$ est en réalité un carré dont les côtés sont parallèles aux axes de coordonnées et enfin, quitte à translater $P$, on peut supposer que le centre du carré $P$ est $(0,0)$. On notera qu'on a utilisé tous les degrés de liberté du problème dont on disposait a priori. \\
Cette simplification faite, appelons $A_{1}$ le sommet inférieur droit de $P$ de coordonnées $(a,-a)$, $A_{2}$ le sommet supérieur droit de $P$ de coordonnées $(a,a)$, $A_{3}$ le sommet supérieur gauche de $P$ de coordonnées $(-a,a)$ et $A_{4}$ le sommet inférieur gauche de $P$ de coordonnées $(-a,-a)$ avec $a > 0$. \\
Nous avons maintenant besoin d'introduire quelques notations pour rappeler un résultat de $\cite{Skriganov}$ duquel on partira pour établir la convergence en loi désirée. Appelons : 
\begin{itemize}
\item  $\tau = \frac{\log(t)^{\frac{1}{4}}}{t}$
\item  $t^{\pm} = t \pm \beta \tau$ avec $\beta \in \mathbb{R}$ approprié
\item  $\rho > 0$ tel que $ \tau = \rho^{- \theta}$ où $\theta \in ]0,1[$ avec $\theta$ aussi proche de $1$ que l'on veut fixé
\item  $\lambda(l) = \frac{l_{1}^{2}}{l_{1}^{2}+ l_{2}^{2}}$
\item  $\omega_{1}$ est une fonction à support compact $\subset B_{f}(0,\frac{1}{4})$, de classe $C^{\infty}$, positive, plus petite que $1$ et tel que $\omega_{1} = 1$ sur $B_{B_{f}(0,\frac{1}{8})}$, à symétrie sphérique
\item  $\omega_{2}$ est la transformée de Fourier de $\frac{\omega_{1}}{\int \omega_{1}}$ (et est donc une fonction à décroissance rapide et à symétrie sphérique)
\end{itemize}
Avec ces notations, on peut désormais poser : 
$$\tilde{S}_{1}^{\pm}(\tau,t,P^{0},R_{1},X,L) = \frac{1}{\log(t)} \frac{1}{(2 \pi i)^{2}} \sum_{l \in L^{\perp} -\{ 0 \} }  \frac{1}{(R_{1} l)_{1}(R_{1} l)_{2}} \lambda(R_{1} l) \omega_{2}(\tau  l) \omega_{1}(\rho^{-1} l )e^{2 i \pi <l, t^{\pm} P^{0} + X > }$$ où $R_{1} \in SO_{2}(\mathbb{R})$ et où $R_{1} L^{\perp}$ est supposée faiblement admissible (voir définition $\ref{def3}$). \\
Enfin, appelons : 
 $R= \begin{pmatrix} 0 & 1 \\ -1 & 0 \end{pmatrix} \textit{, } I_{2} = \begin{pmatrix} 1 & 0 \\ 0 & 1 \end{pmatrix} \text{ et} $
 
$$ \tilde{S}^{\pm}(\tau,t,X,L) = \sum_{i=1}^{4} (-1)^{i} (\tilde{S}_{1}^{\pm}(\tau,t,A_{i},I_{2},X,L)- \tilde{S}_{1}^{\pm}(\tau,t,A_{i},R,X,L)) \textit{.}$$ 
Les deux sommes $\tilde{S}^{\pm}$ correspondent en fait à deux sommes sur ce qui s'appelle l'ensemble des drapeaux du carré $P$, qui est composé dans ce cas particulier de 8 éléments (voir $\cite{Skriganov}$ pour plus de détails).\\
On dispose du résultat suivant grâce à $\cite{Skriganov}$: 
\begin{prop}
\label{prop41}
Pour tout $\alpha > 0$, pour tout $t$ assez grand, 
\begin{equation*}
\mathbb{P}\left( \alpha + \tilde{S}^{+}(\tau,t,X,L) \geqslant \frac{\mathcal{R}(t,P,L,X)}{\log(t)} \geqslant \tilde{S}^{-}(\tau,t,X,L) - \alpha \right) \geqslant 1 - \alpha \textit{.}
\end{equation*}
\end{prop}

Ainsi, pour montrer le théorème $\ref{thm100}$, il nous suffit de montrer que $\tilde{S}^{-}$ et $\tilde{S}^{+}$ ont une distribution asymptotique commune qui est une loi de Cauchy centrée.\\
Par ailleurs, l'étude de la convergence asymptotique de $\tilde{S}^{-}$ pouvant se conduire de manière tout à fait analogue à celle de $\tilde{S}^{+}$, la limite trouvée sera indépendante du signe $\pm$. On se concentre donc sur l'étude $\tilde{S}^{+}$ et on la renomme $\tilde{S}$. On signalera tout au long de la preuve ce qui se serait passé si on avait traité $\tilde{S}^{-}$ à la place de $\tilde{S}^{-}$. \\
\subsection{Étude d'un seul terme}
Dans un but pédagogique, on va simplifier la situation en se concentrant sur l'étude d'un seul terme de $\tilde{S}$ à savoir : 
\begin{equation}
\label{eq144}
\tilde{S}_{2}(\tau,t,X,L) = \tilde{S}_{1}^{+}(\tau,t,A_{2},I_{2},X,L) \textit{.}
\end{equation}
Le but de cette sous-section est de montrer le résultat suivant : 
\begin{theorem}
\label{thm23}
Quand $L \in \mathscr{S}_{2}$ est distribué selon $\tilde{\mu}_{2}$, quand $X$ est distribué selon $\tilde{\lambda}_{2}$, $\tilde{S}_{2}$ converge en loi vers une loi de Cauchy centrée.
\end{theorem}
Une précision toutefois : comme on l'a vu dans la section précédente, on peut supposer, et on le fera, qu'il existe $D > 0$ tel que pour tout $l \in L-\{ 0 \}$, \begin{equation}
\label{eq146} |l_{1} l_{2} | \geqslant D | \log(\lVert l \rVert_{2})|^{-1 - \beta} 
\end{equation} (avec $\beta > 0$) et que \begin{equation}
\label{eq147}
\lVert L \rVert_{1} \geqslant \sqrt{\epsilon} \textit{.}
\end{equation}
Par ailleurs on remplacera $L^{\perp}$ par $L$ dans l'étude (ce qui ne pose aucun soucis car $L \longmapsto L^{\perp}$ est une bijection, d'inverse elle-même, continue). \\
\\
Heuristiquement, le terme au dénominateur $|Num(l)|$ est en fait relié à un certain $\lVert \delta_{k} L\rVert$, tandis qu'au numérateur on a un terme qui va se comporter à l'infini comme un $\theta_{k}$. Comme on a réussi à prouver le théorème $\ref{thm21}$, on devrait réussir à prouver ce dernier théorème. \\
\\
Introduisons : 
\begin{equation}
\label{eq161}
 \tilde{S}_{9}(\tau,t,\epsilon,X,L) =  \frac{2}{(2 \pi i)^{2}} \sum_{h \in \tilde{I}_{-} (t)}  \frac{\Gamma_{h,t}}{\Xi_{h,t}}   
\end{equation}
où 
\begin{equation}
\label{eq211}
\Xi_{h,t} = \text{Num}(l(L,h)) \log(t) \textit{ , }
\end{equation}
\begin{equation}
\label{eq212}
\Gamma_{h,t}= \sum_{k \in \mathbb{N}-\{ 0 \}} \frac{1}{k^{2}} \cos(2 k \pi <l(L,h), t^{+}A_{2} + X > ) \textit{ et }
\end{equation}
\begin{align} \tilde{I}_{-} (t)= \{ h \in [- \lceil \log(t) \rceil, 0]  \text{ } | \text{ } \lVert \delta_{h} L \rVert^{2} < 2 \cosh(1)|\text{Num}(e(\delta_{h}L))| & \text{ et } \\
 |\text{Num}(e(\delta_{h}L))|  \leqslant \frac{1}{\epsilon \log(t)} \} \nonumber
\end{align}
avec $e(L)$ l'unique vecteur d'un réseau $L$ de première coordonnée strictement positive et tel que $\lVert e(L) \rVert = \lVert L \rVert$ ($e(L)$ est défini pour presque tout réseau $L$)
et $\epsilon > 0$ et où $l(L,h) = \delta_{h}^{-1} e(\delta_{h} L)$. \\
Avec ces notations, on va procéder de la manière suivante pour montrer le théorème $\ref{thm23}$. D'abord,
dans la prochaine sous-sous-section, on va montrer la proposition suivante qui permet de passer de l'étude asymptotique de $\tilde{S}_{2}$ à celle de $\tilde{S}_{9}$ 
\begin{prop}
\label{prop42}
Pour tout $\alpha > 0$, pour tout $\epsilon > 0$ assez petit, pour tout $t$ assez grand, 
$$\mathbb{P}(|\tilde{S}_{2} - \tilde{S}_{9}| \geqslant \alpha) \leqslant \alpha \textit{.}$$ 
\end{prop}
Cette proposition sera l'objet de la sous-sous-section suivante et s'obtient après un certain nombre de réductions successives, qui nécessitent des estimations d'intégrales, d'utiliser des outils de l'analyse de Fourier et de l'analyse classique, de faire de la géométrie des réseaux et d'utiliser des outils dynamiques sur l'espace des réseaux muni du flot géodésique (théorème ergodique par exemple). \\
Une troisième sous-section sera réservée à la mise en place d'un cadre permettant de ramener l'étude asymptotique des $\{\Xi_{h,t} \} $ et des $\{ \Gamma_{h,t} \}$ à la vérification des hypothèses (h1) jusqu'à (h8) du théorème $\ref{thm13}$. La sous-section qui suivra sera dédiée à la vérification de ces hypothèses ce qui nous assurera que les $\{\Xi_{h,t} \} $ sont asymptotiquement un processus de Poisson et les $\{ \Gamma_{h,t} \}$ sont asymptotiquement des variables aléatoires réelles indépendantes, identiquement distribuées, symétriques, à support compact et indépendantes de $\{\Xi_{h,t} \} $. Cela nous permettra, $\textit{in fine}$, via le lemme $\ref{lemme41}$ d'obtenir le théorème $\ref{thm23}$. \\ 
Dans la cinquième et dernière sous-section de cette section, fort des résultats des sous-sections précédentes, nous reviendrons sur l'étude de $\tilde{S}_{1}$ et prouverons le théorème $\ref{thm100}$. 
\subsubsection{Preuve de la proposition $\ref{prop42}$ - « Oubli » de la fonction $\omega_{1}$}
On introduit la quantité suivante : 
$$\tilde{S}_{3}(\tau,t,X,L) = \frac{1}{\log(t)} \frac{1}{(2 \pi i)^{2}} \sum_{l \in J_{3}(L, \alpha,t) }  \frac{1}{l_{1}l_{2}} \lambda(l) \omega_{2}(\tau  l) e^{2 i \pi <l, t^{+}A_{2} + X > } $$
où  $$J_{3}(L, \alpha,t) = \{ l \in L-\{ 0 \} \text{ } | \text{ }\lVert l \rVert \leqslant \frac{1}{8} t^{\alpha} \} $$ avec $1 < \alpha < \frac{1}{\theta}$ est quelconque et on le fera tendre plus tard vers 1. \\
Le but de cette sous-sous-section est alors de prouver la proposition suivante : 
\begin{prop}
\label{prop18}
$|\tilde{S}_{2} -\tilde{S}_{3}|$ converge en probabilités vers $0$.
\end{prop}
Heuristiquement, cette proposition signifie qu'on remplace $\omega_{1}$ par sa seule valeur en $0$, qui vaut $1$, et qu'on coupe les termes en $l$ dont la norme est trop importante, ce que fait la fonction $\omega_{1}$ lorsque le terme $l$ en norme est plus grand que $\frac{1}{4} t^{\frac{1}{\theta}}$. C'est donc assez naturel de prouver cela.
Pour démontrer cette proposition, on a besoin du lemme suivant, énoncé en dimension $d \geqslant 2$.
\begin{lemma}
\label{lemme31}
Pour tout $\epsilon_{1} > 0$, pour tout $\beta > 0$, pour tout $\gamma > 0$,
$$\int_{\substack{ l \in \mathbb{R}^{d} \\ 1 < \lVert l \rVert \leqslant t^{1+\epsilon_{1}} \\ |l_{1} \cdots l_{d} | \geqslant D |\log(\lVert l \rVert)|^{1-d-\beta}}} \frac{1}{|l_{1} \cdots l_{d}|} dl_{1} \cdots dl_{d} = O(\log(t)^{d}) \textit{.}$$
\end{lemma}
\begin{proof} Par continuité et par symétrie, cela revient à démontrer que : 
$$\int_{\substack{ l \in \mathbb{R}^{d} \\ 1 < \lVert l \rVert \leqslant t^{1+\epsilon_{1}} \\ l_{1} \cdots l_{d}  \geqslant D \log(\lVert l \rVert)^{1-d-\beta} \\ 0 < l_{1} \leqslant l_{2} \leqslant \cdots \leqslant l_{d} }} \frac{1}{l_{1} \cdots l_{d}} dl_{1} \cdots dl_{d} = O(\log(t)^{d}) \textit{.}$$
Appelons $I$ cette dernière intégrale. \\
Le système d'inégalités portant sur $l$ implique que :\\
pour tout $i \in \{1, \cdots, d \}$, $l_{i} \leqslant t^{1+\epsilon_{1}}$ et $l_{i}^{i} (t^{1+\epsilon_{1}})^{d-i} \geqslant D \log( \lVert l \rVert)^{1-d-\beta} \geqslant C \log(t^{1+\epsilon_{1}})^{1-d-\beta}$. \\
Ainsi, on a, pour tout $i$ : $$t^{1+\epsilon_{1}} \geqslant l_{i} \geqslant C_{i} \frac{\log(t)^{\frac{1-d-\beta}{i}}}{t^{(1+\epsilon_{1})\frac{d-i}{i}}} \textit{.}$$
D'où, pour tout $t$ assez grand, par le théorème de Fubini, on a : $$I \leqslant \prod_{i=1}^{d} \int_{t^{1+\epsilon_{1}} \geqslant l_{i} \geqslant C_{i} \frac{\log(t)^{\frac{1-d-\beta}{i}}}{t^{(1+\epsilon_{1})\frac{d-i}{i}}}} \frac{1}{l_{i}}dl_{i} \textit{.}$$
Or pour tout $i$, $\int_{t^{1+\epsilon_{1}} \geqslant l_{i} \geqslant C_{i} \frac{\log(t)^{\frac{1-d-\beta}{i}}}{t^{(1+\epsilon_{1})\frac{d-i}{i}}}} \frac{1}{l_{i}}dl_{i} = O(\log(t)) \textit{.}$
D'où le résultat voulu.
\end{proof}
Démontrons maintenant la proposition $\ref{prop18}$.
\begin{proof}[Démonstration de la proposition $\ref{prop18}$]
Vu les propriétés sur $\omega_{1} $ et vu la définition de $\rho$, on a (en sous-entendant les diverses variables) : 
\begin{equation}
\label{eq148}
|\tilde{S}_{2} -\tilde{S}_{3}| \leqslant \frac{M}{\log(t)} \sum_{\substack{ l \in L-\{0 \} \\  \frac{t^{\frac{1}{\theta}}}{4} \geqslant \lVert l \rVert \geqslant \frac{1}{8} t^{\alpha}}} \frac{1}{|\text{Num}(l)|} | \omega_{2} (\tau l) | \textit{.} 
\end{equation}
Or $\omega_{2}$ est à décroissance rapide et pour $l$ comme dans la somme, on a $\lVert \tau l \rVert \geqslant \tau t^{\alpha} \geqslant t^{\alpha-1}$ et $\alpha - 1 > 0$. \\
D'où de ($\ref{eq148}$) on tire que : 
\begin{equation}
\label{eq149}
|\tilde{S}_{2} -\tilde{S}_{3}| \leqslant \frac{M}{\log(t) t^{17(\alpha -1)}} \sum_{\substack{ l \in L-\{0 \} \\  \frac{t^{\frac{1}{\theta}}}{4} \geqslant \lVert l \rVert \geqslant \frac{1}{8} t^{\alpha}}} \frac{1}{|\text{Num}(l)|}  \textit{.} 
\end{equation}
D'où le résultat voulu grâce au lemme $\ref{lemme24}$ et grâce à ($\ref{eq146}$) qui nous permet d'appliquer le lemme $\ref{lemme31}$.
\end{proof}
On est donc ramené à l'étude de $\tilde{S}_{3}$ \footnote{ On voit dans la preuve de la proposition $\ref{prop18}$ que l'on a majoré le module des termes $e^{2 i \pi <l, t^{+}A_{2} + X > }$ par $1$. Donc cette preuve est aussi valable en ayant $t^{-}$ à la place de $t^{+}$. }. 
\subsubsection{Preuve de la proposition $\ref{prop42}$ - Centrage sur les vecteurs premiers}
On introduit : 
\begin{equation}
\label{eq150} \tilde{S}_{4}(\tau,t,X,L) = \frac{2}{\log(t)} \frac{1}{(2 \pi i)^{2}} \sum_{l \in J_{4}(L,\alpha,t)}  \frac{1}{l_{1}l_{2}} \lambda(l) \sum_{k \in \mathbb{N}-\{ 0 \}} \frac{\omega_{2}(k \tau  l)}{k^{2}} \cos(2 k \pi <l, t^{+} A_{2} + X > ) 
\end{equation}
où $$J_{4}(L, \alpha,t) = \{ l \in L-\{ 0 \} \text{ } | \text{ }\lVert l \rVert \leqslant \frac{1}{8} t^{\alpha} \text{, } l \text{ premier} \text{, } l_{1} > 0 \} \textit{.} $$
Le but de cette sous-sous-section est de prouver la proposition suivante :

\begin{prop}
\label{prop19}
$\tilde{S}_{4}-\tilde{S}_{3}$ tend en probabilités vers $0$.
\end{prop}
Cette proposition nous dit d'une part, qu'on peut regrouper tous les termes qui sont multiples d'un vecteur premier $l$, et, d'autre part, qu'on peut se contenter de considérer les vecteurs premiers $l$ qui voient leur première coordonnée être strictement positive (via un argument de parité et quitte à éliminer un ensemble de réseaux négligeable). 
\begin{proof}
Tout d'abord, en utilisant l'invariance par la transformation $l \longmapsto -l$ de $\omega_{2}(l)$, de $\text{Num}(l)$ (cela dépend en général de la dimension $d$) et de $\lambda(l)$, il vient : 
\begin{equation}
\label{eq151}
\tilde{S}_{3}(\tau,t,X,L) = \frac{2}{\log(t)} \frac{1}{(2 \pi i)^{2}} \sum_{\substack{ l \in L \\ \lVert l \rVert \leqslant \frac{1}{8} t^{\alpha} \\ l_{1} > 0} }  \frac{1}{l_{1}l_{2}} \lambda(l) \omega_{2}(\tau  l) \cos (2  \pi <l, t^{+} A_{2} + X > ) \textit{.}
\end{equation}
De ($\ref{eq150}$) et de ($\ref{eq151}$), on obtient : 
$$|\tilde{S}_{3}- \tilde{S}_{4}| \leqslant \frac{M}{\log(t)} \sum_{\substack{ l \in L \\ \lVert l \rVert \geqslant \frac{1}{8} t^{\alpha} \\ l_{1} > 0} } \frac{|\omega_{2}(\tau l)|}{| \text{Num} (l)|} \textit{.} $$
On conclut comme pour la proposition $\ref{prop18}$.
\end{proof}
On est donc ramenés à l'étude de $\tilde{S}_{4}$ \footnote{ On voit que dans la preuve de la proposition $\ref{prop19}$, on a essentiellement utilisé un argument de parité et que l'on a majoré les termes $|\cos (2  \pi <l, t^{+} A_{2} + X > )|$ par 1. Cette preuve est donc encore valable en considérant $t^{-}$ à la place de $t^{+}$.}.
\subsubsection{Preuve de la proposition $\ref{prop42}$ - Réduction de la somme aux termes $l$ tels que $Num(l)$ est petit}
Soit $\epsilon > 0$. On introduit : 
\begin{equation}
\label{eq152}
 \tilde{S}_{5}(\tau,t,\epsilon,X,L) = \frac{2}{\log(t)} \frac{1}{(2 \pi i)^{2}} \sum_{l \in J_{5}(L, \alpha, \epsilon, t)}  \frac{1}{l_{1}l_{2}} \lambda(l) \sum_{k \in \mathbb{N}-\{ 0 \}} \frac{\omega_{2}(k \tau  l)}{k^{2}} \cos(2 k \pi <l, t^{+}A_{2} + X > ) 
\end{equation}
où $$J_{5}(L, \alpha, \epsilon, t) = \{ l \in L-\{ 0 \} \text{ } | \text{ }\lVert l \rVert \leqslant \frac{1}{8} t^{\alpha} \text{, } l \text{ premier} \text{, } l_{1} > 0 \text{, }  |Num(l)| \leqslant \frac{1}{ \log(t) \epsilon} \} \textit{.} $$
Le but de cette sous-sous-section est alors de démontrer la proposition suivante :  
\begin{prop}
\label{prop20}
Pour tout $\kappa > 0$, pour tout $\epsilon > 0$ assez petit, pour tout $t$ assez grand, on a 
$$\mathbb{P}(| \tilde{S}_{4}(\tau,t,X,L) - \tilde{S}_{5}(\tau,t,\epsilon,X,L)| \geqslant \kappa) \leqslant \kappa \textit{.} $$ 
\end{prop}
Cette proposition nous dit essentiellement qu'on peut se ramener aux termes $l \in J_{4}(L,\alpha,t)$ qui voient leur $\textit{Num}$ être petit (ce qui est naturel puisque comme la quantité inverse intervient dans la somme, cela correspond à des termes qui contribuent beaucoup à la somme). \\
La démonstration s'appuie sur le lemme suivant énoncé en dimension $d \geqslant 2$ : 
\begin{lemma}
\label{lemme32}
Il existe $M > 0$ tel que pour tout $\epsilon_{1} > 0$,
$$\int_{\substack{ l \in \mathbb{R}^{d} \\ 1 \leqslant \lVert l \rVert \leqslant t^{1+\epsilon_{1}} \\ |\text{Num}(l)| \geqslant \frac{1}{\epsilon \log(t)^{d-1}}}} \frac{1}{l_{1}^{2} \cdots l_{d}^{2}} dl_{1} \cdots dl_{d} \leqslant M \log(t)^{2d-2} \epsilon \textit{.}$$
\end{lemma}
\begin{proof}
Par symétrie, on se ramène au domaine où les $l_{i}$ sont strictement positifs. \\
On pose $ \phi : (l_{1}, \cdots , l_{d}) \longmapsto (l_{1}, l_{1}l_{2}, \cdots, l_{1} \cdots l_{d})$. Alors $\phi$ réalise un $C^{\infty}$-difféomorphisme de $(\mathbb{R}_{+}-\{ 0 \})^{d}$ sur lui-même et la matrice jacobienne évaluée en $(l_{1}, \cdots, l_{d})$, à savoir $\text{Jac}(l_{1}, \cdots , l_{d})$, vérifie : 
\begin{equation}
\label{eq153}
\text{Jac}(l_{1}, \cdots , l_{d}) = \prod_{i=1}^{d-1} \phi_{i}(l_{1},\cdots,l_{d}) \textit{.}
\end{equation}
Par ailleurs, puisque $l$ appartient au domaine de l'intégrale dont il est question dans le lemme $\ref{lemme32}$, on voit que pour tout $i \in \{1, \cdots, d-1 \}$, $b_{i} = \frac{1}{\epsilon \log(t)^{d-1} t^{(1+\epsilon_{1})(n-i)}} \leqslant \phi_{i}(l) \leqslant t^{(1+\epsilon_{1})i}=h_{i}$ et $\phi_{d}(l) \geqslant \frac{1}{\epsilon \log(t)^{d-1}}=b_{d}$. \\
D'où de ($\ref{eq153}$), on tire : 
\begin{equation}
\label{eq154}
\int_{\substack{\forall 1 \leqslant i \leqslant d \textit{, } l_{i} > 0 \ \\ 1 \leqslant \lVert l \rVert_{2} \leqslant t^{1+\epsilon_{1}} \\ \text{Num}(l) \geqslant \frac{1}{\epsilon \log(t)^{d-1}}}} \frac{1}{l_{1}^{2} \cdots l_{d}^{2}} dl_{1} \cdots dl_{d} \leqslant \int_{ \substack{ \forall 1 \leqslant i \leqslant d-1 \textit{, } b_{i} \leqslant u_{i} \leqslant h_{i} \\ b_{d} \leqslant u_{d} }} \frac{1}{u_{d}^{2} \prod_{i=1}^{d-1} u_{i}} du \textit{.} 
\end{equation}
Le membre de droite se calcule et donne le résultat voulu.
\end{proof}
Passons maintenant à la démonstration de la proposition $\ref{prop20}$.
\begin{proof}[Démonstration de la proposition $\ref{prop20}$]
En appelant $\Delta$ la norme $2$ au carré de la différence entre $\tilde{S}_{4}$ et $\tilde{S}_{5}$ relativement à $X$, la formule de Parseval donne que : 
\begin{equation}
\label{eq155}
\Delta = \frac{M_{1}}{\log(t)^{2}}\sum_{\substack{l \in J_{4}(L, \alpha, t) \\ |Num(l)| > \frac{1}{\epsilon \log(t)} }}  \frac{1}{(l_{1}l_{2})^{2}} (\lambda(l))^{2} \sum_{k \in \mathbb{N}-\{ 0 \}} \frac{(\omega_{2}(k \tau  l))^{2}}{k^{4}}
\end{equation}
où $M_{1}> 0$. \\
En utilisant le fait que $\omega_{2}$ et $\lambda$ sont bornées et en intégrant selon $L$ et en utilisant le lemme $\ref{lemme24}$ puis le lemme $\ref{lemme32}$ (dans le cas où $d=2$), on obtient que pour tout $t$ assez grand : 
$$\mathbb{E}(\Delta) \leqslant \frac{M_{2}}{\log(t)^{2}} \epsilon \log(t)^{2} $$
où $M_{2} > 0$.
D'où le résultat voulu.
\end{proof}
On est donc ramené à l'étude, pour $\epsilon > 0$, de $\tilde{S}_{5}$ \footnote{ On voit que dans la preuve, si l'on remplace $t^{+}$ par $t^{-}$, celle-ci fonctionne encore : le terme $\cos(2 k \pi <l, t^{+}A_{2} + X > )$ intervient seulement au début de la preuve et on l'élimine via la formule de Parseval, qui fonctionne encore en remplaçant $t^{+}$ par $t^{-}$.}.
\subsubsection{Preuve de la proposition $\ref{prop42}$ - Passage à une somme géodésique}
Le but de cette sous-sous-section est de passer d'une somme sur $l \in L-\{ 0 \}$ à une somme sur $\delta_{h} L$ avec $h$ entier. Pour parler plus précisément, introduisons quelques notations. Pour $L \in \mathscr{S}_{2}$ faiblement admissible, soit $e(L)$ l'unique vecteur de $L$ tel que $e(L)_{1} > 0$ et $\lVert e(L) \rVert = \lVert L \rVert_{1}$ ($e(L)$ est nécessairement premier) et on se donne $\eta > \alpha - 1$, aussi proche que l'on veut de $\alpha -1$.\\
 Introduisons maintenant l'ensemble 
\begin{align*} I(t, \eta )= \{ h \in [- \lceil (1+\eta) \log(t) \rceil, \lceil (1+ \eta) \log(t) \rceil   ]  \text{ } | \text{ } \lVert \delta_{h} L \rVert^{2} < 2 \cosh(1)|\text{Num}(e(\delta_{h}L))| & \text{ et } \\
 |\text{Num}(e(\delta_{h}L))|  \leqslant \frac{1}{\epsilon \log(t)} \} 
\end{align*}
les vecteurs de $L$ de la forme
\begin{equation}
\label{eq157}
l(L,h) = \delta_{h}^{-1} e(\delta_{h}L)
\end{equation}
où $h$ parcourt $\mathbb{Z}$ et la somme
\begin{align}
\label{eq156}
 \tilde{S}_{6}(\tau,t,\epsilon,X,L) =  \frac{2}{(2 \pi i)^{2}} \sum_{h \in I(t, \eta)}  \frac{1}{\Xi_{h,t}} \lambda(l(L,h)) & \\
  (\sum_{k \in \mathbb{N}-\{ 0 \}} \frac{\omega_{2}(k \tau l(L,h))}{k^{2}} \cos(2 k \pi <l(L,h), t^{+} A_{2} + X > )) \nonumber \textit{.}
\end{align}
On rappelle que $\Xi_{h,t}$ a été défini par l'équation ($\ref{eq211}$). \\
Le but de cette sous-sous-section est de démontrer la proposition qui suit :
\begin{prop}
\label{prop21}
$\tilde{S}_{5} - \tilde{S}_{6}$ converge en probabilités vers $0$. 
\end{prop}
Cette proposition nous permet de nous ramener à l'étude de sommes géodésiques, à savoir $\tilde{S}_{6}$. L'idée principale est que les points $l$ des sommes $\tilde{S}_{5}$ sont en fait de la forme $l(L,h)$ avec $h \in I(t, \eta )$ (voir les lemmes $\ref{lemme34}$ et $\ref{lemme36}$) et que les termes de $\tilde{S}_{6}$ qui sont en plus par rapport aux sommes $\tilde{S}_{5}$ sont en quantité négligeable (ce qui est donné par les lemmes $\ref{lemme34}$, $\ref{lemme35}$ et $\ref{lemme36}$). \\
Pour démontrer cette proposition, nous avons d'abord besoin de revenir et de préciser le début de la section 7 de $\cite{Skriganov}$.\\
On pose $\mathcal{P} = ]0, \infty[ \times (\mathbb{R}-\{ 0 \} ) $. \\
Dans cette section, Skriganov introduit l'action de groupe qui à $(\delta,x) \in \Delta \times \mathcal{P} \longmapsto \delta x $. \\
Celle-ci admet pour domaine fondamentale : $$\mathcal{F}_{\Delta} = \{ \text{ } m (e^{y},e^{-y}) \text{ } | \text{ } m > 0, \text{ } y \in [-\frac{1}{2},\frac{1}{2}[ \text{ } \} \dot{\cup} \{ \text{ } m (e^{y},-e^{-y}) \text{ } | \text{ } m > 0, \text{ } y \in [-\frac{1}{2},\frac{1}{2}[ \text{ } \} \textit{.} $$ 
\begin{lemma}
\label{lemme33}
Pour tout $x=(x_{1},x_{2}) \in \mathcal{P}$, l'unique $h(x) \in \mathbb{Z}$ tel que $\delta_{h(x)} x \in \mathcal{F}_{\Delta}$ est donné par : \\ $h(x) = \frac{1}{2} \lceil  \log( \frac{|x_{2}|}{x_{1}} )  \rceil$ si $\lceil  \log( \frac{|x_{2}|}{x_{1}} )  \rceil$ est paire et sinon $h(x) = \frac{1}{2} \lfloor  \log( \frac{|x_{2}|}{x_{1}} )  \rfloor$. \\
Par ailleurs, on a :  $\delta_{h(x)}x= m (e^{y},\text{sgn}(x_{2})e^{-y})$ où $m = \sqrt{x_{1} |x_{2}|}$ et $y = h(x) - \frac{1}{2} \log( \frac{|x_{2}|}{x_{1}} ) \textit{.}$
\end{lemma}
\begin{proof}
La démonstration est élémentaire
\end{proof}
On remarque que, quitte à enlever à $\mathcal{P}$ un ensemble dénombrable de portion de droites, on peut supposer que $y \in ]-\frac{1}{2}, \frac{1}{2}[$. Cela symétrise la situation et à cause de ce fait on considèrera des réseaux $L$ qui ne touchent jamais cet ensemble de morceaux de droites. Cet ensemble constitue un ensemble mesurable de mesure pleine. \\
On dispose alors du lemme suivant : 
\begin{lemma}
\label{lemme34}
Soit $l \in L$ comme dans le domaine de la somme $S_{5}$. Alors $h(l) \in I(t, \eta)$. 
\end{lemma}
\begin{proof}
Soit $l$ comme dans l'énoncé du lemme. Alors on a $\lVert \delta_{h(l)}l \rVert^{2} = 2 \cosh(2 y) |l_{1} l_{2} |$ et donc, comme $y \in ]-\frac{1}{2}, \frac{1}{2}[$, 
\begin{equation}
\label{eq126}
2 |l_{1} l_{2} | \leqslant \lVert \delta_{h(l)}l \rVert^{2} < 2 \cosh(1) |l_{1} l_{2}| \textit{.}
\end{equation}
Ainsi, comme $|l_{1} l_{2}|$ est suffisamment petit et que $l$ est premier, pour peu que $t$ soit assez grand, $ \delta_{t(l)}l $ est un plus court vecteur du réseau $\delta_{t(l)} L$ d'après le raisonnement qui a mené au lemme $\ref{lemme25}$ et comme $(\delta_{h(l)}l)_{1} > 0$, on a 
\begin{equation}
\label{eq158}
e(\delta_{h(l)}L) = \delta_{h(l)}l \textit{.}
\end{equation}
De plus, comme pour tout $x$, pour tout $k \in \mathbb{Z}$, $\text{Num}(\delta_{k}x) = \text{Num}(x)$, on a, d'après ($\ref{eq158}$) : 
\begin{equation}
\label{eq158}
Num(e(\delta_{h(l)}L)) \leqslant \frac{1}{\epsilon \log(t)} \textit{.}
\end{equation}
Enfin, on a : $ h(l) = \lceil \frac{1}{2} \log( \frac{|l_{2}|}{l_{1}} )  \rceil $ ou $h(l) = \lfloor \frac{1}{2} \log(\frac{|l_{2}|}{l_{1}}) \rfloor$
et $ \frac{|l_{2}|}{l_{1}} = \frac{l_{1} |l_{2}|}{l_{1}^{2}} = \frac{ l_{2}^{2}}{l_{1}|l_{2}|} \textit{.} $ \\
Or, d'après ($\ref{eq146}$) : $$D \alpha^{-1 - \beta} \log(t)^{-1 - \beta} \leqslant  D \log(\lVert l \rVert)^{-1 - \beta} \leqslant l_{1} |l_{2}| \leqslant \frac{1}{\epsilon \log(t)} \textit{.}$$
D'où, comme $\lVert l \rVert \leqslant t^{\alpha}$, il vient que pour $t$ assez grand (la grandeur dépendant de $D$, $\alpha$, $\beta$ et $\epsilon$), $$- \lceil (1+\eta) \log(t) \rceil \leqslant t(l) \leqslant \lceil (1+\eta) \log(t) \rceil $$
où $\eta  >  \alpha - 1 $ et est aussi proche que l'on veut de $\alpha - 1$. D'où le résultat voulu.
\end{proof}
Le lemme suivant assure une forme de réciproque : 
\begin{lemma}
\label{lemme35}
Pour $h \in I(t,\eta)$, $$l(L,h) \in V(L,t) = \{ l \in L | l_{1} > 0 \text{, } l \text{ est premier, } \lVert l \rVert \leqslant C t^{1 + \eta} \text{, } l_{1}|l_{2}| \leqslant \frac{1}{\epsilon \log(t)} \}$$ où $C > 0$ est une constante.
\end{lemma}
\begin{proof}
On pose $l = l(L,h)= \delta_{h}^{-1} e(\delta_{h} L)$. Il est clair que $l_{1} > 0$ et que $l$ est premier par définition de $e$.\\
Par ailleurs, le théorème de Minkowski nous donne qu'il existe $C>0$ tel que : $$ \lVert e(\delta_{h} L)\rVert  \leqslant C $$ d'où l'on tire qu'il existe $C > 0$ tel que $$\lVert L \rVert_{2} \leqslant C t^{1+ \eta} \textit{.}$$ 
Enfin, on note que $\text{Num}(l) = \text{Num}(e(\delta_{h}L))$ et ainsi $l(L,h) \in  V(L,t).$
\end{proof}
On voit qu'avec $l(L,\cdot)$ on ne retombe pas exactement dans le domaine de $S_{5}$ a priori (mais on en tombe pas trop loin). Mais $l(L,\cdot)$ joue quand même le rôle de réciproque de $t(l)$ : 
\begin{lemma}
\label{lemme36}
Soit $k \in L$ premier tel que $|Num(k)| \leqslant \frac{1}{\epsilon \log(t)}$. Pour peu que $t$ soit assez grand, on a $l(L,t(k)) = k$. \\
Réciproquement, soit $h \in \mathbb{Z}$ tel que $|\text{Num}(e(\delta_{h}L))| \leqslant \frac{1}{\epsilon \log(t)}$ et $ \lVert \delta_{h}L \rVert^{2} < 2 \cosh(1) |\text{Num}(e(\delta_{h}L))|$. Pour peu que $t$ soit assez grand, $t(l(L,h))= h$.
\end{lemma}
Remarquons que ce lemme implique que les $h \in I(\eta,t)$ sont des minima locaux (au sens où on l'a défini dans la section précédente). 
\begin{proof}
Occupons-nous de la première partie du lemme et prenons $t > 2 \cosh(1)$ et $k$ comme dans l'énoncé. Alors d'après ($\ref{eq126}$), $e(\delta_{t(k)}L) = \delta_{t(k)}k$. Ainsi, on a : 
$$l(L,t(k)) = \delta_{t(k)}^{-1} e(\delta_{t(k)}L) = k \textit{.}$$
Occupons-nous maintenant de la seconde partie du lemme et prenons à nouveau $t > 2 \cosh(1)$ et $h$ comme dans l'énoncé. On rappelle que $t(l(L,h))$ est l'unique $p \in \mathbb{Z}$ tel que $$\delta_{p} l(L,h) \in \mathcal{F}_{\Delta} \textit{.} $$
Or, on a : $l(L,h) = \delta_{h}^{-1} e(\delta_{h}L)$ et donc $\delta_{h} l(L,h)=e(\delta_{h}L)$. \\
Or, on a : $$ ||\delta_{h}L||^{2} < 2 \cosh(1) |\text{Num}(e(\delta_{h}L))| \textit{.}$$
Comme $e(\delta_{h}L)$ s'écrit : $e(\delta_{h}L) = \sqrt{|\text{Num}(e(\delta_{h}L))|}(e^{y},\pm e^{-y})|$ où $y \in \mathbb{R}$, l'inéquation qui précède implique que : 
$$2 \cosh(2y) < 2 \cosh(1) \textit{.}$$
D'où $y \in ]-\frac{1}{2}, \frac{1}{2}[$ et $e(\delta_{h}L) \in \mathcal{F}_{\Delta}$.
\end{proof}

Nous pouvons maintenant démontrer la proposition $\ref{prop21}$. 
\begin{proof}[Démonstration de la proposition $\ref{prop21}$]
On a, d'après les lemmes $\ref{lemme33}$, $\ref{lemme34}$, $\ref{lemme35}$ et $\ref{lemme36}$, pour tout $t$ assez grand : 
$$|(\tilde{S}_{5} - \tilde{S}_{6})(\tau,t,\epsilon,X,L)| \leqslant \frac{M}{\log(t)} \sum_{l \in H(L,\alpha,\epsilon,\eta,t)}  |\lambda(l)| \frac{1}{|l_{1} l_{2}|} \sum_{k \in \mathbb{N}-\{ 0 \}} \frac{|\omega_{2}(k \tau  l)|}{k^{2}} |\cos(2 k \pi <l, t^{+} A_{2} + X > )| $$
où $$ H(L,\alpha,\epsilon,\eta,t) = \{ l \in L-\{ 0 \} \text{ } | \text{ } \frac{1}{8} t^{\alpha} < \lVert l \rVert \leqslant C t^{1 + \eta } \text{, } l \text{ premier} \text{, } l_{1} > 0 \text{, }  |Num(l)| \leqslant \frac{1}{ \log(t) \epsilon} \} \textit{.} $$
On conclut comme on a conclu la démonstration de la proposition $\ref{prop19}$ \footnote{À la fin de la preuve, on majore le terme en $|\cos|$ par $1$. La preuve reste donc encore valable en remplaçant $t^{+}$ par $t^{-}$.}.
\end{proof}

\subsubsection{Preuve de la proposition $\ref{prop42}$ - Réduction du nombre de termes de la somme}
On introduit $\tilde{I}(t) =  I(t, \eta) \cap \{ h \in \mathbb{Z} \text { } | \text{ } |h| \leqslant \lceil \log(t) \rceil \}$ puis
\begin{align}
\label{eq159}
 \tilde{S}_{7}(\tau,t,\epsilon,X,L) =  \frac{2}{(2 \pi i)^{2}} \sum_{h \in \tilde{I}(t) }  \frac{1}{\Xi_{h,t}} \lambda(l(L,h)) & \\
  (\sum_{k \in \mathbb{N}-\{ 0 \}} \frac{\omega_{2}(k \tau l(L,h))}{k^{2}} \cos(2 k \pi <l(L,h), t^{+} A_{2} + X > )) \nonumber \textit{.}
\end{align}
\begin{prop} 
\label{prop22}
$\mathbb{P}(\tilde{S}_{7}-\tilde{S}_{6} \neq 0) \leqslant M \frac{\eta}{\epsilon}$
\end{prop}
Cette proposition dit, $\textit{grosso modo}$, que le facteur $\eta$ ne touche aucun rôle dans le problème que l'on étudie (ce qui se comprend bien, c'est une quantité qui n'est pas intrinsèque au problème) et qu'on peut se ramener entre ce qui passe entre $- \lceil \log(t) \rceil$ et $\lceil \log(t) \rceil$. 
\begin{proof}
On a, pour tout $h \in I(t,\eta)$, $\lVert \delta_{h}L \rVert^{2} < 2 \cosh(1) | \text{Num}(e(\delta_{h}L))|$ et $$| \text{Num}(e(\delta_{h}L))| \leqslant \frac{1}{\epsilon \log(t)} \textit{.}$$ En utilisant cette remarque, combinée au lemme $\ref{lemme1}$, on a  : 
$$\mathbb{P}(\tilde{S}_{7}-\tilde{S}_{6} \neq 0) \leqslant \frac{M  \eta \log(t)}{\epsilon \log(t)} \textit{.}$$
D'où le résultat voulu.
\end{proof}
Comme on peut prendre le paramètre $\eta$ aussi petit que l'on veut (en prenant $\alpha$ aussi proche de $1$ que l'on veut), on est donc ramené à l'étude de $\tilde{S}_{7}$.
\subsubsection{Preuve de la proposition $\ref{prop42}$ - Élimination des termes "proches" de l'axe des ordonnées}
On introduit $\tilde{I}_{-} (t) = \tilde{I}(t) \cap \mathbb{R}_{-}$ et les deux quantités suivantes : 
\begin{align}
\label{eq160}
 \tilde{S}_{8}(\tau,t,\epsilon,X,L) =  \frac{2}{(2 \pi i)^{2}} \sum_{h \in \tilde{I}_{-} (t)}  \frac{1}{\Xi_{h,t}}  & \\
  (\sum_{k \in \mathbb{N}-\{ 0 \}} \frac{\omega_{2}(k \tau l(L,h))}{k^{2}} \cos(2 k \pi <l(L,h), t^{+} A_{2} + X > )) \nonumber \textit{.}
\end{align}
Le but de cette sous-sous-section est de démontrer la proposition suivante 
\begin{prop}
\label{prop23}
$\tilde{S}_{7}-\tilde{S}_{8}$ converge vers $0$ en probabilités.
\end{prop}
Le but de cette proposition est de retirer les termes $h > 0$ et de simplifier le terme $\lambda$. En réalité, les termes $h$ en question vont devenir grand au fur et à mesure que $t$ grandit et cela correspondra à des termes $l$ tel que $l_{1}l_{2}$ sera petit et $l_{2}$ grand (en valeur absolue), c'est-à-dire « proches » de l'axe des ordonnées et le terme en $\lambda$ sera donc proche de $0$. Quant aux termes qui correspondent à $h < 0$, ceux-ci vont devenir de plus en plus petits au fur et à mesure que $t$ grandit et cela correspondra à des termes $l$ tels que $l_{1}l_{2}$ sera petit et $l_{1}$ grand (en valeur absolue), c'est-à-dire « proches » de l'axe des abscisses et le terme en $\lambda$ sera donc proche de $1$. 
\begin{proof}
Supposons que l'on se donne $h \in \tilde{I}(t)$ tel que $h \geqslant 0$. On omet $L$ et $h$ dans le $l(L,h)$ par soucis de simplification et on pose $l(L,h) = (l_{1},l_{2})$. Supposons que $l_{1} \geqslant l_{2}$. Comme $h \in \tilde{I}(t)$, on a $$e^{2 h} l_{1}^{2} + e^{-2 h} l_{2}^{2} < \frac{2 \cosh(1)}{ \epsilon \log (t)} \textit{.}$$
D'où, on obtient : $$l_{1}^{2} < \frac{2 \cosh(1) e^{-2 h}}{\epsilon \log(t)} $$ et $$ l_{2}^{2} < \frac{ \cosh(1)}{\epsilon \log(t) \cosh(2 h)} \textit{.}$$
Mais, d'après ($\ref{eq147}$), $l_{1}^{2} + l_{2}^{2} \geqslant \epsilon$, ce qui est donc exclu pour $t$ assez grand. \\
Ainsi, on a $l_{1} \leqslant l_{2}$. On obtient, comme $h \in \tilde{I}(t)$, 
\begin{equation}
\label{eq202}
|l_{1}| \leqslant \frac{1}{\sqrt{ \epsilon \log(t)}} \textit{ et }
\end{equation}

\begin{equation}
\label{eq203}
 |l_{2}| \geqslant \sqrt{\frac{\epsilon}{2}} \textit{.}
\end{equation}
À partir de l'expression de $h$ à partir des coordonnées de $l$ donné par le lemme $\ref{lemme33}$, on note au passage que, pour $t$ assez grand, 
$$ h \geqslant H \log(\log(t)) $$
où $H$ est une constante $>0$. \\
Grâce à $(\ref{eq202})$ et $(\ref{eq203})$, on a aussi : $$ 0 \leqslant \lambda(l) \leqslant \frac{1}{\epsilon^{2} \log(t)} \textit{.}$$ 
De même, dans le cas où $h \leqslant 0$, on obtient que : $l_{1} \geqslant l_{2}$, $$|l_{2}| \leqslant \frac{1}{\sqrt{\epsilon \log(t)}} \textit{,}$$ $$l_{1} \geqslant \sqrt{ \frac{\epsilon}{2}} \textit{,} $$ $$h \leqslant - H \log(\log(t)) \textit{ et}$$ $$|\lambda(l) - 1 | \leqslant \frac{1}{\epsilon^{2}  \log(t)} \textit{.} $$
Au vu des remarques précédentes, on voit qu'il existe une constante $F > 0$ telle que : 
\begin{equation}
\label{eq132}
|(\tilde{S}_{8}-\tilde{S}_{7})(\tau,t,P^{0},X,L)| \leqslant \frac{F}{\epsilon^{2} \log(t)^{2}} \sum_{h \in \tilde{I}(t)} \frac{1}{|| \delta_{h} L||^{2}} 
\end{equation}
et, grâce au lemme 3.2 de $\cite{Skriganov}$, cette dernière quantité, pour presque tout $L \in \mathscr{S}_{2}$, est un $O(\frac{1}{\log(t)^{\frac{1}{2}}})$ (ce lemme de $\cite{Skriganov}$ découle du théorème $\ref{thm20}$ et du lemme $\ref{lemme23}$). D'où le résultat voulu \footnote{ On voit qu'ici aussi on majore la valeur absolue du $\cos$ par $1$. Cette preuve est donc aussi valable lorsque l'on remplace $t^{+}$ par $t^{-}$.}.
\end{proof}

\subsubsection{Preuve de la proposition $\ref{prop42}$ - Remplacement de $\omega_{2}$ par $\omega_{2}(0)$ et conclusion}

Comme $\omega_{2}$ est la transformée de Fourier de $\frac{\omega_{1}}{\int \omega_{1} }$, $\omega_{2} (0) = 1$ et donc on a : 
\begin{align}
 \tilde{S}_{9}(\tau,t,\epsilon,X,L) =  \frac{2}{(2 \pi i)^{2}} \sum_{h \in \tilde{I}_{-} (t)}  \frac{1}{\Xi_{h,t}}  & \\
  (\sum_{k \in \mathbb{N}-\{ 0 \}} \frac{\omega_{2}(0)}{k^{2}} \cos(2 k \pi <l(L,h), t^{+} A_{2} + X > )) \nonumber 
\end{align}
d'après l'équation ($\ref{eq161}$).\\
Le but de cette sous-sous-section est de prouver la proposition suivante qui permet de remplacer le terme en $\omega_{2}$ par sa seule valeur en $0$ : 
\begin{prop}
\label{prop24}
$\tilde{S}_{8}-\tilde{S}_{9}$ converge en probabilités vers 0.
\end{prop}
\begin{proof}
Notons que l'on a, pour $k \in \mathbb{N}-\{0 \}$, pour $h \in \tilde{I}(t)$ : 
$$\lVert k \tau l(L,h) \rVert \leqslant k \tau e^{\log(t)} \lVert \delta_{h} L \rVert \leqslant k \tau e^{\log(t)} \sqrt{\frac{2\cosh(1)}{\epsilon \log(t)}} \textit{.}$$
Comme $\tau = \frac{\log(t)^{\frac{1}{4}}}{t}$, on a avec $l=l(L,h)$ : 
$$\lVert k \tau l \rVert \leqslant k \frac{1}{\log(t)^{\frac{1}{4}}} \sqrt{\frac{2\cosh(1)}{\epsilon }} \textit{.}$$
D'où, par le théorème des accroissements finis, et comme $\omega$ est une fonction à décroissance rapide : 
\begin{equation}
\label{eq135}
|\frac{\omega(k \tau l ) \cos (2 \pi k < l, t^{+}  A_{2} + X >)}{k^{2}}-\frac{\omega(0 ) \cos (2 \pi k < l, t^{+}  A_{2} + X >)}{k^{2}} | \leqslant \min(\frac{F}{k^{2}}, \frac{F}{k \log(t)^{\frac{1}{4}}})
\end{equation}
où $F$ est une constante $> 0$ qui dépend de $\epsilon$.
À partir de $(\ref{eq135})$, on trouve que : 
\begin{align}
|\sum_{k \in \mathbb{N}-\{ 0 \}} \frac{\omega_{2}(k \tau l)}{k^{2}} \cos(2 k \pi <l(L,h), t^{+}  A_{2} + X > -\sum_{k \in \mathbb{N}-\{ 0 \}} \frac{\omega_{2}(0)}{k^{2}} \cos(2 k \pi <l(L,h), t^{+}  A_{2} + X >| & \nonumber \\
 \leqslant \sum_{k=1}^{\infty} \min(\frac{F}{k^{2}}, \frac{F}{k \log(t)^{\frac{1}{4}}}) = O(\frac{1}{\log(t)^{\frac{1}{8}}}) 
\end{align}
où le $O$ dépend seulement de $\epsilon$. 
Ainsi, on trouve que : 
\begin{equation}
\label{eq136}
|\tilde{S}_{8}-\tilde{S}_{9}| \leqslant F \frac{\sum_{h \in \tilde{I}_{-} (t)} ||\delta_{h}L||^{-2}}{\log(t)^{1+ \frac{1}{8}}}
\end{equation}
où $F >0 $ est une constante dépendant de $\epsilon$.  \\
Encore grâce au lemme 3.2 de $\cite{Skriganov}$, on trouve que ce terme est un $O(\frac{1}{\log(t)^{\frac{1}{16}}})$. D'où le résultat voulu \footnote{Là encore on a majoré la valeur absolue du $\cos$ par $1$. La preuve est donc encore valable si l'on remplace $t^{+}$ par $t^{-}$.}.
\end{proof}
\begin{proof}[Démonstration de la proposition $\ref{prop42}$]
C'est une conséquence des propositions $\ref{prop18}$, $\ref{prop19}$, $\ref{prop20}$, $\ref{prop21}$, $\ref{prop22}$, $\ref{prop23}$, $\ref{prop24}$ \footnote{Toutes les propositions en question étant encore valables si l'on remplace $t^{+}$ par $t^{-}$, la réduction de $\tilde{S}_{2}$ à $\tilde{S}_{9}$ est encore valable avec $t^{-}$ à la place de $t^{+}$.}.
\end{proof}
\subsection{Transfert du problème à la vérification des hypothèses du théorème $\ref{thm13}$}
Maintenant, avec $\tilde{S}_{9}$, on va conclure comme dans la section précédente, c'est-à-dire montrer la proposition suivante : 
\begin{prop}
\label{prop25}
On a : 
\begin{itemize}
\item 1) $\{ \Xi_{h,t} \}_{h \in \tilde{I}_{-} (t)}$ converge vers un processus de Poisson sur $[-\frac{1}{\epsilon}, \frac{1}{\epsilon}]-\{0 \}$ d'intensité constante (indépendante de $\epsilon$).\\
\item 2) Les variables aléatoires réelles $(\Gamma_{h,t})_{h \in \tilde{I}_{-} (t)}$ convergent asymptotiquement vers des variables aléatoires réelles symétriques, indépendantes, identiquement distribuées, à support compact, indépendantes de $\{ \Xi_{h,t} \}_{h \in \tilde{I}_{-} (t)}$.
\end{itemize}
\end{prop}
Concernant l'item 2), comme $t^{+}= t + b \tau$, comme $\tau= \frac{\log(t)^{\frac{1}{4}}}{t}$, il suffit de montrer ce résultat en remplaçant $t^{+}$ par $t$ (quitte à adapter le raisonnement menant à l'équation ($\ref{eq136}$) \footnote{Notons que ce raisonnement ne fait que intervenir $|t-t^{+}|$. Or $|t-t^{+}|=|t-t^{-}|$. Donc la proposition $\ref{prop25}$ sera aussi valable en remplaçant $t^{+}$ par $t^{-}$ et donc on obtiendra au final que $\tilde{S}_{2}^{+}$ et $\tilde{S}_{2}^{-}$ convergent vers une même loi de Cauchy centrée.}). Par ailleurs, on peut aussi remplacer $\log(t)$ par $\lceil \log(t) \rceil$. \\ 
\\
Posons quelques nouvelles notations : \begin{itemize} 
\item $M = \lceil \log(t) \rceil$
\item  $\nu_{0}^{M} (L) = ( M \text{Num}(e( \cdot)) \mathbf{1}_{2 \cosh(1) |\text{Num}(e(\cdot))| > || \cdot ||^{2}} ) (L) $
\item $\nu_{h}^{M}(L) = \nu_{0}^{M}(\delta_{h} L) \textit{, }$
\item $\xi_{h}^{M}= \mathbf{1}_{\nu_{h}^{M} \in [- \frac{1}{\epsilon}, \frac{1}{\epsilon}]-\{0\}} $
\item $\zeta_{h}^{M}(L,X) = < e(\delta_{h}L), \delta_{h}^{-1} (e^{M} A_{2} + X) > (\text{mod } 1) $
\end{itemize}
Si l'on démontre la proposition suivante, la proposition $\ref{prop25}$ est acquise : 
\begin{prop}
\label{prop30}
On a : 
\begin{itemize}
\item 1) $(\nu_{h}^{M})_{\xi_{h}^{M} = 1 \textit{, } h \in [-M,0]} $ converge vers un processus de Poisson sur $[-\frac{1}{\epsilon}, \frac{1}{\epsilon}]-\{0 \}$ d'intensité constante (indépendante de $\epsilon$). \\
\item 2) $(\zeta_{h}^{M})_{\xi_{h}^{M} = 1 \textit{, } h \in [-M,0]}$ convergent asymptotiquement vers des variables aléatoires à valeurs dans $\mathbb{R}/\mathbb{Z}$, indépendantes, identiquement distribuées selon la mesure de Haar normalisée sur $\mathbb{R}/\mathbb{Z}$, indépendantes de $(\nu_{h}^{M})_{\xi_{h}^{M} = 1 \textit{, } h \in [-M,0]}$.
\end{itemize}
\end{prop}
\begin{proof}[Démonstration du fait qu'il suffit de démontrer la proposition $\ref{prop30}$]
Il suffit de remarquer que : \begin{itemize}
\item $h \in  [-M,0]$ signifie exactement que $h \in \tilde{I}_{-} (t)$. \\
\item $(\xi_{h}^{M})(L) = 1$ signifie exactement que $|\text{Num}(l(L,h)) \lceil \log(t) \rceil | \leqslant \frac{1}{\epsilon}$, et, dans ce cas, $(\nu_{h}^{M})(L) = \text{Num}(l(L,h)) \lceil \log(t) \rceil$  \\
\item $(\zeta_{h}^{M})(L,X) = (<l(L,k), e^{M}  A_{2} + X > )$
\end{itemize}
 
\end{proof}
Par ailleurs, le terme dominant dans $< e(\delta_{h}L), \delta_{h}^{-1} (e^{M} A_{2} + X) >$ est $< e(\delta_{h}L), \delta_{h}^{-1} e^{M}A_{2} >$. On peut donc supposer, et on le fera par la suite, que : 
 $$\zeta_{h}^{M}(L) = < e(\delta_{h}L), \delta_{h}^{-1} e^{M}  A_{2} > (\text{mod } 1) $$
 (en particulier on oublie la dépendance en $X$). \\
 Concentrons-nous maintenant sur la preuve de la proposition $\ref{prop30}$.
\subsection{Vérification des hypothèses du théorème $\ref{thm13}$} 
Le théorème $\ref{thm13}$, la remarque qui le suit et la proposition (de cette sous-section) qui suit permettent de montrer la proposition $\ref{prop30}$. Pour énoncer cette proposition, gardons en tête les notations de la sous-section précédente et posons :\begin{itemize} 
\item $\lambda_{1}(h) = - 2 h$
\item $\mathbf{X}=[-\frac{1}{\epsilon}, \frac{1}{\epsilon}]-\{0 \}$ 
\item $\mathbf{Q}=(\mathbf{Q}_{P})_{P \in \mathbb{N}-\{ 0 \}}$ où $\forall P \in \mathbb{N}-\{0 \}$, $\mathbf{Q}_{P}=(\mathbf{X}_{k,P})_{k \in \{0,\cdots,P-1\}}$ avec \\
 \\
 pour tout $k \in [0 ; P-1]$, $\mathbf{X}_{k,P} = [-\frac{k+1}{P}\frac{1}{\epsilon},-\frac{k}{P}\frac{1}{\epsilon}[ \cup ]\frac{k}{P}\frac{1}{\epsilon},\frac{k+1}{P}\frac{1}{\epsilon}]$
\item $\tilde{\mathbf{X}} = \mathbb{R}/\mathbb{Z}$ et $\tilde{\mathbf{Q}}=((\tilde{\mathbf{X}}_{k,P})_{k \in \{0,\cdots,P-1\}})_{P \in \mathbb{N}-\{0 \}}$ une collection de partitions d'intervalles de longueur $\frac{1}{P}$ avec $P \in \mathbb{N}-\{ 0 \}$
\item $ \hat{\lambda}(h) = M - h$ si bien que l'on a $\hat{\lambda} > \lambda_{1}$ sur $\text{Int}(\Pi_{M})$
\end{itemize}
On remarque alors que : 
 \begin{prop}
 \label{prop31}
$(\lambda_{1},\hat{\lambda},\{\xi_{t} \},\{\nu_{t} \},\{\zeta_{t}\})$ satisfont les hypothèses (h1) jusqu'à (h8) du théorème $\ref{thm13}$. \\
La proposition $\ref{prop30}$ est en conséquence valide.
 \end{prop}
 \begin{proof}[Démonstration de la deuxième partie de la proposition $\ref{prop31}$]
 Cela repose sur le lemme $\ref{lemme20}$ et le théorème $\ref{thm13}$.
 \end{proof}
On va donc maintenant montrer que (h1) jusqu'à (h8) sont ici vérifiées.
\subsubsection{Vérification des hypothèses (h1) jusqu'à (h8) et conclusion}
Le but de cette section est de vérifier la proposition suivante qui est la première partie de la proposition $\ref{prop31}$ : 
\begin{prop}
\label{prop26}
$(\lambda_{1},\hat{\lambda},\{\xi_{t} \},\{\nu_{t} \},\{\zeta_{t}\})$ satisfont les hypothèses (h1) jusqu'à (h8) du théorème $\ref{thm13}$.
\end{prop}
On a besoin de quelques lemmes préliminaires, dont on trouvera les équivalents dans la section précédente et les démonstrations sont peu ou prou les mêmes (le théorème de Green-Riemann peut être utile pour le deuxième lemme)
\begin{lemma}
\label{lemme38}
$\xi_{0,P}  $, $\xi_{0} $ sont dans $H^{s,s}$ et 
$$ \lVert \xi_{0,P} , \xi_{0} \rVert_{H^{s,s}} = O(1) \textit{.}$$
\end{lemma}
Posons  $\hat{\mathbf{X}} = \{ x \in ]0, \infty [ \text{ } | \text{ } d(x,\partial \mathbf{X}) \leqslant M^{-1000} \}$ et $\hat{\mathbf{X}_{i}}= \{ x \in ]0, \infty [ \text{ } | \text{ } d(x,\partial \mathbf{X}_{i}) \leqslant M^{-1000} \}$ (l'indice $P$ étant sous-entendu).\\
On pose aussi $\hat{\xi_{m}}= \mathbf{1}_{ \nu_{m} \in \hat{\mathbf{X}}}$ et $\hat{\xi_{m,i}} = \mathbf{1}_{ \nu_{m} \in \hat{\mathbf{X}_{i}}} $ (l'exposant $M$ étant sous-entendu). 
\begin{lemma}
\label{lemme39}
\begin{equation}
\label{eq132}
\tilde{\mu}_{\mathscr{S}_{2}}(\hat{\xi_{0,i}}) = O(T^{-1000})
\end{equation}
\begin{equation}
\tilde{\mu}_{\mathscr{S}_{2}}(\hat{\xi_{0}}) = O(T^{-1000}) 
\end{equation}
\end{lemma}
Avant de passer à la démonstration de la proposition $\ref{prop26}$, disons-en quelques mots. (h1) jusqu'à (h6) sont vérifiées de manière tout à fait analogue à ce qui a été fait dans la section précédente. Il en est de même pour (h8). Seule la vérification de (h7) se détache et l'avant-dernier paragraphe lui est dédiée.
\begin{proof}[Démonstration de la proposition $\ref{prop26}$]
Supposons pour simplifier que $\mu_{2}= \tilde{\mu}_{2}$, le cas général se traitant de manière analogue à l'aide de la propriété de mélange exponentielle et à l'aide des résultats du cas simplifié que l'on traite maintenant.\\
$\bullet $ Vérification de (h1)
Par le même raisonnement que dans la section précédente, en utilisant le lemme $\ref{lemme24}$, on trouve, en posant $ D=\int_{u > 0 \textit{, } v > 0 } \mathbf{1}_{u v \leqslant 1} \mathbf{1}_{ 2 \cosh(1) u v > \lVert (u,v) \rVert^{2}}$, que
\begin{equation}
\label{eq131}
E(\xi_{k}^{M}) = \frac{2 D }{\epsilon M} \textit{.}
\end{equation}
$\bullet $ Vérification de (h2) \\
Vu la définition de $\xi_{k}^{M}$, il est clair que (h2) est vérifiée.\\
$\bullet $ Vérification de (h3) \\
Par ailleurs, on a vu que si on disposait auparavant d'un $k$ tel que $\xi_{k}^{M} = 1$ alors nécessairement ce $k$ était un minimum local de $p \in \mathbb{Z} \longmapsto \lVert \delta_{p}L \rVert^{2}$ et les $k$ tels que $\xi_{k}^{M} = 1$ sont des $k$ tels que $$ \lVert \delta_{p}L \rVert^{2} \leqslant \frac{2 \cosh(1) e}{\epsilon M } \textit{.}$$
Ainsi, par le même raisonnement qui mène à l'inéquation ($\ref{eq120}$), on voit que $(h3)$ est vérifiée : les $k$ qui conviennent sont suffisamment distants.\\
$\bullet $ Vérification de (h4a), (h4b) et (h4c) \\
Soit $ p \in \mathbb{N}-\{  0 \}$.  Ici aussi, on va avoir $\eta_{k,p}= \xi_{k,p}$. Dès lors, il suffit juste de vérifier (h4a) et (h4b), (h4c) étant alors vérifiée automatiquement. \\
En suivant toujours le même raisonnement que dans la section précédente, et en utilisant le $W_{1}^{-}$-feuilletage (qui correspond aux matrices $h_{1}(\tau)$ transposés, tandis que le $W_{1}^{+}$-feuilletage correspond aux matrices $h_{1}(\tau)$), on voit qu'il existe $\tilde{E}_{1}$ et $F_{t}$, pour tout $t \in \Pi$, tels que (h4) soit vérifiée grâce au lemme $\ref{lemme38}$. \\
$\bullet $ Vérification de (h5) \\
On suit là encore le même raisonnement que dans la section précédente. Ici aussi il faut remplacer l'ensemble sur lequel on travaille, $ \tilde{I}_{-} (t)$, par l'ensemble $\tilde{\tilde{I}}_{-} (t)$ où : 
\begin{align*} \tilde{\tilde{I}}_{-} (t) = \{ h \in [- \lceil  \log(t) \rceil, 0   ]  \text{ } | \text{ } \lVert \delta_{h} L \rVert^{2} + M^{-2001} \leqslant 2 \cosh(1)|\text{Num}(e(\delta_{h}L))|  & \text{ et } \\
 |\text{Num}(e(\delta_{h}L))| \leqslant \frac{1}{\epsilon_{4} \log(t)} - M^{-2001} \} \textit{,}
\end{align*}
Faire cela ne coûte qu'un $O(M^{-1000})$. \\
Pour mener à son terme le même raisonnement, il faut utiliser l'égalité exhibée dans la section précédente, à savoir $$M \lVert h(\tau) \delta_{k} u \rVert^{-2} = M \lVert \delta_{k} u \rVert^{-2} + O(M^{-10^{9}+1})$$ (valable dès lors que le feuilletage utilisé soit suffisamment fin), utiliser le fait que $$|\text{Num}(e(h(\tau) \delta_{k}L))| = |\text{Num}(e(\delta_{k}L))|(1+ O(\tau))$$ quand $ \tau \rightarrow 0$ et utiliser le lemme $\ref{lemme39}$. \\
$\bullet $ Vérification de (h6) \\
La vérification de (h6) se fait de manière tout à fait analogue à ce qui a été fait dans la section précédente. 
\\
$\bullet $ Vérification de (h7) \\
On dispose désormais d'un ensemble $E$ que l'on a construit pour lequel (h1) jusqu'à (h6) sont vérifiées. Soient $k, k'$ dans $[-M,0]$ tels que $\hat{\lambda}(k) \geqslant \lambda_{1}(k') + R \log(M) \geqslant \lambda_{1}(k) + 2 R \log(M) $. Soit $ L \in E$ tel que $\xi_{k}(L) = \xi_{k,p}(L) = 1$ où $p \in \{0 , \cdots, P-1 \}$. On veut maintenant montrer, dans l'optique de vérifier (h7), que : 
$$ \mu_{2}(< e(\delta_{k}L), \delta_{k}^{-1} (e^{M} P^{0}) > (\text{mod } 1) \in \tilde{\mathbf{X}}_{j,P} | \mathcal{F}_{k'})(L) =   \text{Leb}(\tilde{\mathbf{X}})_{j,P} ( 1 + o(1) ) \textit{.} $$
En posant $$ h_{\tau} =  \begin{pmatrix} 1 & \tau \\ 0 & 1 \end{pmatrix} \textit{,}$$ pour $\gamma_{k'} = F_{k'}(L)$, on a $$\gamma_{k'} = \{ h_{\tau}^{T} \overline{L} \}_{0 \leqslant \tau \leqslant (e^{\lambda_{1}(k')} M^{1000})^{-1} } $$ où $\overline{L} \in \mathscr{S}_{2}$. \\
 Quitte à augmenter légèrement $R$, on peut supposer, grâce à (h5), que $\xi_{k,p} = 1$ sur $\gamma$, et donc, en particulier, en posant $$e=(e_{1},e_{2}) = e(\delta_{k} \overline{L} ) \text{, }$$ $e$  vérifie  :  $$2 \cosh(1) | \text{Num} (e) | > \lVert e \rVert^{2} \textit{ et }M |\text{Num}(e)| \in \mathbf{X}_{p} \textit{.} $$ 
Or, on a $\delta_{k} h_{\tau}^{T} = h_{e^{\lambda_{1}(k)} \tau }^{T} \delta_{k}$ et, quitte à prendre $R$ assez grand, on peut supposer que $e^{\lambda_{1}(k)} \tau \leqslant 1$ et ainsi on a : 
$$e(\delta_{k}h_{\tau} \overline{L}) = h_{e^{\lambda_{1}(k)} \tau }^{T} e =  \begin{pmatrix} e_{1}  \\ \tau e^{\lambda_{1}(k)} e_{1} + e_{2}  \end{pmatrix} \textit{.}$$ \\
Ainsi, on obtient : 
$$ < e(\delta_{k} h_{\tau}^{T}L), \delta_{k}^{-1} e^{M} A_{2} >  = e_{1} e^{-k} e^{M} (A_{2})_{1}  + e_{2} e^{k} e^{M} (A_{2})_{2} +  e^{\lambda_{1}(k)} \tau e_{1} e^{k}e^{M} (A_{2})_{2}  ( \text{mod } 1) \textit{.} $$
Or, vu les conditions que vérifie $e$, on a $|e_{1}| > \beta$ où $\beta$ est une constante strictement positive. Par ailleurs, on a $e^{M} e^{k} e^{\lambda_{1}(k)} = e^{\hat{\lambda}(k)}$ (voir le début de la section 5.9) et $\tau$ varie dans un intervalle de longueur $(e^{\lambda_{1}(k')} M^{1000})^{-1}$. Ainsi, comme $\hat{\lambda}(k) \geqslant \lambda_{1}(k') + R \log(M)$, on obtient le résultat voulu comme $(A_{2})_{2} \neq 0$. \\
$\bullet $ Vérification de (h8) \\
On sait que : $F_{\overline{k}}(L) = \{  h_{\tau} \overline{L} \}_{0 \leqslant \tau \leqslant (e^{\lambda_{1}(\overline{k})} M^{1000})^{-1} }$. D'après (h5), $\xi_{k}$ vaut $1$ sur $F_{\overline{k}}(L)$. Le calcul fait précédemment donne que : 
 $$< e(\delta_{k} h_{\tau}^{T}L), \delta_{k}^{-1} e^{M}A_{2} >  = e_{1} e^{-k} e^{M} (A_{2})_{1} + e_{2} e^{k} e^{M} (A_{2})_{2} +  \tau e_{1} e^{\hat{\lambda}(k)} (A_{2})_{2} + ( \text{mod } 1) \textit{.} $$
 Ainsi, si $\mathbf{1}_{\zeta_{k} \in \tilde{K}_{p}}$ n'est pas constant sur $F_{\overline{k}}(L)$, $\zeta_{k}$ est dans un $O(M^{-1000})$-voisinage de $\partial \tilde{\mathbf{X}}_{p}$. Si on pose 
 \begin{align*} 
 \tilde{E} = \{ L \in E | \forall k \textit{, } \overline{k} \in [-M,0] \text{ avec } \lambda_{1}(\overline{k}) > \hat{\lambda}(k) + R \log(M) \text{ on a } & \\
   \forall p \in [0,P-1] \textit{, } \mathbf{1}_{ \zeta_{k} \in \tilde{\mathbf{X}}_{p}} \text{ est constante sur } F_{\overline{k}}(L) \}
 \end{align*} 
 alors (h8) est vérifiée sur $\tilde{E}$. Enfin le lemme $\ref{lemme24}$ permet de montrer que $\mu_{2}(E- \tilde{E}) = O(M^{-998})$.
\end{proof}
On peut maintenant prouver le théorème $\ref{thm23}$ : 
\begin{proof}[Démonstration du théorème $\ref{thm23}$]
La proposition $\ref{prop26}$ et la proposition $\ref{prop31}$ donnent alors que la proposition $\ref{prop30}$ est vérifiée. La proposition $\ref{prop25}$ est donc acquise grâce au lemme $\ref{lemme37}$. D'où le résultat voulu grâce aux lemmes $\ref{lemme20}$ et $\ref{lemme22}$ et grâce à la proposition $\ref{prop42}$.
\end{proof}
\subsection{Retour sur $\tilde{S}$ et preuve du théorème $\ref{thm100}$}
On va maintenant démontrer le théorème $\ref{thm100}$ dans le cas de la dimension $2$.
On rappelle que la somme que l'on doit étudier est la suivante : 
\begin{equation}
\label{eq143}
\tilde{S}(\tau,t,X,L) = \sum_{i=1}^{4} (-1)^{i} (\tilde{S}(\tau,t,A_{i},I_{2},X,L)-\tilde{S}(\tau,t,A_{i},R,X,L))  
\end{equation}
et les différentes sommes portent sur des termes $l \in L$ et non sur des termes $l \in L^{\perp}$. \\
Le cheminement pour prouver le théorème $\ref{thm100}$ dans ce cas est le suivant. Pour chaque terme en $ \tilde{S}^{\pm}$, on réduit la somme correspondante grâce à ce qui a été fait dans la section précédente, les résultats s'étendant naturellement. Ensuite on regroupe ces 8 nouvelles sommes par de petits calculs. Enfin on conclut en suivant la même démarche que précédemment (application du théorème $\ref{thm13}$). \\
Plus précisément, après des réductions et des petits calculs, on voit que l'étude asymptotique de $\tilde{S}$ se ramène à l'étude de : 

\begin{align}\Sigma(t,P,\epsilon,X,L) = \frac{2}{ \pi^{2} \log(t)} \sum_{h \in \tilde{I}(L,t,\epsilon) }  \frac{1}{\text{Num}(l(L,h))} \\
\sum_{m \in \mathbb{N}-\{0 \}} \frac{ \sin(2 m \pi (l(L,h))_{1} t^{+} a )\sin(2 m \pi (l(L,h))_{2} t^{+} a )\cos(2 m \pi <l(L,h),X> )}{m^{2}} \textit{.} \nonumber
\end{align}

Partant de là, la preuve du théorème $\ref{thm100}$ s'énonce de la manière suivante : 
On peut maintenant démontrer le théorème $\ref{thm100}$, dans le cas où $d=2$, à laquelle cette section est dédiée.
\begin{proof}[Démonstration du théorème $\ref{thm100}$]
En suivant la même méthode que dans la section précédente, via le théorème $\ref{thm13}$, on montre que : \begin{itemize}
\item $(\text{Num}(l(L,h)))_{h \in \tilde{I}(L,t,\epsilon)}$ converge vers un processus de Poisson sur $[-\frac{1}{\epsilon}, \frac{1}{\epsilon}]$ d'intensité constante et strictement positive, indépendante de $ \epsilon$
\item La suite de variables aléatoires $((l(L,h))_{1} t a)_{h \in \tilde{I}(L,t,\epsilon)}$ à valeurs dans $\mathbb{R} / \mathbb{Z}$ convergent vers des variables aléatoires indépendantes de $(\text{Num}(l(L,h)))_{h \in \tilde{I}(L,t,\epsilon)}$, indépendantes entre elles et de même loi
\item La suite de variables aléatoires $((l(L,h))_{2} t a)_{h \in \tilde{I}(L,t,\epsilon)}$ à valeurs dans $\mathbb{R}/ \mathbb{Z}$ convergent vers des variables aléatoires indépendantes de $(\text{Num}(l(L,h)))_{h \in \tilde{I}(L,t,\epsilon)}$, indépendantes entre elles et de même loi.
\end{itemize}
Enfin, on montre que à l'infini et modulo $1$, la suite de variables aléatoires $((l(L,h))_{1} t a)_{h \in \tilde{I}(L,t,\epsilon)}$ est indépendante de $((l(L,h))_{2} t a)_{h \in \tilde{I}(L,t,\epsilon)}$. \\
\\
Grâce à cela, on obtient que : $$(\sum_{m \in \mathbb{N}-\{0 \}} \frac{ \sin(2 m \pi (l(L,h))_{1} t^{+} a )\sin(2 m \pi (l(L,h))_{2} t^{+} a )\cos(2 m \pi <l(L,h),X> )}{m^{2}})_{h \in \tilde{I}(L,t,\epsilon)}$$ convergent vers des variables aléatoires réelles indépendantes entre elles, ayant une loi commune qui est symétrique et dont le support est compact, et qui sont indépendantes asymptotiquement de $(\text{Num}(l(L,h)))_{h \in \tilde{I}(L,t,\epsilon)}$. Le lemme $\ref{lemme42}$ s'applique. Comme ce raisonnement fonctionne aussi en remplaçant $t^{+}$ par $t^{-}$, on obtient le résultat voulu.
\end{proof}
Ci-après, on expose d'abord la réduction que nous permet de faire la section précédente sur les $8$ termes de la somme $\tilde{S}$ puis ensuite quelques petits calculs qui nous permettent de nous ramener à la somme $\Sigma$ sous la forme ici exposée \footnote{Ces deux parties fonctionnent indépendamment de si l'on considère $t^{+}$ ou $t^{-}$.}.
\subsubsection{Réduction des 8 termes $\tilde{S}_{1}$ de $\tilde{S}$}
En suivant ce qui a été fait tout le long de la section précédente et en posant (pour $A$ un point de $\mathbb{R}^{2}$ et $R_{1} \in SO_{2}(\mathbb{R})$) :
 \begin{align}
\label{eq166}
I(R_{1},L,t,\epsilon) = \{ h \in [- \lceil  \log(t) \rceil, 0    ]  \text{ } | \text{ } \lVert \delta_{h} R_{1} L \rVert^{2} < 2 \cosh(1)|\text{Num}(e(\delta_{h} R_{1}L))|  \\ \text{ et } 
 |\text{Num}(e(\delta_{h} R_{1}L))| \leqslant \frac{1}{\epsilon \log(t)} \} \nonumber
\end{align}
\begin{equation}
\label{eq165}
G_{1}(\tau,t,A,R_{1},X,l)=\sum_{m \in \mathbb{N}-\{0 \}} \frac{\cos(2 m \pi <l, t^{+} R_{1}A + R_{1}X >)}{m^{2}}
\end{equation}
\begin{equation}
\label{eq210}
\Sigma_{1}(\tau, t,A,R_{1},X,L) = \frac{2}{\log(t)} \frac{1}{(2 \pi i)^{2}} \sum_{h \in I(R_{1},L,t,\epsilon) }  \frac{1}{\text{Num}(l(R_{1}L,h))} G_{1}(\tau,t,A,R_{1},X,l(R_{1}L,h)) 
\end{equation}
\begin{equation}
\label{eq163}
\Sigma(\tau,t,X,L) = \sum_{i=1}^{4} (-1)^{i} (\tilde{S}_{r}(\tau,t,A_{i},I_{2},X,L)-\tilde{S}_{r}(\tau,t,A_{i},R,X,L))  
\end{equation}
on obtient la proposition suivante :
\begin{prop}
\label{prop27}
Pour tout $\alpha > 0$, pour tout $\epsilon > 0$ assez petit, pour tout $t$ assez grand, on a 
$$\mathbf{P}(|\tilde{S}_{1}-\Sigma| \leqslant \alpha) \leqslant \alpha \textit{.}$$ 
\end{prop}
On est donc ramené à l'étude de $\Sigma$ \footnote{Vu ce que l'on a noté tout au long du passage de $\tilde{S}_{2}$ à $\tilde{S}_{9}$, cette réduction s'applique aussi si l'on remplace $t^{+}$ par $t^{-}$.}.  
\subsubsection{Un dernier regroupement de termes} 
Le but de cette sous-sous-section est de voir que $\Sigma(\tau,t,X,L)$ s'écrit aussi :  
\begin{prop}
\label{prop32}
\begin{equation}
\label{eq204}
\Sigma(t,P,\epsilon,X,L) = \frac{-1}{2 \pi^{2} \log(t)} \sum_{h \in \tilde{I}(L,t,\epsilon) }  \frac{ 1}{\text{Num}(l(L,h))} \sum_{m \in \mathbb{N}-\{0 \}} \sum_{i=1}^{4} (-1)^{i} \frac{\cos(2 m \pi <l(L,h), t^{+} A_{i} + X >)}{m^{2}}
\end{equation}
où
\begin{multline}
\label{eq205} \tilde{I}(L,t,\epsilon) =  \{ h \in [- \lceil  \log(t) \rceil, \lceil \log(t) \rceil  ]  \text{ } | \text{ } \lVert \delta_{h}  L  \rVert^{2} < 2 \cosh(1)|\text{Num}(e(\delta_{h}L))|  \\ \text{ et } 
 |\text{Num}(e(\delta_{h} L))| \leqslant \frac{1}{\epsilon \log(t)} \} \text{. }
\end{multline}
\end{prop}
La clé de cette proposition est que quand on a un $h$ qui est dans $I(R,L,t,\epsilon)$, $-h$ appartient à $I(I_{2},L,t,\epsilon)$ et réciproquement. C'est juste la traduction du fait que si un vecteur $l \in L$ est grand en norme et « proche » de l'axe des abscisses, il est « loin » de l'axe des ordonnées et vice-versa \footnote{Cette proposition est indépendante du fait que $t=t^{+}$. Elle est donc encore valable en remplaçant $t^{-}$ par $t^{+}$.}. 
\begin{proof}
Rappelons que $R = \begin{pmatrix} 0 & 1 \\ -1 & 0 \end{pmatrix}$. \\
Pour $t$ assez grand, si on a $h$ tel que $h \in I(I_{2},L,t,\epsilon)$ alors $$e(\delta_{h} L ) = e(R \delta_{-h} R  L) = \pm R e( \delta_{-h} R L)$$ et donc $-h \in I(R,L,t,\epsilon)$ et $\text{Num}(e(\delta_{h} L )) = - \text{Num}(e( \delta_{-h} R L))$.\\
En inversant le rôle de $R$ et de $I_{2}$, on a de même. \\
En utilisant les équations $(\ref{eq166})$, $(\ref{eq165})$ et $(\ref{eq163})$ et les remarques précédentes dans cette preuve, on obtient le résultat voulu en regroupant les termes correspondants aux différents drapeaux.
\end{proof}
\subsubsection{Une dernière réécriture}
Pour arriver à la formule finale de $\Sigma$ donné en début de cette sous-section, nous avons besoin de réaliser une dernière réécriture des termes de la $\tilde{S}_{r,2}(t,P,\epsilon,X,L)$. \\
Rappelons que $A_{1}$ est de coordonnées $(a,-a)$, $A_{2}$ de coordonnées $(a,a)$, $A_{3}$ de coordonnées $(-a,a)$ et $A_{4}$ de coordonnées $(-a,-a)$. Une utilisation des formules trigonométriques classiques donne alors : 
\begin{prop}
\label{prop33}
\begin{align}\tilde{S}_{r,2}(t,P,\epsilon,X,L) = \frac{2}{ \pi^{2} \log(t)} \sum_{h \in \tilde{I}(L,t,\epsilon) }  \frac{1}{\text{Num}(l(L,h))} \\
\sum_{m \in \mathbb{N}-\{0 \}} \frac{ \sin(2 m \pi (l(L,h))_{1} t^{+} a )\sin(2 m \pi (l(L,h))_{2} t^{+} a )\cos(2 m \pi <l(L,h),X> )}{m^{2}} \textit{.} \nonumber
\end{align}
\end{prop}
Cette proposition est encore valable si l'on remplace $t^{+}$ par $t^{-}$.
\section{Conclusion}
Nous avons donc démontré le théorème $\ref{thm100}$ qui s'énonce de la manière suivante : 
\begin{theorems}
Lorsque $L \in \mathscr{S}_{2}$ est distribué selon la mesure $\tilde{\mu}_{2}$ et $X$ est distribué selon $\tilde{\lambda}_{2}$ alors on a : 
$$\frac{\mathcal{R}(t,P,X,L)}{\log(t)} \overset{\mathcal{L}}{\underset{t \rightarrow \infty}{\rightarrow}} \mathcal{C}_{c} $$ 
où $\overset{\mathcal{L}}{\rightarrow}$ signifie que la convergence a lieu en loi et où $\mathcal{C}_{c}$ désigne une loi de Cauchy centrée. 
\end{theorems}
Ce que ce théorème donne en substance c'est que l'erreur commise dans l'estimation de $N(tP + X,L)$ par la quantité $\text{Aire}(t P)$ est, en moyenne spatiale (c'est-à-dire en faisant varier légèrement le réseau), exactement de l'ordre de $\log(t)$ quand $t$ devient grand. \\
Dans un prochain article, nous nous intéresserons à un même type de problème mais en remplaçant le parallélogramme $P$ par un ellipsoïde plein de dimension $d$, $\mathcal{E}^{d}$, avec $d \geqslant 2$, et montrerons le résultat suivant : 
\begin{theorem}
Lorsque $L \in \mathscr{S}_{d}$ est distribué selon une mesure $\tilde{\mu}_{d}$, de densité régulière et bornée par rapport à la mesure de Haar normalisée sur $\mathscr{S}_{d}$, on a : 
$$\frac{\mathcal{R}(t,\mathcal{E}^{d},0,L)}{t^{\frac{d-1}{2}}} $$
converge en loi quand $t \rightarrow \infty$. Par ailleurs, la loi limite admet un moment d'ordre $1$, est d'espérance nulle et lorsque que l'on considère $\tilde{\mu}_{d} = \mu_{d}$, la loi limite n'admet pas de mot d'ordre $2$. 
\end{theorem} 

\bibliographystyle{plain}
\bibliography{bibliographie}
\end{document}